\documentclass[11pt]{amsart}

\usepackage[paperheight=11in, 
    paperwidth=8.5in,
    outer=1in,
    inner=1in,
    bottom=1in,
    top=1in]{geometry}

\usepackage{amsmath, amscd, amsthm, amsfonts, amssymb, mathrsfs}
\usepackage{graphicx}
\usepackage{url}
\usepackage{xcolor}
\usepackage{caption}
\usepackage{subcaption}
\usepackage{bbm}
\usepackage{tikz-cd}
\usepackage{enumitem}

\usepackage[T1]{fontenc}
 \usepackage[hidelinks]{hyperref} 
 \hypersetup{
    colorlinks,
    citecolor=magenta,
    filecolor=magenta,
    linkcolor=blue,
    urlcolor=black
}
\usepackage{cleveref}

\usepackage[backend=bibtex, style=alphabetic, maxnames=99]{biblatex}
\bibliography{BEZ_TropFlagVar.bib}
\AtBeginBibliography{\footnotesize}

\newcommand{\RR}{\mathbb R}

\newcommand{\PP}{\mathbb P}
\newcommand{\ZZ}{\mathbb Z}
\newcommand{\TT}{\mathbb T}

\newcommand{\trop}{\operatorname{trop}}

\newcommand{\bb}{\mathcal B}
\newcommand{\cc}{\mathcal C}

\newcommand{\be}{\mathbf e}
\newcommand{\MM}{{\boldsymbol M}}
\newcommand{\UU}{{\boldsymbol U}}
\newcommand{\kk}{\mathbbm k}
\newcommand{\bmu}{\boldsymbol \mu}
\newcommand{\bv}{{\mathbf v}}
\newcommand{\bu}{{\mathbf u}}

\theoremstyle{definition}
\newtheorem{thm}{Theorem}[subsection]
\newtheorem{cor}[thm]{Corollary}
\newtheorem{lem}[thm]{Lemma}
\newtheorem{prop}[thm]{Proposition}
\newtheorem{defn}[thm]{Definition}
\newtheorem{eg}[thm]{Example}
\newtheorem{rem}[thm]{Remark}
\newtheorem{notation}[thm]{Notation}

\newtheorem{maintheorem}{Theorem}	
\title{Tropical Flag Varieties}
\author{Madeline Brandt, Christopher Eur, Leon Zhang}

\address{Department of Mathematics, Brown University,  Providence, RI 02912}
\email{madeline\_brandt@brown.edu}

\address{Department of Mathematics, Stanford University.  Stanford, CA. USA}
\email{chriseur@stanford.edu}

\address{University of California, Berkeley. Berkeley, CA. USA}
\email{leonyz@berkeley.edu}

\begin{document}

\maketitle

\begin{abstract}
Flag matroids are combinatorial abstractions of flags of linear subspaces, just as matroids are of linear subspaces.  We introduce the flag Dressian as a tropical analogue of the partial flag variety, and prove a correspondence between: (a) points on the flag Dressian, (b) valuated flag matroids, (c) flags of projective tropical linear spaces, and (d) coherent flag matroidal subdivisions.
We introduce and characterize projective tropical linear spaces, which serve as a fundamental tool in our proof.  We apply the correspondence to prove that all valuated flag matroids on ground set up to size 5 are realizable, and give an example where this fails for a flag matroid on 6 elements.
\end{abstract}

\vspace{-2mm}
\section{Introduction}
\label{introduction}

The Grassmannian $Gr(r;n)$ over a field $\kk$ parameterizes $r$-dimensional linear subspaces in $\kk^{[n]}$, or equivalently, realizations of matroids of rank $r$ on the ground set $[n] = \{1,\ldots, n\}$.
It can be embedded in $\mathbb{P}\big(\kk^{\binom{[n]}{r}}\big)$, where it is cut out by the quadratic Grassmann-Pl\"{u}cker relations.
For a fixed matroid $M$, one can modify the Grassmann-Pl\"{u}cker relations to cut out only the points in the Grassmannian realizing $M$.
The tropical prevariety of these equations is the \textbf{Dressian of $M$}, denoted $Dr(M)$, which was introduced in \cite{tropplanes}.  The Dressian of a loopless matroid $M$ has multiple interpretations as
\begin{enumerate}[label = (\alph*)]
\item the tropical prevariety of (modified) Grassmann-Pl\"ucker relations,
\item the set of all valuated matroids with underlying matroid $M$ \cite{dress},
\item the weight vectors inducing a matroidal subdivision of the base polytope of $M$ \cite{speyer}, or
\item the parameter space of all tropical linear spaces given by $M$ \cite{speyer}.
\end{enumerate}

The (partial) flag variety $Fl(r_1,\ldots,r_k;n)$ parameterizes flags of linear spaces $L_1 \subseteq \cdots \subseteq L_s$ in $\kk^{[n]}$ where $\dim_\kk(L_i) = r_i$.
A point on $Fl(r_1, \ldots, r_k;n)$ corresponds to a realization of a \textbf{flag matroid}, which is a sequence of matroids $\MM = (M_1, \ldots, M_k)$ of ranks $(r_1, \ldots, r_k)$ on $[n]$ such that every circuit of $M_{j}$ is a union of circuits of $M_{i}$ for all $1 \leq i <j\leq k$.
Flag matroids are the Coxeter matroids of type $A$ \cite{BGW03}.
The flag variety $Fl(r_1, \ldots, r_k;n)$ can be embedded in $\mathbb{P}\big( \kk^{\binom{[n]}{r_1}}\big) \times \cdots \times\mathbb{P}\big( \kk^{\binom{[n]}{r_k}}\big)$,
where it is cut out by the quadratic incidence-Pl\"{u}cker relations
(see Equation \eqref{eqn:tropIP})
in addition to the Grassmann-Pl\"ucker relations. 
For a fixed flag matroid $\MM$, one can modify these relations to cut out only the points in the flag variety which realize $\MM$.  We define the \textbf{flag Dressian of $\MM$}, denoted $FlDr(\MM)$, as the tropical prevariety of these equations, and establish several characterizations.

\begin{maintheorem}
\label{thm:main}
Let $\bmu = (\mu_1, \ldots, \mu_k)$ be a sequence of valuated matroids such that its sequence of underlying matroids $\MM = (M_1, \ldots, M_k)$ is a flag matroid.  Then the following are equivalent:
\begin{enumerate}[label = (\alph*)]
\item \label{thm:pt-tropIP} $\bmu$ is a point on $FlDr(\MM)$, i.e.\ it satisfies tropical incidence-Pl\"ucker relations, 
\item \label{thm:pt-valflag} $\bmu$ is a valuated flag matroid  with underlying flag matroid $\MM$,
\item \label{thm:pt-subdiv} $\bmu$ induces a subdivision of the base polytope of $ \MM$ 
into base polytopes of flag matroids, and
\item \label{thm:pt-troplin} the projective tropical linear spaces $\overline{\trop}(\mu_i)$ form a flag $\overline{\trop}(\mu_1) \subseteq \cdots \subseteq \overline{\trop}(\mu_{k})$
\end{enumerate}
The concepts appearing here are introduced in \Cref{defn:flagDressian} for \ref{thm:pt-tropIP},
\Cref{defn:valflagmat} for \ref{thm:pt-valflag}, \Cref{defn:baseflagmat} for \ref{thm:pt-subdiv}, and \Cref{thm:main2}.\ref{thm2:pt1} for \ref{thm:pt-troplin}.
\end{maintheorem}


%

\begin{eg}
\label{eg:pointplane}
Consider the flag matroid $\UU_{1,3;4} = (U_{1,4},U_{3,4})$ consisting of uniform matroids on 4 elements. Its flag Dressian of $\UU_{1,3;4}$, denoted $FlDr({\UU_{1,3;4}})$, is
\begin{enumerate}[label=(\alph*)]
\item the tropical prevariety of the flag variety $Fl(1,3;4)$ embedded in $\mathbb{P}\big(\kk^{\binom{4}{1}}\big) \times\mathbb{P}\big(\kk^{\binom{4}{3}}\big)$ by the single equation $p_1p_{234}-p_2p_{134}+p_3p_{124}-p_4p_{123}$,
\item the valuations on $U_{1,4}$ and $U_{3,4}$ making $(U_{1,4},U_{3,4})$ a valuated flag matroid,
\item the space parameterizing weights that induce flag matroidal subdivisions of the base polytope of $(U_{1,4},U_{3,4})$, which is the cuboctohedron 
$
\operatorname{Conv}\big(
\sigma (1,1,2,0) \mid \sigma \in S_4
\big)\subset \mathbb{R}^4 
$, and
\item the space parameterizing the data of a (tropical) point on a tropical plane.
\end{enumerate}
It is a pure simplicial fan in $\RR^{\binom{4}{1}}/\RR\mathbf 1 \times \RR^{\binom{4}{3}}/\RR\mathbf 1$ of dimension 5 with a 3 dimensional lineality space.  The 3 dimensional lineality space corresponds to the 3 dimensional freedom of selecting the location of the vertex of the tropical plane. Modulo the lineality space, it consists of 4 rays and 6 two-dimensional cones, as does a tropical plane in 3-space.  Up to combinatorial equivalence, there are two types of nontrivial subdivisions of the cuboctohedron into smaller flag matroid polytopes, with the corresponding data of a point in a tropical plane as indicated in Figure \ref{fig:df134_subs}. For more examples, see Figures \ref{fig:df124_subs}, \ref{fig:fl4_cells_edges}, and \ref{fig:fl4_cells}.
\end{eg}

\begin{figure}[h]
\centering
\begin{subfigure}{.24\textwidth}
  \centering
  \includegraphics[width=.9\linewidth]{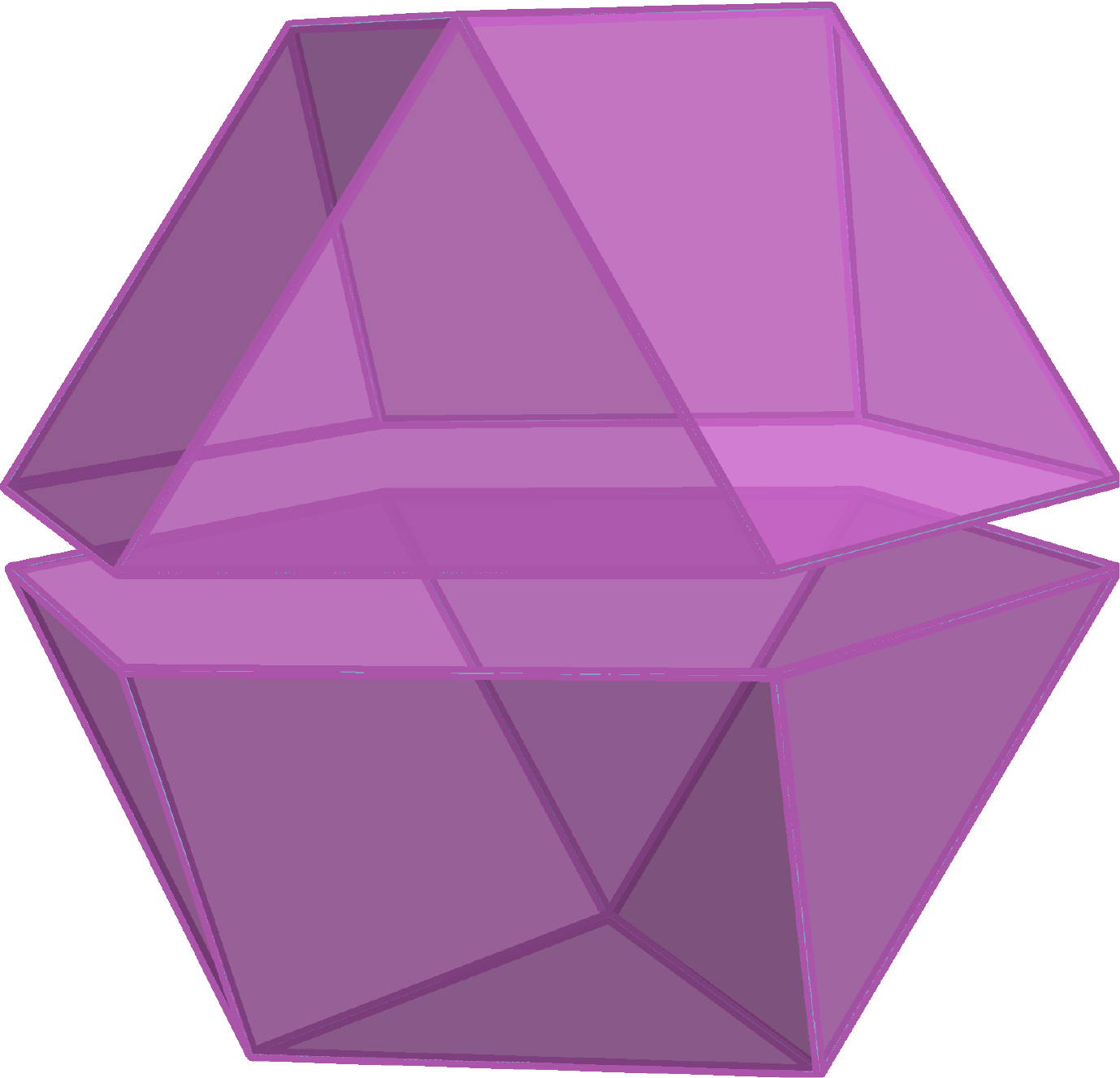}
\end{subfigure}%
\begin{subfigure}{.24\textwidth}
  \centering
  \includegraphics[width=.9\linewidth]{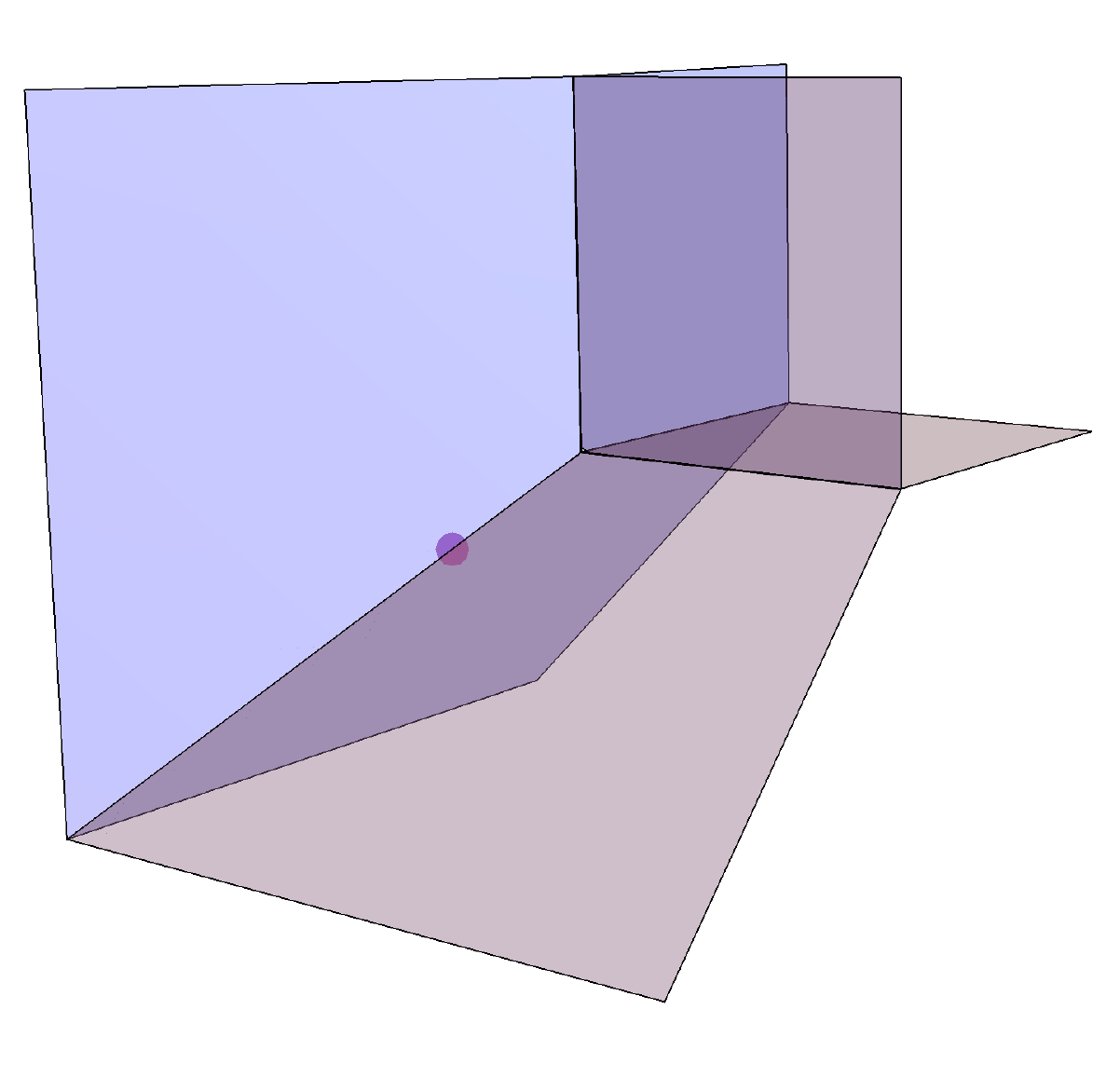}
\end{subfigure}%
\begin{subfigure}{.24\textwidth}
  \centering
  \includegraphics[width=.9\linewidth]{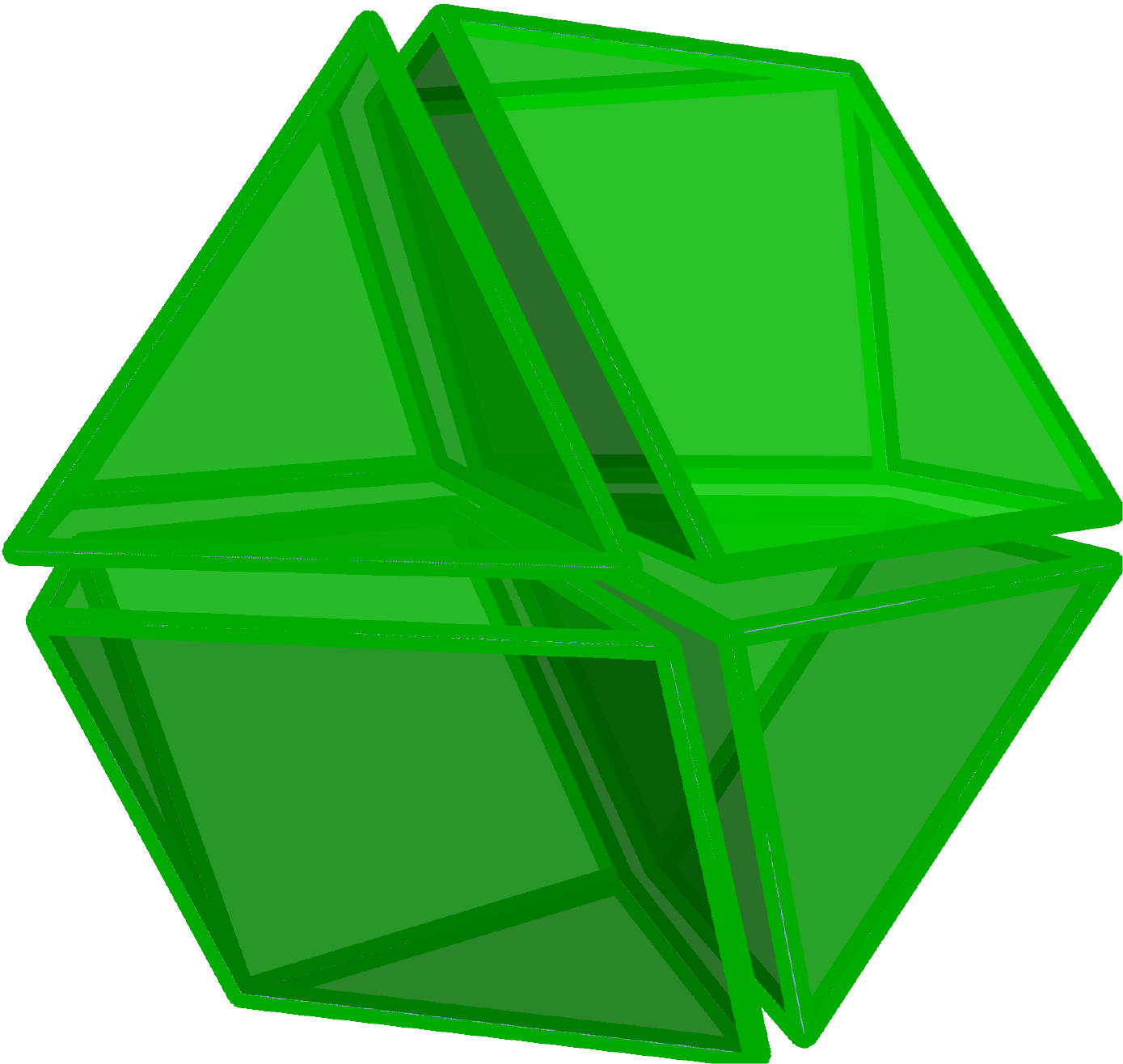}
\end{subfigure}
\begin{subfigure}{.24\textwidth}
  \centering
  \includegraphics[width=.9\linewidth]{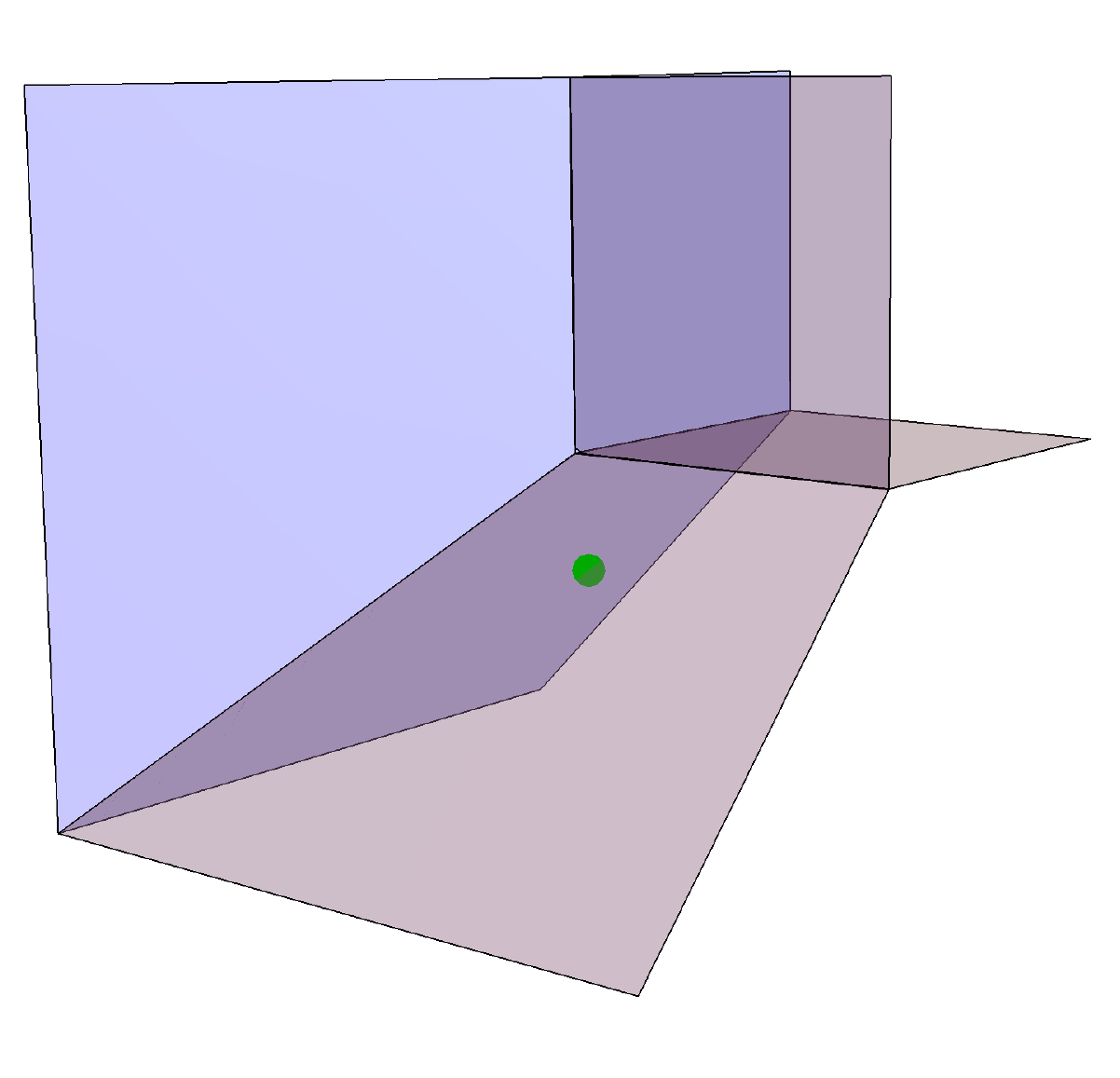}
\end{subfigure}
\caption{Base polytope subdivisions and their tropical flags for $(U_{1,4},U_{3,4})$.}
\label{fig:df134_subs}
\end{figure}


A fundamental tool in our proof of \Cref{thm:main} is the notion of \emph{projective tropical linear spaces}.  The usual tropical linear spaces are defined only for matroids without loops, which is a harmless restriction in studying matroids, but not in studying flag matroids (\Cref{rem:looplessbad}).
In order to treat matroids with and without loops consistently, we introduce projective tropical linear spaces, which have previously appeared in the literature in various guises (\Cref{rem:previoustroplin}).
We collect their characterizations, adding two new ones  (\ref{thm2:pt3} and \ref{thm2:pt5}) to this list (\Cref{thm:main2}).  See \S \ref{subsec:projtropgeo} for terminology in projective tropical geometry, and \S \ref{subsection:dressians} for terminology concerning valuated matroids.

\begin{maintheorem}\label{thm:main2}
Let $\mu$ be a valuated matroid on a ground set $[n]$.  Let $\ell\subseteq [n]$ be the set of loops of its underlying matroid.  The following sets in the tropical projective space $\PP(\TT^{[n]})$ coincide:
\begin{enumerate}[label=(\roman*)]
\item \label{thm2:pt1} The \textbf{projective tropical linear space}, defined as
\[
\overline{\trop}(\mu) := \displaystyle \bigcup_{\emptyset\subseteq S \subsetneq [n]} \Big( \trop(\mu/_S)\times \{\infty\}^S \Big) \subset \PP(\TT^{[n]}),
\]

\item \label{thm2:pt2}The projective tropical prevariety of the valuated circuits of $\mu$, i.e.\
\[
\bigcap_{\tiny \begin{matrix} \textnormal{valuated}\\ \textnormal{circuits $\mathbf C$}\end{matrix}} \Big\{ \overline\bu \in \PP(\TT^{[n]}) \ \left| \ \textnormal{the minimum is achieved at least twice among } \{C_i + v_i\}_{i\in [n]}\Big\}\right.,
\]

\item \label{thm2:pt3} The union of coloopless cells of the closure of the dual complex of $\mu^*$ in $\PP(\TT^{[n]})$, i.e.\
\[
\begin{split}
\Big\{\overline\bu\in \PP(\TT^{[n]}) \mid \Delta_{\mu^*}^{\overline\bu} \textnormal{ is a base polytope of a coloopless matroid}\Big\},
\end{split}
\]

\item \label{thm2:pt4} The tropical span of the valuated cocircuits of $\mu$, i.e.\
\[
\left\{\left. \begin{matrix} \textnormal{the image in $\PP(\TT^{[n]})$ of} \\
(a_1\odot \mathbf{C_1^*}) \oplus \cdots \oplus (a_l \odot \mathbf{C_m^*}) \in \TT^{[n]}
\end{matrix}
\ \right|\  \begin{matrix}
\mathbf{C_i^*}\in \TT^{[n]} \textnormal{ a valuated cocircuit of } \mu,\\
\ a_i \in \RR, \quad \forall 1\leq i \leq m
\end{matrix}
\right\},
\]

\item \label{thm2:pt5}The closure of $\trop(\mu/_\ell) \times \{\infty\}^\ell$ inside $\PP(\TT^{[n]})$.
\end{enumerate}
\end{maintheorem}

\medskip
We apply \Cref{thm:main} to establish a relation between Dressians and flag Dressians, and deduce a realizability result for valuated flag matroids.  First, let us recall that the \textbf{Dressian} $Dr(r;n)$ is defined as the union of $Dr(M)$ over all matroids $M$ of rank $r$ on $[n]$.  We define the \textbf{flag Dressian} $FlDr(r_1, \ldots, r_k;n)$ as the union of $FlDr(\MM)$ over all flag matroids $\MM$ of rank $(r_1, \ldots, r_k)$ on $[n]$.

\newtheorem*{thm:fibration}{\Cref{thm:fibration} \& \Cref{thm:Fl=FlDr}}
\begin{thm:fibration}
The natural isomorphism $\RR^{\binom{[n+1]}{r+1}} \overset\sim\to \RR^{\binom{[n]}{r}} \times \RR^{\binom{[n]}{r+1}}$ induces a surjective map from a subset of $Dr(r+1;n+1)$ to $FlDr(r,r+1;n)$, whose fiber over each point is isomorphic to $\RR$.
As a consequence, every valuated flag matroid on a ground set of size $\leq 5$ is realizable; the tropicalization of a flag variety $Fl(r_1, \ldots, r_k;n)$ coincides with the flag Dressian $FlDr(r_1, \ldots, r_k;n)$ whenever $n\leq 5$.
\end{thm:fibration}

The tropicalization of a flag variety may differ from  the flag Dressian when $n\geq 6$.  See \Cref{eg:m4u26}.
%
%
%
%


\subsection{Previous works}
\label{previousworks}
In the unpublished manuscript \cite{Haq12}, the author established \ref{thm:pt-tropIP}$\iff$\ref{thm:pt-troplin} in \Cref{thm:main} for loopless matroids.\footnote{We also note an error in the proof of Proposition 3 of \cite{Haq12}: there can be many more sets $I$ satisfying $T\cap S \subseteq I \subseteq T \cup S$ than are considered. This nullifies his Lemma 2 and Theorem 2 on the convex hull of the base polytopes of a matroid quotient.}
In \cite[\S4.3]{MS15}, the flag Dressian $FlDr(1,r;n)$ appeared implicitly as the universal family over $Dr(r;n)$.
In \cite{BLMM17}, the authors computed the tropicalizations of the full flag varieties $Fl(1,2,3;4)$ and $Fl(1,2,3,4;5)$ in order to compute toric degenerations.  In a related work \cite{FFFM19}, the authors identified some distinguished maximal cones in the tropicalizations of full flag varieties to study PBW-degenerations.
In \cite[\S5]{FM16} and \cite[\S6]{FM19}, in order to describe the parameter space of matroids over valuation rings, the authors studied the space of valuated flag matroids $(\mu_1, \mu_2)$ of ranks $(r,r+1)$ given a fixed valuated matroid $\mu_2$.
In \cite{JMRS20}, the authors studied the same space as tropicalized Fano schemes under the assumption that $\mu_2$ is realizable.

\subsection{Organization}
\label{organization}
In \S\ref{section:projtrop}, we review projective tropical geometry,  dual complexes, and M-convex functions.  In \S\ref{section:projtroplin}, we review Dressians of matroids and prove \Cref{thm:main2}.  In \S\ref{section:morph}, after a review of flag matroids, we introduce flag Dressians and prove \Cref{thm:main}.  In \S\ref{sec:realizability}, we apply \Cref{thm:main} to relate Dressians and flag Dressians, and obtain a realizability result for valuated flag matroids.

\subsection{Notation}
\label{notations}
For a finite set $E$, we write $\{\be_i \mid i\in E\}$ for the standard basis of $\RR^{E}$, and denote $\be_S := \sum_{i\in S} \be_i$ for subsets $S\subseteq E$.
All-one-vectors $(1,1,\ldots, 1)$ in appropriate coordinate spaces are denoted $\mathbf 1$.
Let $\langle \cdot, \cdot \rangle$ be the standard inner product.  We will follow the "$\min$" convention for all polyhedral operations, such as taking faces and coherent subdivisions.  Likewise, the tropical semifield $\TT = \RR \cup \{\infty\}$ is the min-plus algebra, with operations $a\odot b := a+b$ and $a\oplus b := \min\{a,b\}$.  The topology on $\TT$ is the standard one that makes $\TT$ homeomorphic to $(-\infty,0]$.  The field $\kk$ is algebraically closed, with a (possibly trivial) valuation $\operatorname{val}:\kk\to \TT$.  Denote $[n] = \{1,\ldots, n\}$.  For $0\leq r \leq n$, the set of $r$-subsets of $[n]$ is denoted $\binom{[n]}{r}$.

\section{Preliminaries 
}
\label{section:projtrop}

In \S\ref{subsec:projtropgeo}, we review tropical projective spaces and their products, since these are the ambient spaces of Dressians, flag Dressians, and projective tropical linear spaces.
In \S\ref{subsec:mixedsub}, we review point configurations, dual complexes, and mixed subdivisions, since we will need these notions to study mixed subdivisions of base polytopes of flag matroids in \S\ref{subsec:flagmatsubdiv}.
Our novel contribution here is \Cref{thm:mixed} concerning coherence of mixed subdivisions.
In \S\ref{subsec:mconvex}, we review M-convex functions because the structure of their dual complexes will play a central role in the proof of \Cref{thm:main2} and \Cref{thm:troplinsubdiv}.
 \Cref{thm:Mconvexdualcplx} explicitly describes the closures of their dual complexes inside tropical projective spaces. 
Let $E$ be a finite set throughout.

\subsection{Projective tropical geometry}
\label{subsec:projtropgeo}

We review projective tropical geometry, and explain the underlying algebraic geometry in the remarks.  See \cite[Chapter 6]{MS15} for a detailed treatment.

\begin{defn}
Let $E = [n]$. The \textbf{tropical projective space} $\PP(\TT^E)$ is 
\[
\begin{split}
\PP(\TT^E) :=& \big(\TT^E \setminus \{(\infty, \ldots, \infty)\} \big) / \RR\mathbf 1\\
=& \{ \bu \in \TT^E \mid \bu \neq (\infty, \ldots, \infty)\}/ \sim, \ \ \textnormal{where $\bu\sim \bu'$ if $\bu' = \bu+ c\mathbf 1$ for some $c\in \RR$.}
\end{split}
\]
\end{defn}

For $\bu = (u_i)_{i\in E}$ in $\RR^E$ or $\TT^E$, write $\overline \bu$ for its image in $\RR^E/\RR\mathbf 1$ or $\PP(\TT^E)$.  The \textbf{support} of $\bu$ is $\operatorname{supp}(\bu) := \{i\in E \mid u_i \neq \infty\}$.  For a nonempty subset $S\subseteq E$, denote by
\[
T_S := \{\overline \bu \in \PP(\TT^E) \mid \operatorname{supp}(\bu) = S\},
\]
the image of $\RR^{S}\times \{\infty\}^{E\setminus S}$ in $\PP(\TT^E)$. The set $T_E = \RR^E/\RR\mathbf 1$ is the \emph{tropical projective torus}. By abuse of notation, we often identify $\RR^S/\RR\mathbf 1$ with $T_S$, and $\PP(\TT^{S})$ with the closure of $T_S$ in $\PP(\TT^E)$.  The subsets $\{T_S\}_{\emptyset \subsetneq S\subseteq E}$ partition $\PP(\TT^E)$.

\begin{rem}
The space $\PP(\TT^E)$ is the \textbf{tropicalization} of the projective space $\PP(\kk^E)$.  The projective space $\PP(\kk^E)$ is a toric variety with the projective torus $(\kk^*)^E/\kk^*$.  For each nonempty subset $S\subseteq E$, the torus orbit $O_S := \big((\kk^*)^{S}\times \{0\}^{E\setminus S}\big) /\kk^*$ in $\PP(\kk^E)$ tropicalizes to be the stratum $T_S$ of $\PP(\TT^E)$.  We often identify $(\kk^*)^S/\kk^*$ with $O_S$, and $\PP(\kk^S)$ with the closure $\overline{O_S} = \{\mathbf y\in \PP(\kk^E) \mid y_i = 0 \textnormal{ if } i\notin S\}$.  The orbits $\{O_S\}_{\emptyset\subsetneq S \subseteq E}$ partition the space $\PP(\kk^E)$.  See \cite[\S6.2]{MS15} or \cite[\S3.2]{MR} for tropicalizations of toric varieties in general.
\end{rem}

Let $\mathcal A$ be a finite subset of $\ZZ_{\geq 0}^E$.  A \textbf{tropical polynomial} $F$ with \textbf{support} $\operatorname{supp}(F) = \mathcal A$ is 
\[
F = \bigoplus_{\bv \in \mathcal A} c_\bv \odot \mathbf x^{\bigodot \bv}.
\]
It represents the function $\TT^E\to \TT, \ (x_i)_{i\in E} \mapsto \min_{\bv \in \mathcal A}\{ c_\bv + \sum_{i\in E} v_i \cdot x_i\}$.  Here, by convention  $0\odot \infty = 0$ and $a\odot \infty = \infty$ if $a \neq 0$.  We always assume that a tropical polynomial $F$ is homogeneous; that is, there exists $d\in \ZZ_{\geq 0}$ such that $d = \sum_{i\in E} v_i$ for all $\bv \in \operatorname{supp}(F)$.

\begin{defn}
Let $F$ be a tropical polynomial with support in $\ZZ_{\geq 0}^E$.  We define the \textbf{projective tropical hypersurface of $F$} to be
\[
\overline\trop(F) := \left\{\overline{\bu} \in \PP(\TT^E) \ \middle | \textnormal{ the minimum in } \Big\{ c_\bv + \sum_{i\in E} v_i \cdot u_i \Big\}_{\bv\in \operatorname{supp}(F)} \textnormal{ is achieved at least twice}\right\}.
\]
When $\{ c_\bv + \sum_{i\in E} v_i \cdot u_i\}_{\bv\in \operatorname{supp}(F)} = \{\infty\}$, by convention the minimum in $\{c_\bv + \sum_{i\in E} v_i \cdot u_i\}_{\bv\in \operatorname{supp}(F)}$ is said to be achieved at least twice even if $\operatorname{supp}(F)$ is a single element.
The set $\overline\trop(F)$ is well-defined in $\PP(\TT^E)$ because one may pass from $\TT^E$ to $\PP(\TT^E)$ by the homogeneity of $F$. 
\end{defn}

\begin{eg}
\label{ex:trophyp}
Let $F = x_0 \odot x_1 \oplus x_0 \odot x_2 \oplus x_1 \odot x_3 \oplus x_2 \odot x_3  =  \min(x_0 + x_1, x_0 + x_2, x_1 + x_3, x_2 + x_3)$. Then the projective tropical hypersurface $\overline{\trop}(F) \subset \mathbb{P}(\mathbb{T}^{\{0,1,2,3\}})$ is as pictured in Figure \ref{fig:trophypex}.
\begin{figure}[h]
    \centering
    \includegraphics[height = 1.4 in]{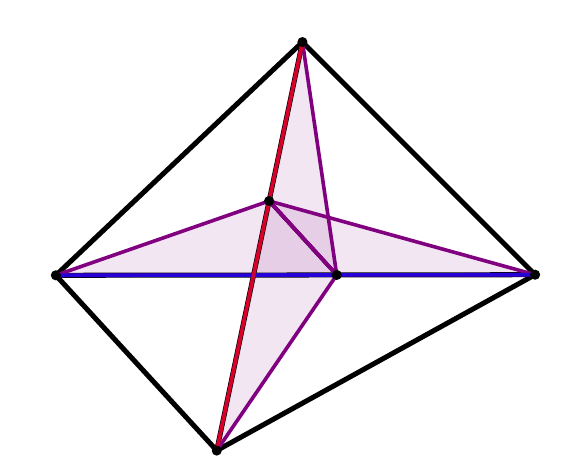}
    \caption{The projective tropical hypersurface from Example \ref{ex:trophyp}. The red line is where $x_0=x_3=\infty$ and the blue line is where $x_1=x_2=\infty$.}
    \label{fig:trophypex}
\end{figure}
\end{eg}

Suppose $F$ is multi-homogeneous; that is, there is a partition $E = \bigsqcup_{j \in J} E_j$ and integers $\{d_j\}_{j\in J}$ such that $d_j = \sum_{i\in E_j} v_i$ for all $j\in J$ and $\bv \in \operatorname{supp}(F)$.  Then the \textbf{multi-projective tropical hypersurface} of $F$ is defined analogously as a subset of $\prod_{j\in J} \PP(\TT^{E_j})$.

\begin{defn}
If $F_1, \ldots, F_l$ are tropical polynomials with supports in $\ZZ_{\geq 0}^E$, we define their \textbf{projective tropical prevariety} to be
\[
\overline\trop(F_1, \ldots, F_l) := \bigcap_{i=1}^l \overline\trop(F_i) \subset \PP(\TT^E).
\]
If there is a common partition $S = \bigsqcup_{j\in J} E_j$ such that each $F_i$ is multi-homogeneous in $S$, the \textbf{multi-projective tropical prevariety} is defined analogously as a subset of $\prod_{j\in J}\PP(\TT^{E_j})$.
Multi-projective tropical prevarieties are closed.
\end{defn}
In \S\ref{section:projtroplin}, Dressians and projective tropical linear spaces are defined as projective tropical prevarieties in $\PP\big(\TT^{\binom{[n]}{r}}\big)$ and $\PP(\TT^{[n]})$, respectively.  In \S \ref{section:morph}, flag Dressians will be defined as multi-projective tropical prevarieties in $\PP\big(\TT^{\binom{[n]}{r_1}}\big)\times \cdots \times \PP\big(\TT^{\binom{[n]}{r_k}}\big)$.

\begin{rem}\label{rem:usualtrop}
The intersection $\overline\trop(F) \cap T_E \subset \RR^E/\RR\mathbf 1$ is the usual \textbf{tropical hypersurface} of a tropical polynomial $F$, and is denoted $\trop(F)$.  More generally, for a nonempty subset $S\subseteq E$, consider the intersection $\overline\trop(F) \cap T_S$ as a subset of $\RR^{S}/\RR\mathbf 1$.  Then it is equal to $\trop(F_S)$, where $F_S$ is the tropical polynomial obtained from $F$ by keeping only the terms with exponent supports in $\ZZ_{\geq 0}^{S}$. 
The set $\overline\trop(F)$ is the closure of $\trop(F)$ in $\PP(\TT^E)$ when $F$ has no nontrivial monomial factors, i.e.\ there is no $\mathbf 0 \neq \bv' \in \ZZ_{\geq 0}^E$ such that $\bv - \bv' \in \ZZ_{\geq 0}^E$ for all $\bv\in \operatorname{supp}(F)$.  
\end{rem}

We now give the underlying algebraic geometry. See \cite[\S6.2]{MS15} for proofs of statements.

\begin{rem}\label{rem:projtrophyper}
Let $\operatorname{val}: \kk \to \TT$ be a (possibly trivial) valuation on $\kk$.  Let
$Y = V(f) \subset \PP(\kk^E)$ be a projective subvariety defined by a homogeneous polynomial
\[
f = \sum_{\bv \in \ZZ_{\geq 0}^E} c_\bv \mathbf x^{\bv} \in \kk[x_i \mid i\in E] \qquad\textnormal{(all but finitely many $c_\bv$ are zero)}.
\]
The \textbf{projective tropicalization} of $Y$, denoted $\overline\trop(Y)$, is the set $\overline\trop(f^{\trop})$ where
\[
f^{\trop} = \bigoplus_{\bv\in \operatorname{supp}(f)}\operatorname{val}(c_\bv) \odot \mathbf x^{\bigodot \bv}.
\]
If $f = 0$, then $\overline\trop(Y) = \PP(\TT^E)$.  Recall the notation $O_S = \big((\kk^*)^{S}\times \{0\}^S\big) /\kk^*$ for a nonempty subset $S\subseteq E$.  For $\mathring Y := Y \cap O_E$ a subvariety of the projective torus $(\kk^*)^{E}/\kk^*$, 
the usual tropical hypersurface $\trop(f^{\trop}) \subset \RR^S/\RR\mathbf 1$ is the usual tropicalization of $\mathring Y$, denoted $\trop(\mathring Y)$.  More generally, consider $\mathring Y_S := Y\cap O_S$, regarded as a subvariety in $(\kk^*)^{S}/\kk^*$.  Then $\trop(\mathring Y_S)$ is equal to $\overline\trop(Y) \cap T_S$, regarded as a subset of $\RR^S/\RR\mathbf 1$.
The set $\overline\trop(Y)$ is the closure of $\trop(\mathring Y)$ in $\PP(\TT^E)$ when $Y$ is the closure of $\mathring Y$ in $\PP(\kk^E)$.
\end{rem}


\begin{rem}\label{rem:projtrop}
Suppose now that $Y\subset \PP(\kk^E)$ is a projective subvariety defined by a homogeneous ideal $I\subset \kk[x_i \mid i\in S \subset E]$.  The projective tropicalization of $Y$ is defined as
\[
\overline\trop(Y) := \bigcap_{f\in I} \overline\trop(f^{\trop}),
\]
which is a finite intersection for a suitable choice of generators of $I$, and hence $\overline\trop(Y)$ is a projective tropical prevariety. As in the hypersurface case (\Cref{rem:projtrophyper}), the usual tropicalization of $\mathring Y = Y \cap O_E$ is $\trop(\mathring Y) := \bigcap_{f\in I} \trop(f^{\trop})$.  For a nonempty subset $S \subseteq E$, we have $\overline\trop(Y) \cap T_S = \trop(\mathring Y_S)$ where $\mathring Y_S := Y\cap O_S$.  The set $\overline\trop(Y)$ is the closure of $\trop(\mathring Y)$ when $Y$ the closure of $\mathring Y$ in $\PP(\kk^S)$.  If $I$ is principal, generated by $f$, then $\overline{\trop}(Y) = \overline{\trop}(f^{\trop})$, but in general, the set $\overline\trop(Y)$ may not equal $\bigcap_{i=1}^l \overline\trop(f_i^{\trop})$ for an arbitrary generating set $\{f_1, \ldots, f_l\}$ of $I$.
\end{rem}

\subsection{Point configurations, dual complexes, and mixed subdivisions}
\label{subsec:mixedsub}

We review point configurations, dual complexes of their coherent subdivisions, and mixed subdivisions. Point configurations, which generalize the notion of subsets of points, are necessary for discussing mixed subdivisions.  See \cite{DLRS10} for a detailed treatment.  Our novel contribution here is \Cref{thm:mixed} concerning mixed coherent subdivisions.


\begin{defn}
Let $\mathcal A$ be a finite index set.  A \textbf{point configuration} $(\mathcal A, \mathbf a)$ \textbf{in} $\RR^E$ is a map $\mathbf a_{(\cdot)}: \mathcal  A \to \RR^{E}$.  In other words, it is a finite set of points $\{\mathbf a_i \in \RR^E: i\in \mathcal A\}$ labeled by the set $\mathcal A$, where some points may have multiple labels.
\end{defn}

We often abbreviate $(\mathcal A, \mathbf a)$ to $\mathcal A$ when the map $\mathbf a$ is understood.  For $A\subset \RR^E$ a finite subset, we write $A$ also for the point configuration $(A, s\mapsto s)$.  
For $Q$ a lattice polytope in $\RR^E$, we write $Q$ for the point configuration of its lattice points.  We write $\operatorname{Conv}(\mathcal A)$ for the polytope $\operatorname{Conv}(\mathbf a_i \mid i\in \mathcal A) \subset \RR^E$.

\proof[Assumption] The point configuration $\mathcal A$ is always integral, i.e.\ the image $\{\mathbf a_i\}_{i\in \mathcal A}$ lies in $\ZZ^E$, and it is homogeneous, i.e.\ there exists $d \in \ZZ$ such that $d = \langle \be_E, \mathbf a_i\rangle$ for all $i\in \mathcal A$, where $e_E = \sum_{i\in E} e_i$.

\smallskip
For a point configuration $(\mathcal A, \mathbf a)$, a subset $\mathcal A' \subset \mathcal A$ defines a subconfiguration $(\mathcal A', \mathbf a|_{\mathcal A'})$. In particular, a vector $\overline\bu \in \RR^E/\RR\mathbf 1$ defines a subconfiguration $\mathcal A^{\overline \bu}$ by 
\[
 \mathcal A^{\overline \bu}:= \left\{i\in \mathcal A \mid \langle \bu, \mathbf a_i \rangle = \min_{j\in \mathcal A}\langle \bu, \mathbf a_j\rangle \right\}.
\]
This does not depend on the choice of the representative $\bu$ of $\overline \bu$ because $\mathcal A$ is homogeneous.
A subconfiguration $\mathcal F \subset \mathcal A$ arising in this way 
is called a \textbf{face} of $\mathcal A$, denoted $\mathcal F \leq \mathcal A$.  

\begin{defn}
A collection $\Delta$ of subconfigurations of $\mathcal A$ is a \textbf{subdivision of $\mathcal A$} if
\begin{enumerate}
    \item for all $\mathcal F \in \Delta$ and $\mathcal F' \leq \mathcal F$, one has $\mathcal F'\in \Delta$, and
    \item the set of polytopes $\{\operatorname{Conv}(\mathcal F)\}_{\mathcal F \in \Delta}$ forms a polyhedral subdivision of $\operatorname{Conv}(\mathcal A)$.  So, we have $\bigcup \{\operatorname{Conv}(\mathcal F)\}_{\mathcal F \in \Delta} =\operatorname{Conv}(\mathcal A)$, and for any $\mathcal F_1 \neq \mathcal F_2 \in \Delta$, the intersection $\operatorname{Conv}({\mathcal F_1}) \cap \operatorname{Conv}({\mathcal F_2})$ of $\operatorname{Conv}({\mathcal F_1})$ and $\operatorname{Conv}({\mathcal F_2})$ is a proper face of each.
\end{enumerate}
The elements $\mathcal F\in \Delta$ are called the \textbf{faces of $\Delta$}.  A subdivision of $\mathcal A$ is \textbf{tight} if every $i\in \mathcal A$ is in some face of the subdivision.
\end{defn}

We will study subdivisions of $\mathcal A$ induced by weights.
A \textbf{weight} on a point configuration $(\mathcal A,\mathbf a)$ is a function $w: \mathcal A \to \RR$.  Like point configurations, we write $w^{\overline\bu}$ for the restriction $w|_{\mathcal A^{\overline\bu}}$ for $\overline\bu\in \RR^E/\RR\mathbf 1$.  We set the following notations for the subdivision induced by a weighted point configuration $w$.

\begin{notation}
\label{notation:weights}
\hfill
\begin{itemize}
\item $\Delta_w$ is the \textbf{coherent subdivision} of $\mathcal A$, consisting of the lower faces of the point configuration $\Gamma_w(\mathcal A) := (\mathcal A, (\mathbf a, \nu))$ where $(\mathbf a, w): i \mapsto (\mathbf a_i, w(i))\in \RR^E\times \RR$ for $i\in \mathcal A$.
\item $\Delta^{\overline\bu}_w$ is the face of the coherent subdivision $\Delta_w$ corresponding to ${\overline\bu} \in \RR^E/\RR\mathbf 1$, defined by
\[
\Delta_w^{\overline\bu} := \Gamma_w(\mathcal A)^{({\bu},1)} = \left \{i\in \mathcal A\  \middle  | \   \langle {\bu}, \mathbf a_i \rangle +w(i) = \min_{j\in \mathcal A}\Big( \langle {\bu},\mathbf a_j \rangle +w(j)  \Big) \right\}.
\]
\item $\Sigma_w$ is the \textbf{dual complex} in $\RR^E/\RR\mathbf 1$ of the coherent subdivision $\Delta_w$.  It is a polyhedral complex consisting of polyhedra corresponding to faces of $\Delta_w$ by
\begin{equation}\label{eqn:dualcplx}\tag{\textdagger}
\{ {\overline\bu} \in \RR^E/\RR\mathbf 1 \mid \Delta_w^{\overline\bu} \geq \mathcal F\} \longleftrightarrow \mathcal F \in \Delta_w.
\end{equation}
The relative interiors $\{{\overline\bu} \in \RR^E/\RR\mathbf 1 \mid \Delta_w^{\overline\bu} = \mathcal F\}$ as $\mathcal F$ ranges over all faces of $\Delta_w$ partition $\RR^E/\RR\mathbf 1$.  We call the relative interiors the \textbf{cells} of the polyhedral complex $\Sigma_w$. 
\end{itemize}
\end{notation}

We note a useful observation.

\begin{lem}\label{lem:linearshift}
Let $w$ be a weight on a point configuration $(\mathcal A,\mathbf a)$ in $\RR^E$, and ${\bu} \in \RR^E$.  Consider a new weight defined by $i \mapsto w(i) + \langle \bu , \mathbf a_{i}\rangle$ for $i\in \mathcal A$.  Then $\Delta_w^{\overline\bu} = \Delta_{w(\cdot) + \langle {\bu},  \mathbf a_{(\cdot)}\rangle}^{\mathbf 0}$.
\end{lem}

In \Cref{cor:Mconvexdualcplx}, we will extend the correspondence \eqref{eqn:dualcplx} to a correspondence between points of $\PP(\TT^E)$ and projections of faces of $\Delta_w$ for a particular family of weight configurations $w$.  For now, we discuss mixed subdivisions of Minkowski sums,   because Minkowski sums of base polytopes of matroids and their mixed subdivisions are the focus of \S\ref{subsec:flagmatsubdiv}.

\begin{defn}
\label{def:mink_sum_weights}
Let $(\mathcal A_1, \mathbf {a_1}), \ldots, (\mathcal A_k, \mathbf {a_k})$ be point configurations in $\RR^E$.  Their \textbf{Minkowski sum}, denoted $\sum_{i=1}^k \mathcal A_i$, is a point configuration $(\mathcal A_1 \times \cdots \times \mathcal A_k, \sum_i \mathbf{a_i})$ defined by
\[
\textstyle \sum_i \mathbf{a_i}: \displaystyle (j_1, \ldots, j_k) \mapsto \sum_{i=1}^k \mathbf{a_i}_{j_i} \quad \textnormal{ for } (j_1, \ldots, j_k) \in \mathcal A_1 \times \cdots \times \mathcal A_k.
\]
If $w_1, \ldots, w_k$ are weights on $\mathcal A_1, \ldots, \mathcal A_k$ (respectively), then their Minkowski sum $\sum_{i} w_i$ is a weight on $\sum_i \mathcal A_i$ defined by $(j_1, \ldots, j_k) \mapsto \sum_{i=1}^k w_i(j_i)$.
\end{defn}

We will repeatedly make use of the following observation.

\begin{lem}\label{lem:Minkowskifaces}
Let $w = \sum_{i=1}^k w_i$ be a Minkowski sum of weight point configurations.  Then for $\overline \bu \in \RR^S/\RR\mathbf 1$, we have $\Delta_w^{\overline\bu} = \sum_{i=1}^k \Delta_{w_i}^{\overline\bu}$.
\end{lem}


\begin{defn}
A subdivision $\Delta$ of a Minkowski sum $\sum_i \mathcal A_i$ is \textbf{mixed} if there exist subdivisions $\Delta_1, \ldots, \Delta_k$ of $\mathcal A_1, \ldots, \mathcal A_k$ (respectively) such that each face $\mathcal F \in \Delta$ is a Minkowski sum $\sum_{i=1}^k \mathcal F_i$ of faces $\mathcal F_i$ of $\Delta_i$.  If there exist weights $w_i: \mathcal A_i \to \RR$ such that their Minkowski sum $w := \sum_i w_i$ satisfies $\Delta_w = \Delta$, we say that $\Delta$ is a \textbf{mixed coherent subdivision}, which is mixed by \Cref{lem:Minkowskifaces}.
\end{defn}

A priori, the terminology "mixed coherent subdivision" can be ambiguous:  if a weight $w$ on $\sum_i \mathcal A_i$ induces a coherent subdivision that is mixed, is $w$ necessarily a Minkowski sum of weights?  In general, the answer is no, as displayed in the following example.

\begin{eg}
\label{ex:hexagon}
Let $A = \{(1,0,0),(0,1,0),(0,0,1)\}$ and $B = \{(0,1,1),(1,0,1),(1,1,0)\}$ be two point configurations in $\mathbb{R}^3$. Their Minkowski sum $A+B$ is labeled by the nine elements of $A \times B$ and is the collection of points $\{(2,0,1),(2,1,0),(1,2,0),(0,2,1),(0,1,2),(1,0,2),(1,1,1)\}$. The first six points have unique labels, and the last point has three labels, because it arises in three ways: $(0,0,1)+(1,1,0)=(1,0,0)+(0,1,1)=(0,1,0)+(1,0,1)$. This is shown in Figure \ref{fig:hexagon}.
\begin{figure}[h]
    \centering
    \includegraphics[height = 1.5 in]{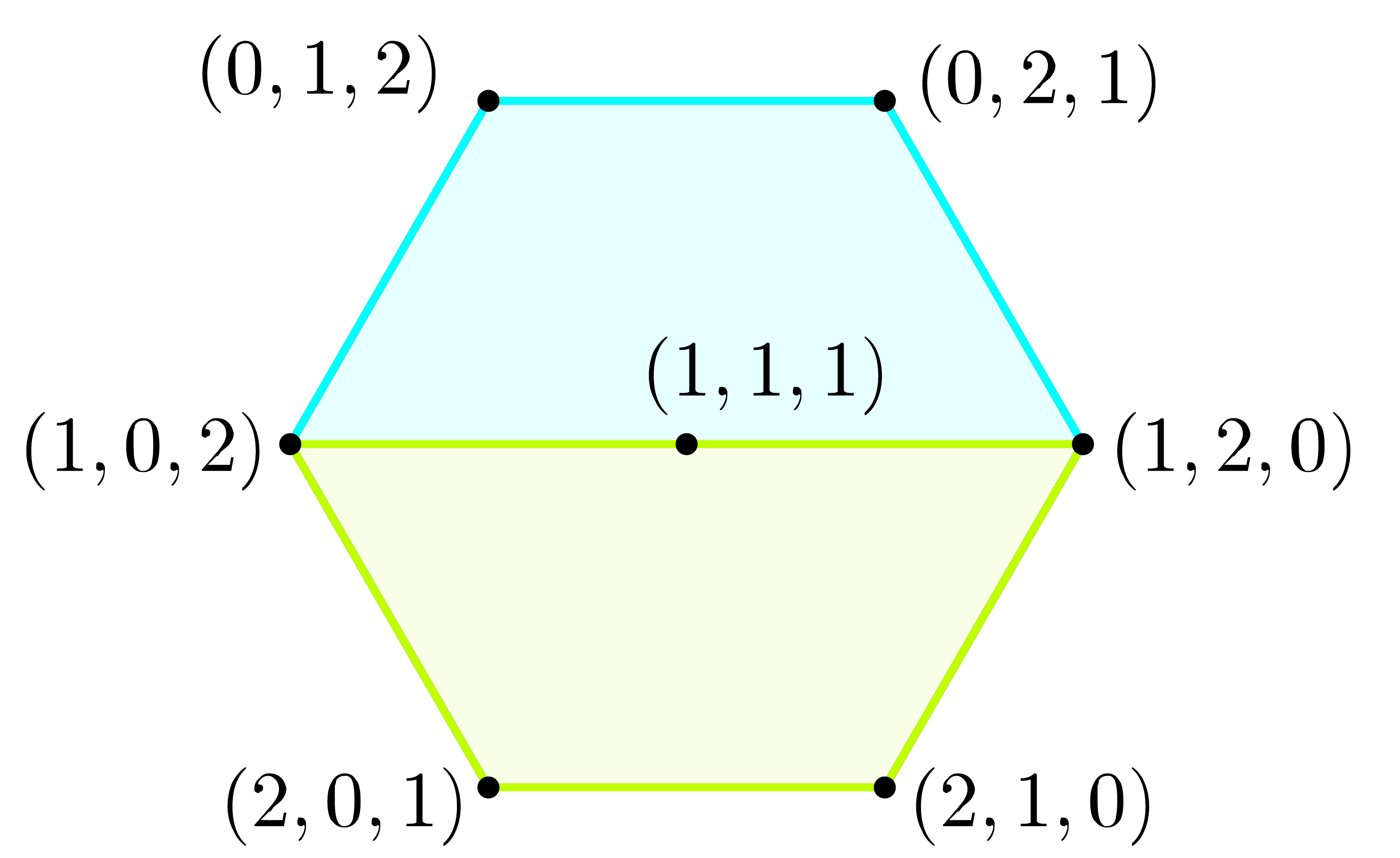}
    \caption{The point configuration in Example \ref{ex:hexagon}}
    \label{fig:hexagon}
\end{figure}

Consider the following two weight vectors.
$$
\begin{array}{cccccccccc}
&201&210&120&021&012&102&001+110&100+011&010+101\\
w_1 & 0&0&0&1&1&0&0&0&17\\
w_2 & 0&0&0&1&1&0&0&1&0\\
\end{array}
$$
Both $w_1$ and $w_2$ induce the subdivision indicated in Figure \ref{fig:hexagon}, which is mixed. The first is not a Minkowski sum of weights on $A$ and $B$, while the second is the Minkowski sum of weight vectors $w_A$ and $w_B$ where
\[
w_A : \begin{cases}
(1,0,0) \mapsto 0\\
(0,1,0)\mapsto 0\\
(0,0,1)\mapsto 0
\end{cases} \quad \textnormal{and}\quad 
w_B : \begin{cases}
(0,1,1) \mapsto 1\\
(1,0,1) \mapsto 0\\
(1,1,0) \mapsto 0.
\end{cases}
\]
This example shows that not every weight vector inducing a coherent subdivision that is mixed is a Minkowski sum of weights. However, there does exist a weight vector which is a Minkowski sum inducing the same subdivision.
\end{eg}

We establish the following weaker statement about coherent mixed subdivisions.  Together with \Cref{thm:troplinsubdiv}, it will imply a strengthening of the equivalence \ref{thm:pt-tropIP}$\iff$\ref{thm:pt-subdiv} in  \Cref{thm:main} (\Cref{cor:coherentflagmixed}).
We will only need \Cref{thm:mixed} for the proof of \Cref{cor:coherentflagmixed}.

\begin{thm}\label{thm:mixed}
Let $\mathcal A = \sum_{i=1}^k \mathcal A_i$ be a Minkowski sum of point configurations.  For simplicity, let us assume that if $\dim \operatorname{Conv}(A_i) = 1$ then $|\mathcal A_i| = 2$. Suppose that a weight $w: \mathcal A\to \RR$ induces a coherent subdivision $\Delta_w$ that is mixed.  Then there exist weights $w_1, \ldots, w_k$ on $\mathcal A_1, \ldots, \mathcal A_k$ such that $\Delta_{\sum_{i=1}^k w_i} = \Delta_w$.
\end{thm}


We prepare with the following observation.

\begin{lem}\label{lem:connected}
Let $Q$ be a $d$-dimensional polytope, and let $\{Q_1, \ldots, Q_m\}$ be the maximal (i.e.\ $d$-dimensional) faces of a polyhedral subdivision of $Q$.  The graph on $[m]$ with edges $(i,j)$ whenever $Q_i \cap Q_j$ has dimension $d-1$ is connected.  In particular, if $d \geq 2$, or if $d=1$ and the subdivision is trivial, then the maximal cells $Q_1, \ldots, Q_m$ are connected through dimension $\geq 1$.
\end{lem}

\begin{proof}
For any two vertices $i,j\in [m]$, pick points $p_i$ and $p_j$ in the interior of $Q_i$ and $Q_j$ (respectively).  Perturbing $p_i$ and $p_j$ if necessary, we have that the line segment $\overline{p_ip_j}$ meets faces of the polyhedral subdivision only of dimension $\geq d-1$.
\end{proof}



\begin{proof}[Proof of \Cref{thm:mixed}]
Let $\Delta_1, \ldots, \Delta_k$ be subdivisions of $\mathcal A_1, \ldots, \mathcal A_k$ (respectively) making up the mixed subdivision $\Delta_w$. 
For each $\overline\bu\in \RR^E/\RR\mathbf 1$, the face $\Delta_w^{\overline\bu}$ is a Minkowski sum $\sum_{i=1}^s \mathcal F_{i,\overline\bu}$ where $\mathcal F_{i,\overline\bu}$ is a face of $\Delta_i$.  For each $i = 1, \ldots, k$, consider the partition of $\RR^E$ by the equivalence relation $\overline\bu\sim_i \overline\bu' \iff \mathcal F_{i,\overline\bu} = \mathcal F_{i,\overline\bu'}$.  This partition consists of components whose closures define a polyhedral complex $\Sigma_i$ that coarsens the dual complex  $\Sigma_w$.  
We claim that each $\Sigma_i$ is a dual complex $\Sigma_{w_i}$ for some weight $w_i: \mathcal A_i \to \RR$ such that $\Delta_{w_i}^{\overline\bu} = \mathcal F_{i,\overline\bu}$ for all $\overline\bu\in \RR^E/\RR\mathbf 1$.  We are then done by \Cref{lem:Minkowskifaces}.

For the claim, fix $\overline\bu \in \RR^E/\RR\mathbf 1$ lying in a non-maximal cell of $\Sigma_{w}$. By \cite[Lemma 3.3.6]{MS15}, the polyhedral complex $\operatorname{star}_{\Sigma_{w}}(\overline\bu)$ is the normal fan of the polytope $\operatorname{Conv}(\Delta_{w}^{\overline\bu})=\sum_{i=1}^k \operatorname{Conv}(\mathcal F_{i,\overline \bu})$. Now fix any $1\leq i \leq k$.  By construction of $\Sigma_i$, the normal fan of $\operatorname{Conv}(\mathcal F_{i,\overline \bu})$ is equal to 
$\operatorname{star}_{\Sigma_i}(\overline\bu)$.
As $\operatorname{Conv}(\mathcal F_{i,\overline\bu})$ is a lattice polytope by our running integrality assumption on point configurations,
it follows that the union of non-maximal cells of $\Sigma_i$ is a rational, pure, balanced, polyhedral complex of codimension 1.  In other words, the complex $\Sigma_i$ satisfies the condition of \cite[Proposition 3.3.10]{MS15}, which states that there exists a weighted point configuration $\widetilde w_i: \widetilde{\mathcal A}_i \to \RR$ with $\Sigma_i = \Sigma_{\widetilde w_i}$.  

We now use $\widetilde w_i$ to define weights $w_i'$ on $\mathscr V_i$, where $\mathscr V_i = \operatorname{Vert}(\Delta_i)$ is the set of elements of $\mathcal A_i$ that appear as vertices of the subdivision $\Delta_i$.  This will have the property that the induced coherent subdivision satisfies $\operatorname{Conv}(\Delta_{w_i'}^{\overline \bu}) = \operatorname{Conv}(\mathcal F_{i,\overline\bu})$ for all $\overline\bu\in \RR^E/\RR\mathbf 1$, so that $w_i'$ naturally extends to a weight $w_i$ on $\mathcal A_i$ satisfying $\Delta_{w_i}^{\overline\bu} = \mathcal F_{i,\overline\bu}$ for all $\overline\bu\in \RR^E/\RR\mathbf 1$.  

By construction, the two polytopes $\operatorname{Conv}(\Delta_{\widetilde w_i}^{\overline\bu})$ and $\operatorname{Conv}(\mathcal F_{i,\overline\bu})$ are dilates of each other (up to translation) for every $\overline\bu\in \RR^E/\RR\mathbf 1$.  Since we assumed that $|\mathcal A_i| = 1$ if $\dim \operatorname{Conv}(\mathcal A_i) = 1$, by \Cref{lem:connected} the polyhedral subdivision from $\Delta_i$ is connected through dimension $\geq 1$.  Hence, the dilation factor is global; that is, (up to translation) the set $\mathscr V_i$ is a dilation of the set of vertices of $\Delta_{\widetilde w_i}$.  Assign the weight $w_i'$ on $\mathscr V_i$ via this dilation correspondence.
\end{proof}

\begin{rem}\label{rem:mixed}
Note that if $w$ was already a Minkowski sum $w_1'+\cdots+w_k'$, then the constructed weights $\{w_i\}_{1\leq i \leq k}$ in the proof satisfy $\Sigma_{w_i} = \Sigma_{w_i'}$ for all $1\leq i \leq k$.
\end{rem}


\subsection{M-convex functions and their dual complexes}
\label{subsec:mconvex}

We review M-convex functions, and establish \Cref{thm:Mconvexdualcplx} concerning the structure of their dual complexes.

\begin{defn}
A function $\mu: \ZZ^{[n]}\to \TT$ is \textbf{M-convex} if for $\mathbf a = (a_1, \ldots, a_n),\ {\mathbf{b} = (b_1, \ldots, b_n)} \in \ZZ^{[n]}$ and $i\in [n]$ such that $ a_i > b_i$, there exists $j\in [n]$ such that $ a_j <  b_j$ and
\begin{equation}\label{eqn:Mconvex}\tag{M}
\mu(\mathbf a) + \mu(\mathbf b) \geq \mu(\mathbf a - \be_i + \be_j) + \mu(\mathbf b-\be_j+\be_i).
\end{equation}
The set $\{\bv \in \ZZ^{[n]} \mid \mu( \bv) \neq \infty\}$ is the \textbf{effective domain} $\operatorname{dom}(\mu)$ of $\mu$, and is assumed to be finite.
\end{defn}

We view $\mu$ as a weighted point configuration $\mu: \operatorname{dom}(\mu) \to \RR$.  For M-convex functions $\mu_1$ and $\mu_2$, their Minkowski sum as weighted point configurations (not as functions) is denoted $\mu_1 + \mu_2$. 
%
M-convex functions are studied in several contexts. For instance, they are foundational objects of discrete convex analysis \cite{Mur03}.  We focus on their connection to generalized permutohedra.

\begin{defn} A lattice polytope $Q$ in $\RR^{[n]}$ is a \textbf{generalized permutohedron} if every edge of $Q$ is parallel to $\be_i - \be_j$ for some $i,j \in [n]$.
\end{defn}

The definition implies that a generalized permutohedron is homogeneous as a point configuration.

\medskip
Generalized permutohedra form a rich combinatorial class of lattice polytopes \cite{Edm70, Pos09, AA17}.  For example, base polytopes of matroids and flag matroids, which we discuss in \S\ref{subsection:dressians} and \S\ref{subsec:flagmatsubdiv}, are examples of generalized permutohedra \cite{GGMS87, BGW03}.  
Generalized permutohedra are related to M-convex functions in the following way.

\begin{thm}\label{thm:Mconvex}
Let $\mu: \ZZ^{[n]} \to {\TT}$ be a function with an effective domain $\operatorname{dom}(\mu)$.
\begin{enumerate}
\item If $\mu$ takes only two values $\{c,\infty\}$ for some $c\in \RR$, then $\mu$ is M-convex if and only if $\operatorname{dom}(\mu)$ is the set of lattice points of a generalized permutohedron.
\item More generally, $\mu$ is M-convex if and only if the subdivision $\Delta_\mu$ of $\operatorname{dom}(\mu)$ is tight and its faces are the sets of lattice points of generalized permutohedra.
\end{enumerate}
In particular, if $\mu$ is M-convex, the point configuration $\operatorname{dom}(\mu)$ is a generalized permutohedron, and hence is homogeneous.
\end{thm}

\begin{proof}
The first statement (1) is \cite[Theorem 4.15]{Mur03}.  For the second statement (2), we note the following observations.
\begin{itemize}
\item Let $\mu: \ZZ^{[n]} \to \TT$.  For any $\bu \in \RR^{[n]}$, the function $\mu(\cdot) + \langle \bu, \cdot \rangle: \ZZ^{[n]} \to \TT$ defined by $\bv \mapsto \mu(\bv) + \langle \bu, \bv \rangle$ is M-convex if and only if $\mu$ is.
\item Let $\mu: \ZZ^{[n]} \to \TT$ be an M-convex function.  Then the function defined by
$\bv \mapsto \min(\mu)$ if $\mu(\bv) = \min(\mu)$ and $\bv \mapsto 
\infty$ otherwise
is also M-convex.  In other words, by the first statement, the face $\Delta_\mu^{\mathbf 0}$ is the set of lattice points of a generalized permutohedron.
\end{itemize}
The second statement now follows from the first by applying \Cref{lem:linearshift} to these observations.
\end{proof}

Let us now turn to the dual complex $\Sigma_\mu$ of $\mu$.  Its polyhedral cells are subsets of $\RR^{[n]}/\RR\mathbf 1$.  Consider the closures of these polyhedral cells inside $\PP(\TT^{[n]})$.  For each nonempty proper subset $S\subsetneq [n]$, this defines a polyhedral complex structure on the boundary $T_S \subset \PP(\TT^{[n]})$.  While these polyhedral complex structures can be difficult to describe for general weighted point configurations, for M-convex functions we give an explicit description in \Cref{thm:Mconvexdualcplx}.  This explicit description will be instrumental in our proof of \Cref{thm:main2} and \Cref{thm:troplinsubdiv}.  We first note the following general boundary behavior.

\begin{lem}\label{lem:boundaryface}
Let $w$ be a weight on a point configuration $\mathcal A$ in $\RR^{[n]}$.  For a nonempty subset $S$, fix $\overline{\bu'}\in T_S$.  For a sufficiently small open neighborhood $U$ of $\overline{\bu'}$, one has $\Gamma_w(\mathcal A)^{(\overline \bu,1)} = \Gamma_{w}(\mathcal A^{\be_{[n]\setminus S}})^{(\overline \bu,1)}$ for any $\overline{\bu}\in U\cap T_{[n]}$. In other words, near $T_S$, the dual complex $\Sigma_w$ is the same as the dual complex of the restriction of $w$ to $\mathcal A^{\be_{[n]\setminus S}}$.
\end{lem}

\begin{proof}
Let $\bu = (u_i)_{i\in S} \times (u_j)_{j\notin S}$.  Shrinking $U$ if necessary, we can make $\min\{u_i - u_j \mid i\in S,\ j\notin S\}$ arbitrarily large.  Since $\mathcal A$ is finite and $w$ is fixed, this means that for $i\in \mathcal A$
to minimize $\langle {\bu}, \mathbf a_i \rangle +w(i)$, it must first minimize $\langle \be_{[n]\setminus S}, \mathbf a_i\rangle$.
\end{proof}

Next, we note that a property known as the \emph{Hopf monoid structure} of generalized permutohedra extends to M-convex functions.

\begin{notation}
 We need the following notations:  For a lattice polytope $Q\subset \RR^{[n]}$ and a nonempty subset $S\subseteq [n]$, the projection of the face $Q^{\be_{[n]\setminus S}}$ under $\RR^{[n]} \to \RR^S$ is denoted $Q|_S$, and the projection of $Q^{\be_{[n]\setminus S}}$ under $\RR^{[n]} \to \RR^{[n]\setminus S}$ is denoted $Q/_S$.  Both are lattice polytopes, and we write $Q|_S \times Q/_S \subset \RR^S \times \RR^{[n]\setminus S} \simeq \RR^{[n]}$ for their product, considered as a polytope in $\RR^{[n]}$.
 
 Our notation here differs from \cite{AA17} by a complementation ($\be_{[n]\setminus S}$ instead of $\be_S$).  Since $-\be_S$ and $\be_{[n]\setminus S}$ are equal as elements in $\RR^{[n]}/\RR\mathbf 1$, the difference is due to our "min" convention for polyhedral operations instead of the "max" convention used in \cite{AA17}.
\end{notation}

\begin{thm}\label{thm:AA17}\cite[Theorem 6.1]{AA17} Let $Q$ be a generalized permutohedron in $\RR^{[n]}$, and $S\subseteq [n]$ be a nonempty subset.  Then the polytopes $Q|_S$ and $Q/_S$ are generalized permutohedra in their respective spaces.  Moreover, 
$Q^{\be_{[n]\setminus S}} = Q|_S \times Q/_S$ and in particular is a generalized permutohedron.
\end{thm}

This property of generalized permutohedra extends to M-convex functions.  If $w_1, w_2$ are weights on $\mathcal A_1, \mathcal A_2$ in $\RR^{S_1}, \RR^{S_2}$ (respectively), let us write $w_1 \times w_2$ for the weight on $\mathcal A_1 \times \mathcal A_2$ in $\RR^{S_1}\times \RR^{S_2}$ defined by $w(i_1,i_2) := w(i_1) + w(i_2)$. 

\begin{lem}\label{lem:Mconvexhopf}
Let $\mu: \ZZ^{[n]}\to \TT$ be M-convex, and write $Q = \operatorname{dom}(\mu)$.  For a nonempty subset $S \subseteq [n]$, there exist weights $\mu|_S$ and $\mu/_S$ on $Q|_S$ and $Q/_S$ (respectively), each unique up to adding a constant globally, such that
\[
\mu^{\be_{[n]\setminus S}} = \mu|_S \times \mu/_S.
\]
The weighted point configurations $\mu^{\be_{[n]\setminus S}}$, $\mu|_S$, and $\mu/_S$ are M-convex.
\end{lem}

\begin{proof}
As $Q$ is a generalized permutohedron, we have $Q^{\be_{[n]\setminus S}} = Q|_S \times Q/_S$.
Thus, for the first statement, it suffices to show that for every choice of lattice points $p,p'\in Q|_S$ and $q,q'\in Q/_S$, one has $\mu(p,q) - \mu(p',q) = \mu(p,q') - \mu(p',q')$.  Moreover, as $Q|_S$ and $Q/ S$ are both generalized permutohedra, it suffices to check in the case where $p - p'= \be_i - \be_{i'}$ and $q - q' = \be_j- \be_{j'}$ where $i,i'\in S$ and $j,j'\in [n]\setminus S$.  Applying the defining property \eqref{eqn:Mconvex} of an M-convex function twice,  once with $(\mathbf a, \mathbf b) = ((p,q),(p',q'))$ and again with $(\mathbf a, \mathbf b) = ((p,q'),(p',q))$, gives the desired equality.

For the second statement, applying the forward direction of \Cref{thm:Mconvex}.(2) to $\mu$ implies that the face $Q^{\be_{[n]\setminus S}}$ is subdivided into generalized permutohedra, which implies that both $Q|_S$ and $Q/_S$ are too. (If one of them has an edge not parallel to $\be_i - \be_j$, so does the product).  The converse direction of \Cref{thm:Mconvex}.(2) then implies that $\mu^{\be_{[n]\setminus S}}$, $\mu|_S$, and $\mu/_S$ are M-convex.
\end{proof}

We are now ready to describe explicitly the closure of $\Sigma_\mu$ inside $\PP(\TT^{[n]})$.

\begin{thm}\label{thm:Mconvexdualcplx}
Let $\mu$ be an M-convex function, considered as a weighted point configuration in $\RR^{[n]}$.  For a cell $\sigma \subset \RR^{[n]}/\RR\mathbf 1$ of the dual complex $\Sigma_\mu$, denote by $\overline\sigma$ its closure in $\PP(\TT^{[n]})$. 
For a nonempty subset $S\subseteq [n]$, we have
$
\{\overline\sigma \cap T_S \mid \sigma\in\Sigma_\mu\} = \Sigma_{\mu|_S},
$
where $T_S$ is identified with $\RR^{S}/\RR\mathbf 1$.
\end{thm}
\begin{proof}
\Cref{lem:boundaryface} implies that $\{\overline\sigma \cap T_S \mid \sigma\in\Sigma_\mu\} = \{\overline\sigma \cap T_S \mid \sigma\in\Sigma_{\mu^{e_{[n]\setminus S}}}\}$. Applying \Cref{lem:Mconvexhopf} then gives the desired equality.
\end{proof}

\begin{notation}
Let $\mu$ be M-convex and $\overline\bu \in \PP(\TT^{[n]})$.  We denote
\[
\Delta_\mu^{\overline\bu} := \Delta_{\mu|_S}^{\overline {\bu'}},
\]
where $S\subseteq [n]$ is the subset satisfying $\overline\bu\in T_S$, so that $\bu = \bu'\times \infty^{[n]\setminus S}$ for some $\bu'\in \RR^{S}$.
\end{notation}

\begin{cor}\label{cor:Mconvexdualcplx}
Let $\mu$ be M-convex and $\emptyset\subsetneq S\subseteq [n]$.  The correspondence \eqref{eqn:dualcplx} for $T_S=\RR^S/\RR\mathbf 1$ gives
\[
\{ \overline\bu \in T_S \mid \Delta_{\mu}^{\overline\bu} \geq \mathcal F\} \longleftrightarrow \mathcal F \in \Delta_{\mu|_S}.
\]
This correspondence now extends to all of $\PP(\TT^{[n]})$:  the set $\PP(\TT^{[n]})$ is partitioned by the relative interiors $\{\overline\bu \in \PP(\TT^{[n]}) \mid \Delta_{\mu}^{\overline\bu} = \mathcal F\}$ as $\mathcal F$ ranges over all faces $\mathcal F$ of $\Delta_{\mu|_S}$ over all $\emptyset\subsetneq S \subseteq [n]$.
\end{cor}

\section{Dressians and projective tropical linear spaces}
\label{section:projtroplin}

We review Dressians and valuated matroids in \S\ref{subsection:dressians}.  Then, we introduce projective tropical linear spaces in \S\ref{subsection:projtroplin}, and prove \Cref{thm:main2}, which characterizes projective tropical linear spaces in many different ways.  We assume familiarity with matroids.  We point to \cite{Wel76, Oxl11} as references.

\begin{notation}
We adopt the following notations for a matroid $M$ on a ground set $[n]$:
\begin{itemize}
\item $\bb(M)$ is the set of bases, which we will often view as a point configuration $(\bb(M), \be)$, where $B\in \bb(M) \subset {\binom{[n]}{r}}$ maps to $\be_B\in \RR^{[n]}$,
\item $\cc(M)$ is the set of circuits.
\item $\operatorname{rk}_M: 2^{[n]} \to \ZZ$ is the rank function.
\item $Q(M) := \operatorname{Conv}(\be_B \mid B\in \bb(M)) \subset \RR^{[n]}$ the base polytope of $M$, which as a point configuration is identical to $(\mathcal B(M), \be)$ because $Q(M)$ has no non-vertex lattice points.
\item $M^*$ is the dual matroid of $M$.
\item $M|_S$ (resp.\ $M/_S$) is the restriction (resp.\ contraction) of $M$ to (resp.\ by) a subset $S\subseteq [n]$.
\end{itemize}
\end{notation}

As it is customary in matroid theory, we write $S\cup i$ to mean $S\cup \{i\}$ and $S\setminus i$ to mean $S\setminus \{i\}$ for a set $S$ and an element $i$.  We will often use the following.

\begin{thm}\cite{GGMS87}\label{thm:GGMS} A lattice polytope contained in the cube $\operatorname{Conv}(\be_S \mid \emptyset \subsetneq S \subseteq [n]) \subset \RR^{[n]}$ is a generalized permutohedron if and only if it is a base polytope of a matroid.
\end{thm}

\subsection{Dressians and valuated matroids}\label{subsection:dressians}

\medskip
We review Dressians and valuated matroids. As before, the underlying algebraic geometry is explained in the remarks.

\begin{defn}
For $0\leq r \leq n$, the \textbf{tropical Grassmann-Pl\"ucker relations} are tropical polynomials in variables $\{P_I \mid I \in \binom{[n]}{r}\}$ defined as
\begin{equation}\label{eqn:tropGP}\tag{GP}
\mathscr {P}_{r;n}^{\trop} := \left\{ (P_I\odot P_J) \oplus \bigoplus_{j\in J\setminus I} (P_{I\setminus i\cup j} \odot P_{J\setminus j \cup i}) \ \middle | \ I,J\in \binom{[n]}{r},\ |I\cap J| < r-1,\  i \in I\setminus J\right\}.
\end{equation}
The \textbf{Dressian (of rank $r$ in $[n]$)} is the projective tropical prevariety of these tropical Grassmann-Pl\"ucker relations.  That is, we define 
\[
Dr(r;n) := \overline\trop(\mathscr P_{r;n}^{\trop}) \subset \PP(\TT^{\binom{[n]}{r}}).
\]
\end{defn}

Points on Dressians were previously described in several ways \cite{speyer,MS15,tropplanes}; we collect them together in \Cref{thm:Dressians}.  Let us first recall the definition of valuated matroids from \cite{dress}.

\begin{defn}
Let $M$ be a matroid of rank $r$ on $[n]$.  A \textbf{valuated matroid} with underlying matroid $M$ is a function $\mu: \mathcal B(M) \to \RR$ such that for every $B,B'\in \mathcal B(M)$ and $i\in B\setminus B'$ there exists $j\in B'\setminus B$ satisfying
\[
\mu(B) + \mu(B') \geq \mu(B \setminus i \cup j) + \mu(B' \setminus j \cup i).
\]
\end{defn}

\begin{thm}\label{thm:Dressians}
Let $\mu\in \TT^{\binom{[n]}{r}}$.  Then the following are equivalent:
\begin{enumerate}[label = (\alph*)]
\item  The image $\overline\mu \in \PP(\TT^{\binom{[n]}{r}})$ is a point of $Dr(r;n)$.
\item $\mu$ is a valuated matroid with an underlying matroid of rank $r$ on $[n]$.\item When $\mu$ is regarded as a weight on $\{\be_I \in \RR^{[n]} \mid \mu(I) \neq \infty\}$, the faces of $\Delta_\mu$ are base polytopes of matroids.
\end{enumerate}
\end{thm}

\begin{proof}
Let us consider $\mu\in\TT^{\binom{[n]}{r}}$ as a function $\mu: \ZZ^{[n]} \to \TT$ where
\[
\bv \mapsto \begin{cases} \mu(I) & \textnormal{if } \bv = \be_I \textnormal{ for some }I\in \binom{[n]}{r}\\
\infty & \textnormal{otherwise}.
\end{cases}
\]
One can check from the definitions that $\mu$ is M-convex if and only if the image $\overline\mu \in \PP(\TT^{\binom{[n]}{r}})$ lies in $Dr(r;n)$.  The equivalence of (a) and (b) now follows by comparing the definitions of M-convexity and valuated matroids.  The equivalence of (b) and (c) follows from \Cref{thm:Mconvex}.(2) and \Cref{thm:GGMS}.
\end{proof}

For a valuated matroid $\mu$ with underlying matroid $M$, we will freely switch between considering $\mu$ as a point on $\TT^{\binom{[n]}{r}}$, as an M-convex function with effective domain $Q(M)$, and as a weight on the point configuration $(\mathcal B(M),\be)$.

\begin{defn}
Recall the notation $T_S := \RR^S/\RR\mathbf 1 \times \{\infty\}^{E\setminus S} \subset \PP(\TT^E)$ for sets $S\subseteq E$.  For $M$ a matroid of rank $r$ on $[n]$, the \textbf{Dressian of $M$}, denoted $Dr(M)$, is the intersection
\[
Dr(M) = Dr(r;n) \cap T_{\mathcal B(M)} \subset \PP\big(\TT^{\binom{[n]}{r}}\big).
\]
\end{defn}

By \Cref{thm:Dressians}, the set $Dr(M)$, which was introduced in \cite{tropplanes}, parametrizes valuated matroids with underlying matroid $M$, or equivalently, weights $\mathcal B(M)\to \RR$ that induce coherent subdivisions of $Q(M)$ into base polytopes of matroids.  By \Cref{rem:usualtrop}, the set $Dr(M)$ is the usual tropical prevariety in $\RR^{\mathcal B(M)}/\RR\mathbf 1$ of appropriately modified tropical Grassmann-Pl\"ucker relations.

\medskip
Many aspects of matroids extend to valuated matroids.  We will use the following notions.

\begin{defn}\label{defn:valuatedcircuits}
Let $\mu$ be a valuated matroid of rank $r$ on $[n]$ with underlying matroid $M$.  
\begin{itemize}
\item For each $S\in \binom{[n]}{r+1}$, define an element $\mathbf C_\mu(S) \in \TT^{[n]}$ by
\[
C_\mu(S)_i :=
\begin{cases} \mu(S\setminus i) & i\in S\\
\infty & i\notin S.
\end{cases}
\]
Then the set of \textbf{valuated circuits} of $\mu$ is defined as
\[
\mathcal C(\mu) := \left\{ \mathbf C_\mu(S) \ \middle | \ S \in \textstyle \binom{[n]}{r+1}\right\} \setminus \{(\infty, \ldots, \infty)\}.
\]
\item The \textbf{dual} of $\mu$ is the valuated matroid $\mu^*$ defined by setting $\mu^*([n]\setminus I) := \mu(I)$ for $I\in \binom{[n]}{r}$.
\item The \textbf{valuated cocircuits} of $\mu$ are defined as the circuits of $\mu^*$.  Explicitly, the set of valuated cocircuits is
\[
\mathcal C^*(\mu) = \left\{\mathbf{C}_\mu^*(S) \ \middle | \ S\in \textstyle \binom{[n]}{r-1}\right\} \setminus \{(\infty, \ldots, \infty)\},
\]
where
\[
C_\mu^*(S)_i:=
\begin{cases} \mu({S\cup i}) & i\not\in S\\
\infty & i\in S.
\end{cases}
\]
\item For a nonempty subset $S\subseteq [n]$, the \textbf{restriction to $S$} (resp.\ \textbf{contraction by $S$}) \textbf{of} $\mu$ is $\mu|_S$ (resp.\ $\mu/_S$), where $\mu|_S$ and $\mu/_S$ are as in \Cref{lem:Mconvexhopf}. \Cref{lem:Mconvexhopf}, combined with \Cref{thm:GGMS}, implies that these are valuated matroids.
\end{itemize}
\end{defn}

The following facts are easy to verify:
\begin{itemize}
\item The set $\{\operatorname{supp}(\mathbf C) \mid \mathbf C \in \mathcal C(\mu)\}$ is the set of circuits of $M$. (Recall the notation $\operatorname{supp}(\mathbf C) := \{i\in [n] \mid C_i \neq \infty\}$ for $\mathbf C \in \TT^{[n]}$).
\item The underlying matroid of $\mu^*$ is $M^*$, and $(\mu^*)^* = \mu$.
\item The underlying matroid of $\mu|_S$ (resp.\ $\mu/_S$) is $M|_S$ (resp.\ $M/_S$).
\end{itemize}

Lastly, we will need the following description of the valuated circuits in terms of restrictions and contractions.  It is a consequence of \cite[Theorem 3.29 \& Corollary 4.10.(1)]{BB19}.

\begin{thm}\label{thm:circuitsvalmat}
Let $\mu$ be a valuated matroid of rank $r$ on $[n]$, and $S\subseteq [n]$ a nonempty subset.  Then, we have
\[
\begin{split}
\cc(\mu|_S) &= \{\mathbf C \in \cc(\mu) \mid \operatorname{supp}(\mathbf C) \subseteq S\}, \textnormal{ and }\\
\cc(\mu/_S) &= \{\textnormal{points of }\mathbb T^{[n]\setminus S}\textnormal{ of minimal support among }\{ (C_i)_{i\in [n]\setminus S} \mid \mathbf C\in \cc(M)\}\}.
\end{split}
\]
Moreover, the two operations are dual to each other, in the sense that $(\mu/_S)^* = \mu^*|_{([n]\setminus S)}$.
\end{thm}

The underlying geometry behind the definition of Dressians follows.

\begin{rem}\label{rem:geomDr}
See \cite[\S9]{Ful97} for Pl\"ucker embeddings, and see Remarks \ref{rem:projtrophyper} and \ref{rem:projtrop} for tropicalizations of projective subvarieties.  The Grassmannian $Gr(r;n)$, whose points are $r$-dimensional subspaces of $\kk^{[n]}$, is embedded in $\PP(\kk^{\binom{[n]}{r}})$ by the Pl\"ucker embedding.  When $\operatorname{char}\kk = 0$, the defining ideal 
is generated by the Grassmann-Pl\"ucker relations:
\begin{equation}\label{eqn:plucker1}
\mathscr P_{r;n} := \left\{\left.-P_IP_J + \sum_{j\in J\setminus I} \operatorname{sign}(i,j,I,J) P_{I\setminus i \cup j}P_{J\setminus j \cup i} =0 \ \right| \ I,J \in \binom{[n]}{r},\ i\in I\setminus J\right\}
\end{equation}
where $\operatorname{sign}(i,j,I,J) := (-1)^{\#\{a\in I \ | \ \min(i,j)<a<\max(i,j)\}+\#\{b\in J \ | \ \min(i,j)<b<\max(i,j)\}}$.
When $\operatorname{char}\kk >0$ they generate the ideal up to radical.  The tropical Grassmann-Pl\"ucker relations are tropicalizations of these polynomials.  That is, we have $\mathscr P_{r;n}^{\trop} = \{ f^{\trop} \mid f\in \mathscr P_{r;n}\}$.  The projective tropicalization $\overline\trop(Gr(r;n))$ is thus a subset of $Dr(r;n)$.

The inclusion $\overline\trop(Gr(r;n)) \subseteq Dr(r;n)$ is often strict, precisely because not all valuated matroids are realizable in the following sense:
For a linear subspace $L\in Gr(r;n)$, let $(P_I(L))_{I\in \binom{[n]}{r}}\in \PP(\kk^{\binom{[n]}{r}})$ be its coordinates in the Pl\"ucker embedding.  Then the function $I\mapsto \operatorname{val}(P_I(L)) \ \forall I\in \binom{[n]}{r}$ is a valuated matroid, denoted $\mu(L)$, whose underlying matroid is denoted $M(L)$.  Valuated matroids arising in this way are said to be \textbf{realizable (over $\kk$)}.  The points of $\overline\trop(Gr(r;n))$ are exactly the valuated matroids realizable over $\kk$.  When $r\geq 3$ and $n\geq 7$, there are non-realizable valuated matroids, and hence, in these cases the inclusion $\overline\trop(Gr(r;n)) \subseteq Dr(r;n)$ is strict. Realizability can fail in many ways. For example, there are valuated matroids where every cell of the induced subdivision is a realizable matroid, but but the valuated matroid is not realizable \cite{speyerthesis}.
\end{rem}

\subsection{The many faces of projective tropical linear spaces}\label{subsection:projtroplin}

We introduce projective tropical linear spaces, and prove \Cref{thm:main2}, which characterizes them in many different ways.  We start by reviewing usual tropical linear spaces.

\begin{prop}\cite[Lemma 4.4.7]{MS15} \label{prop:troplin}
Let $\mu$ be a valuated matroid on $[n]$.  The following two subsets of $\RR^{[n]}/\RR\mathbf 1$ coincide:
\begin{enumerate}
\item The set
\[
\bigcap_{\mathbf C \in \mathcal C(\mu)} \left\{ \overline\bu \in \RR^{[n]}/\RR\mathbf 1 \ \middle | \ \textnormal{the minimum in } \{C_i + u_i\}_{i\in \operatorname{supp}(\mathbf C)} \textnormal{ is achieved at least twice}\right\},
\]
which is the usual tropical prevariety of the valuated circuits of $\mu$ (\Cref{rem:usualtrop}).
\item With $\mu$ regarded as a weighted point configuration, the set
\[
\{ \overline\bu \in \RR^{[n]}/\RR\mathbf 1 \mid \Delta_{\mu}^{-\overline\bu} \textnormal{ is a base polytope of a loopless matroid}\}.
\]
which is the union of "loopless cells" of the dual complex $\Sigma_{\mu}$ in $\RR^{[n]}/\RR\mathbf 1$.
\end{enumerate}
\end{prop}

\begin{defn}
Let $\mu$ be a valuated matroid of rank $r$ on $[n]$, and let $M$ be its underlying matroid.  The subset of $\RR^{[n]}/\RR\mathbf 1$ in the previous proposition is defined as the \textbf{tropical linear space of $\mu$}, denoted $\trop(\mu)$.  Note that if $M$ has loops, then $\trop(\mu) = \emptyset$.
\end{defn}

We will extend \Cref{prop:troplin} to projective tropical linear spaces.  Since projective tropical linear spaces are subsets of $\PP(\TT^{[n]})$, the negative sign $-\overline\bu$ in \Cref{prop:troplin}.(2) can be problematic because $-\infty$ is not an element of $\TT$.  We will thus use the following reformulation:

\begin{samepage}
\begin{lem}\label{lem:coloopless}
Let $\mu$ be a valuated matroid, and $\mu^*$ its dual.  Then we have
\[
\trop(\mu) = \{\overline\bu \in \RR^{[n]}/\RR\mathbf 1 \mid \Delta_{\mu^*}^{\overline\bu} \textnormal{ is a base polytope of a coloopless matroid}\},
\]
which is the union of "coloopless cells" of the dual complex $\Sigma_{\mu^*}$.
\end{lem}
\end{samepage}

\begin{proof}
For a weight $w$ on a point configuration $(\mathcal A, \mathbf a)$, let us write $w^{op}$ for the weight on the point configuration $(\mathcal A, -\mathbf a)$, defined by $w^{op}(i) := w(i) \ \forall i\in \mathcal A$.  It is easy to verify that $\Delta_w^{-\overline\bu} = \Delta_{w^{op}}^{\overline\bu}$ as subsets of $\mathcal A$.
Now, if $M$ is the underlying matroid of $\mu$, then $Q(M^*) = -Q(M) + \mathbf 1$, so that $\mu^* = \mu^{op}$.  The lemma now follows from the description of $\trop(\mu)$ in \Cref{prop:troplin}.(2), since a matroid is loopless if and only if its dual matroid is coloopless.
\end{proof}


The following remark explains the geometry behind tropical linear spaces via tropicalizations of subvarieties (see \Cref{rem:projtrop}).  It also motivates our definition of projective tropical linear spaces.

\begin{rem}\label{geomtroplin}
Recall from \Cref{rem:geomDr} that a linear subspace $L\subset \kk^{[n]}$ defines a valuated matroid $\mu(L)$.  Let us consider $L$ as a linear projective subvariety of $\PP(\kk^{[n]})$, and write $\mathring L := L \cap (\kk^*)^{[n]}/\kk^*$.  Then the usual tropicalization $\trop(\mathring L)$ of $\mathring L$ is the tropical linear space $\trop(\mu(L))$.  When $L$ is contained in a coordinate hyperplane, or equivalently, when the matroid $M(L)$ has a loop, the intersection $\mathring L$ is empty, and hence $\trop(\mathring L)$ is empty, as is $\trop(\mu(L))$.  See \cite[\S4.3]{MS15} for a more details on tropicalizations of linear subvarieties in a torus $(\kk^*)^{[n]}/\kk^*$.

Now consider the projective tropicalization $\overline\trop(L)$.  For each nonempty subset $S\subseteq [n]$, the torus orbit $O_S$ intersects $L$ to give another (possibly empty) linear subvariety of $(\kk^*)^S/\kk^*$, denoted $\mathring L_S$.  Similarly, let $L_S := L \cap \overline{O_S}$, considered as a subvariety of $\PP(\kk^S)$.  Then the valuated matroid $\mu(L_S)$ is the contraction $\mu(L)/_{([n]\setminus S)}$.  We thus have $\overline\trop(L) \cap T_S = \trop(\mathring L_S) = \trop(\mu(L)/_{([n]\setminus S)})$.  This motivates our definition of projective tropical linear spaces.
\end{rem}

\begin{defn}\label{defn:troplin}
Let $\mu$ be a valuated matroid on $[n]$.  The \textbf{projective tropical linear space} $\overline\trop(\mu)$ of $\mu$ is a subset of $\PP(\TT^{[n]})$ defined by setting
\[
\overline{\trop}(\mu) \cap T_{[n]\setminus S} := \trop(\mu/_S)\times \{\infty\}^S
\]
for each $\emptyset \subseteq S \subsetneq [n]$.
\end{defn}

Projective tropical linear spaces have previously appeared in various forms (see \Cref{rem:previoustroplin}).
\Cref{thm:main2}, reproduced below, unifies them and adds two new characterizations (\ref{thm2r:pt3} and \ref{thm2r:pt5}).

\newtheorem*{thm:main2}{\Cref{thm:main2}}
\begin{thm:main2}
Let $\mu$ be a valuated matroid on a ground set $[n]$.  Let $\ell\subseteq [n]$ be the set of loops of its underlying matroid $M$.  The following sets in the tropical projective space $\PP(\TT^E)$ coincide:
\begin{enumerate}[label=(\roman*)]
\item \label{thm2r:pt1} The {projective tropical linear space} of $\mu$, i.e.\
\[
\overline{\trop}(\mu) := \displaystyle \bigcup_{\emptyset\subseteq S \subsetneq [n]} \Big( \trop(\mu/_S)\times \{\infty\}^S \Big) \subset \PP(\TT^E),
\]

\item \label{thm2r:pt2}The projective tropical prevariety of the valuated circuits of $\mu$, i.e.\
\[
\bigcap_{\tiny \begin{matrix} \textnormal{valuated}\\ \textnormal{circuits $\mathbf C$}\end{matrix}} \Big\{ \overline\bu \in \PP(\TT^{[n]}) \ \left| \ \textnormal{the minimum is achieved at least twice among } \{C_i + v_i\}_{i\in [n]}\Big\}\right.,
\]

\item \label{thm2r:pt3} The union of coloopless cells of the closure of the dual complex of $\mu^*$ in $\PP(\TT^{[n]})$, i.e.\
\[
\begin{split}
\Big\{\overline\bu\in \PP(\TT^{[n]}) \mid \Delta_{\mu^*}^{\overline\bu} \textnormal{ is a base polytope of a coloopless matroid}\Big\},
\end{split}
\]

\item \label{thm2r:pt4} The tropical span of the valuated cocircuits of $\mu$, i.e.\
\[
\left\{\left. \begin{matrix} \textnormal{the image in $\PP(\TT^E)$ of} \\
(a_1\odot \mathbf{C_1^*}) \oplus \cdots \oplus (a_l \odot \mathbf{C_m^*}) \in \TT^E
\end{matrix}
\ \right|\  \begin{matrix}
\mathbf{C_i^*}\in \TT^E \textnormal{ a valuated cocircuit of } \mu,\\
\ a_i \in \RR, \quad \forall 1\leq i \leq m
\end{matrix}
\right\},
\]

\item \label{thm2r:pt5}The closure of $\trop(\mu/_\ell) \times \{\infty\}^\ell$ inside $\PP(\TT^{[n]})$.
\end{enumerate}
\end{thm:main2}

\begin{rem}\label{rem:previoustroplin}  For ordinary matroids (not valuated), the description \ref{thm2r:pt1} appeared in \cite[Definition 2.20]{Sha13}.  The authors of \cite{MR} also considered the description \ref{thm2r:pt1}, and characterized $\overline{\trop}(\mu)$ as a tropical cycle of projective degree 1 \cite[Remark 7.4.15]{MR}.  In the language of hyperfields (see \cite{BB19}), the description \ref{thm2r:pt2} says that a projective tropical linear space is the set of covectors of a matroid over the tropical hyperfield.  This characterization appeared in \cite{MT01}, along with the proof of \ref{thm2r:pt2}$=$\ref{thm2r:pt4} \cite[Theorem 3.8]{MT01}, and was generalized to perfect tracts in \cite{And19}. 
\end{rem}

\begin{proof}[Proof of \Cref{thm:main2}] The equality \ref{thm2r:pt2} $=$ \ref{thm2r:pt4} is \cite[Theorem 3.8]{MT01}.  Recalling from \Cref{thm:circuitsvalmat} that the dual of $\mu/_S$ is $\mu^*|_{([n]\setminus S)}$, combining \Cref{thm:Mconvexdualcplx} with \Cref{lem:coloopless} then  implies \ref{thm2r:pt1} $=$ \ref{thm2r:pt3}.
We now show \ref{thm2r:pt3} $\subseteq$ \ref{thm2r:pt5} $\subseteq$ \ref{thm2r:pt2} $\subseteq$ \ref{thm2r:pt1}.

For all subsets $S\subseteq [n]$ such that $S\not\supset\ell$, the matroid $M/S$ has loops, and so the intersection of the set \ref{thm2r:pt3} with $T_{[n]\setminus S}$ is empty (since \ref{thm2r:pt3} $=$ \ref{thm2r:pt1}).  The same is true for the set \ref{thm2r:pt5}.  Hence, for showing \ref{thm2r:pt3} $\subseteq$ \ref{thm2r:pt5} we may assume that $M$ is loopless.  In this case, both sets are $\trop(\mu)$ on $T_{[n]}$ by \Cref{lem:coloopless}.  Now, suppose $\bu \times \{\infty\}^{[n]\setminus S} \in T_{S}$ is in the set \ref{thm2r:pt3}.  We need to show that it is in the closure of $\trop(\mu)$.  Since $M$ is loopless, so is $M|_{[n]\setminus S}$, and hence $\trop(\mu|_{[n]\setminus S})$ is nonempty.  Let $\overline{\bu'} \in \trop(\mu|_{[n]\setminus S})$ and pick its representative $\bu'\in \RR^{[n]\setminus S}$ to have all positive coordinates.  For a point $\bu \times c\bu' \in \RR^{[n]}$, if $c>0$ is sufficiently high (equivalently, if $\bu \times c\bu'$ is in a small enough open neighborhood of $\bu \times \{\infty\}^{[n]\setminus S}$), \Cref{lem:boundaryface} implies that $\Delta_{\mu^*}^{\bu \times c\bu'} = \Delta_{w}^{\bu \times c\bu'}$, where $w = (\mu^*)^{\be_{[n]\setminus S}}$.  Then by \Cref{lem:Mconvexhopf}, we have $w = \mu^*|_S \times \mu^*/_S$, so that $\Delta_{w}^{\bu \times c\bu'} = \Delta_{\mu^*|_S}^{\bu} \times \Delta_{\mu^*/_S}^{c\bu'}$.  By assumption the matroids of $\Delta_{\mu^*|_S}^{\bu}$ and $\Delta_{\mu^*/_S}^{c\bu'}$ are both coloopless.  We thus conclude that $\bu\times c\bu'$ is in $\trop(\mu)$ for all sufficiently large $c>0$, and hence the point $\bu \times \{\infty\}^{[n]\setminus S}$ is in the closure of $\trop(\mu)$.

For \ref{thm2r:pt5} $\subseteq$ \ref{thm2r:pt2}, we may again assume $M$ loopless, since the fact that a loop is a circuit implies that \ref{thm2r:pt2} is contained in the closure of $T_{[n]\setminus \ell}$.  In this case, both sets are $\trop(\mu)$ on $T_{[n]}$ by \Cref{prop:troplin}.(1).  Since projective tropical prevarieties are closed, we thus have \ref{thm2r:pt5} $\subseteq$ \ref{thm2r:pt2}.

Lastly, for any proper subset $\emptyset \subseteq S \subsetneq [n]$, consider the intersection of the set \ref{thm2r:pt2} with $T_{[n]\setminus S}$.  In other words, for each valuated circuit $\mathbf C$ defining a tropical polynomial $\bigoplus_{i\in [n]} C_i \odot x_i$, we give ignore all $C_i$ with $i\in S$ since $x_i = \infty$.  Thus, the description of the valuated circuits of the contraction $\mu/_S$ in \Cref{thm:circuitsvalmat}, combined with \Cref{prop:troplin}, imply \ref{thm2r:pt2} $\subseteq$ \ref{thm2r:pt1}.
\end{proof}

\section{Valuated flag matroids and flag Dressians}
\label{section:morph}

We now introduce flag Dressians and valuated flag matroids, and prove \Cref{thm:main}.  We review flag matroids in \S \ref{subsec:flagmat}.  In \S \ref{subsection:matquot} we define flag Dressians and valuated flag matroids, and show \ref{thm:pt-tropIP}$\iff$\ref{thm:pt-valflag}, which is mostly definitional (\Cref{prop:FlDr=valflagmat}).  In \S \ref{subsec:flagtroplin}, we prove \ref{thm:pt-valflag}$\iff$\ref{thm:pt-troplin} (\Cref{thm:tropquotient}).  In \S \ref{subsec:flagmatsubdiv}, we define flag matroidal subdivisions and prove \ref{thm:pt-valflag}$\implies$\ref{thm:pt-subdiv} (\Cref{thm:valflagsubdiv}) and \ref{thm:pt-subdiv}$\implies$\ref{thm:pt-troplin} (\Cref{thm:troplinsubdiv}). We give an extended illustration of \cref{thm:main} in \Cref{eg:pointline}.

\subsection{Flag matroids}
\label{subsec:flagmat}

Flag matroids are defined through matroid quotients.

\begin{defn}
\label{defn:matroidquotient}
Let $M$ and $N$ be matroids on a common ground set $[n]$.  We say that $M$ is a \textbf{(matroid) quotient} of $N$, denoted $M \twoheadleftarrow N$, if any of the following equivalent conditions are satisfied \cite[Proposition 7.4.7]{Bry86}:
\begin{enumerate}
\item For all $A\subseteq B \subseteq [n]$, we have $\operatorname{rk}_{M}(B) -\operatorname{rk}_{M}(A) \leq \operatorname{rk}_N(B) - \operatorname{rk}_{N}(A) $,
\item each circuit of $N$ is a union of circuits of $M$,
\item there exist a matroid $\widetilde M$ on $[n]\sqcup [n']$ such that $M = \widetilde M/_{[n']}$ and $N = \widetilde M\setminus_{[n']}$,
\item $N^*$ is a quotient of ${M}^*$.
\end{enumerate}
A sequence $\MM = (M_1, \ldots, M_k)$ of matroids on $[n]$ is a \textbf{flag matroid} if $M_i \twoheadleftarrow M_j$ for every $1\leq i<j \leq k$.  The \textbf{rank} of $\MM$ is the sequence of its constituent matroids $(\operatorname{rk}(M_1), \ldots, \operatorname{rk}(M_k))$.
\end{defn}

The following example gives the geometric origin of the terminology.

\begin{eg}[Realizable quotients and flag matroids]
\label{eg:kquot}
Let ${L'}^* \twoheadleftarrow L^* \twoheadleftarrow \kk^{[n]}$ be quotients of linear spaces.  Equivalently, we have an inclusion of linear subspaces $L' \subseteq L \subseteq \kk^{[n]}$.  Then, the matroids of $L'$ and $L$, which we denote $M(L')$ and $M(L)$ (\Cref{rem:geomDr}), form a matroid quotient $M(L') \twoheadleftarrow M(L)$.  Matroid quotients arising in this way are said to be \textbf{realizable (over $\kk$)}.  Similarly, a flag of linear subspaces $\boldsymbol L = L_1 \subseteq L_2 \subseteq \cdots \subseteq L_k \subseteq \kk^{[n]}$ defines a flag matroid $\MM(\boldsymbol L) = (M(L_1), \ldots, M(L_k))$.  Such flag matroids are \textbf{realizable (over $\kk$)}.
\end{eg}


\begin{rem}\label{rem:nonreal}
A quotient $M \twoheadleftarrow N$ can fail to be realizable even if $M$ and $N$ are realizable over the field.  For a concrete example, see \cite[\S1.7.5.\ Example\ 7]{BGW03}.
\end{rem}

\begin{defn}\label{defn:baseflagmat}
Given a flag matroid $\MM = (M_1, \ldots, M_k)$ on $[n]$, its \textbf{base configuration} $\bb(\MM)$ is a point configuration obtained as the Minkowski sum of the bases of its constituent matroids.  That is, $\bb(\MM) := \bb(M_1) + \cdots +\bb(M_k) = (\bb(M_1) \times \cdots \times \bb(M_k), \be)$, where
\[
(B_1, \ldots, B_k)\in \bb(M_1) \times \cdots \times \bb(M_k)\mapsto \be_{B_1} + \cdots +\be_{B_k} \in \RR^{[n]}.
\]
The \textbf{base polytope} $Q(\MM)$ of $\MM$ is the convex hull of the the image of the base configuration, i.e.\
\[
Q(\MM) := \operatorname{Conv}(\be_{B_1} + \cdots + \be_{B_k} \mid (B_1,\ldots,B_s) \in \bb(M_1)\times \cdots \times \bb(M_k)) \subset \mathbb{R}^{[n]}.
\]
\end{defn}

Properties of matroid polytopes found in \Cref{thm:GGMS} extend to flag matroid base polytopes.

\begin{thm}\label{thm:GGMSmore}\
\begin{enumerate}
\item \cite[Theorem 1.11.1]{BGW03} A lattice polytope $Q\subset \RR^{[n]}$ is a base polytope of a flag matroid of rank $(r_1, \ldots, r_k)$ if and only if it is a generalized permutohedron and its vertices are a subset of the orbit of $\be_{\{1,2,\ldots, r_1\}} + \cdots + \be_{\{1,2,\ldots, r_k\}}$ under the permutation group $S_n$.
\item For a flag matroid $\MM = (M_1, \ldots, M_k)$ on $[n]$, and a subset $S\subseteq [n]$, the sequences $\MM|_S := (M_1|_S, \ldots, M_k|_S)$ and $\MM/_S := (M_1/_S, \ldots, M_k/_S)$ are flag matroids, and the face $Q(\MM)^{\be_{[n]\setminus S}}$ is the product $Q(\MM|_S)\times Q(\MM/_S)$.
\end{enumerate}
\end{thm}

\begin{proof}
The first part of statement (2) is checked directly from the description of matroid quotients by rank functions.  The second part of (2) follows by \Cref{lem:Minkowskifaces} and \Cref{thm:AA17}.
\end{proof}

\begin{rem}\label{rem:looplessbad}
Restricting to only loopless matroids is harmless in studying matroids because the only data lost by deleting the loops of a matroid is the number of loops:  if $\ell$ is the set of loops of a matroid $M$, then $M = M\setminus_{\ell} \oplus U_{0,\ell}$, so one easily recovers $M$ from $M\setminus\ell$ and $|\ell|$.
However, for a flag matroid $\MM = (M_1, \ldots, M_k)$ on $[n]$, an element $e\in [n]$ can be a loop in some but not all of the matroids $M_1, \ldots, M_k$, and in such cases one cannot always recover $\MM$ from $\MM\setminus_e = (M_1 \setminus_e, \ldots, M_k \setminus_e)$ and $\MM|_e = (M_1|_e, \ldots, M_k|_e)$. So, it is necessary for us to develop the theory for matroids with loops in the flag setting.
\end{rem}

\begin{rem}
According to \cite{GS87a, BGW03}, flag matroids are exactly the Coxeter matroids of type $A$.  Coxeter matroids in general are defined by modifying  \Cref{thm:GGMSmore}.(1) with the notion of \emph{Coxeter generalized permutohedra}.  See \cite{ACEP20} for a modern treatment of Coxeter generalized permutohedra and their connection to combinatorics.
\end{rem}

\subsection{Definition of flag Dressians and valuated flag matroids}
\label{subsection:matquot}

We now extend Dressians and valuated matroids, described in \Cref{subsection:dressians}, to the setting of flag matroids.  

\begin{defn}
\label{defn:flagDressian}
Let $0\leq r \leq s \leq n$.  The \textbf{tropical incidence-Pl\"ucker relations} are tropical polynomials in variables $\{P_{I} \mid I\in \binom{[n]}{r}\} \cup \{P_J \mid J \in \binom{[n]}{s} \}$ defined as
\begin{equation}\label{eqn:tropIP}\tag{IP}
\mathscr P_{r,s;n}^{\trop} = \left\{ \bigoplus_{j' \in J'\setminus I'} P_{I'\cup j'} \odot P_{J' \setminus j'} \ \middle | \ I' \in \binom{[n]}{r-1},\ J' \in \binom{[n]}{s+1}\right\}.
\end{equation}
When $r = s$, the sets $\{P_{I} \mid I\in \binom{[n]}{r}\}$ and $\{P_J \mid J \in \binom{[n]}{s} \}$ coincide, and the relations $\mathscr P_{r,s;n}^{\trop}$ in \eqref{eqn:tropIP} above degenerate to $\mathscr P_{r;n}^{\trop}$ in \eqref{eqn:tropGP}.
These tropical polynomials are multi-homogeneous with respect to the partition $\binom{[n]}{r}\sqcup \binom{[n]}{s}$.  For $0\leq r_1\leq \cdots \leq r_k \leq n$, the \textbf{flag Dressian} (of rank $(r_1, \ldots, r_k)$ on $[n]$) is the multi-projective tropical prevariety inside $\PP\big(\TT^{\binom{[n]}{r_1}}\big)\times \cdots \times \PP\big(\TT^{\binom{[n]}{r_k}}\big)$ defined by the tropical Grassmann-Pl\"ucker relations \eqref{eqn:tropGP} and the tropical incidence-Pl\"ucker relations \eqref{eqn:tropIP}:
\[
FlDr(r_1, \ldots, r_k) := \overline\trop\left (F \in \left\{\mathscr P_{r_i;n}^{\trop}\right\}_{1\leq i \leq k} \cup \left\{\mathscr P_{r_i,r_j;n}^{\trop}\right\}_{1\leq i< j \leq n}\right).
\]
\end{defn}

We interpret the tropical incidence-Pl\"ucker relations as a condition for \emph{valuated matroid quotients}, and points on the flag Dressian as \emph{valuated flag matroids}, defined as follows.

\begin{defn}
\label{defn:valflagmat}
Let $\mu$ and $\nu$ be valuated matroids on a common ground set $[n]$, whose underlying matroids are $M$ and $N$ of ranks $r$ and $s$ (respectively) with $r\leq s$.  We say that $\mu$ is a \textbf{valuated (matroid) quotient of $\nu$}, denoted $\mu\twoheadleftarrow \nu$, if for any $I\in \mathcal B(M)$, $J\in \mathcal B(N)$, and $i\in I\setminus J$, there exists $j\in J\setminus I$ such that
\[
\mu(I) + \nu(J) \geq \mu(I \setminus i \cup j) + \nu(J\setminus j \cup i).
\]
A sequence $\boldsymbol \mu = (\mu_1, \ldots, \mu_k)$ of valuated matroids on $[n]$ is a \textbf{valuated flag matroid} if $\mu_i \twoheadleftarrow \mu_j$ for every $1\leq i < j \leq k$.
It follows from the definition that $\mu\twoheadleftarrow \nu$ if and only if $\nu^* \twoheadleftarrow \mu^*$.
\end{defn}

\medskip
We will show that the underlying matroids of a valuated matroid quotients form a matroid quotient (\Cref{cor:firstprop}).  Thus, for a valuated flag matroid $\boldsymbol{\mu} = (\mu_1, \ldots, \mu_k)$, its sequence of underlying matroids $(M_1, \ldots, M_k)$ is called the underlying flag matroid of $\boldsymbol\mu$.

\medskip
We first note that points of the flag Dressian correspond to valuated flag matroids.  The following is the equivalence \ref{thm:pt-tropIP}$\iff$\ref{thm:pt-valflag} in \Cref{thm:main}.

\begin{prop}\label{prop:FlDr=valflagmat}
Let $\mu \times \nu$ be a point on $\TT^{\binom{[n]}{r}} \times \TT^{\binom{[n]}{s}}$.  Its image $\overline\mu\times \overline\nu \in \PP(\TT^{\binom{[n]}{r}}) \times \PP(\TT^{\binom{[n]}{s}})$ is a point on the flag Dressian $FlDr(r,s;n)$ if and only if $\mu$ and $\nu$ are valuated matroids that form a valuated matroid quotient $\mu\twoheadleftarrow \nu$.  In other words, the points on the flag Dressian $FlDr(r_1, \ldots, r_k;n)$ correspond to valuated flag matroids of rank $(r_1, \ldots, r_k)$ on $[n]$.
\end{prop}

\begin{proof}
Each of $\mu$ and $\nu$ satisfies its respective tropical Grassmann-Pl\"ucker relations if and only if it is a valuated matroid by \Cref{thm:Dressians}.  Now, note that the tropical incidence-relation
\[
\bigoplus_{j' \in J'\setminus I'} P_{I'\cup j'} \odot P_{J' \setminus j'}
\]
for $I'\in \binom{[n]}{r-1},\ J'\in \binom{[n]}{s+1}$
can be rewritten as follows:  Fix any $i\in J'\setminus I'$, and set $I = I'\cup i$ and $J = J'\setminus i$.  Then, the above tropical polynomial is the same as
\[
P_I\odot P_J \oplus \bigoplus_{j \in J\setminus I} P_{I\setminus i \cup j} \odot P_{J\setminus j \cup i}. 
\]
The condition that the minimum (if achieved) is achieved by at least two terms of these tropical polynomials is equivalent to the condition imposed by the inequalities in the definition of valuated matroid quotients.
\end{proof}

\begin{defn}
Recall the notation $T_S := \RR^S/\RR\mathbf 1 \times \{\infty\}^{E\setminus S} \subset \PP(\TT^E)$ for sets $S\subseteq E$.
Let $\MM = (M_1, \ldots, M_k)$ be a flag matroid of rank $(r_1, \ldots, r_k)$ on $[n]$.  The \textbf{flag Dressian of $\MM$}, denoted $FlDr(\MM)$, is the intersection 
\[
FlDr(\MM) := FlDr(r_1, \ldots, r_k;n) \cap \big( T_{\mathcal B(M_1)} \times \cdots \times T_{\mathcal B(M_k)} \big) \subset \PP\big(\TT^{\binom{[n]}{r_1}}\big)\times \cdots \times \PP\big(\TT^{\binom{[n]}{r_k}}\big).
\]
\end{defn}

In other words, by \Cref{prop:FlDr=valflagmat} the flag Dressian $FlDr(\MM)$ parametrizes all valuated flag matroids whose underlying flag matroid is $\MM$.

\begin{rem}\label{rem:plucker}
For linear subspaces $K$ and $L$ of $\kk^{[n]}$ of rank $r$ and $s$, let $(p_I)_{I\in \binom{[n]}{r}}$ and $(p_J)_{I\in \binom{[n]}{s}}$ be their Pl\"ucker coordinates (respectively).  Then $K\subseteq L$ if and only if the two Pl\"ucker coordinates satisfy the \textbf{incidence-Pl\"ucker relations} \cite[\S 9, Lemma 2]{Ful97}:
\begin{equation}\label{eqn:plucker2}
\mathscr P_{r,s; [n]} = \left\{ \sum_{j'\in J'\setminus I'} \operatorname{sign}(j';I',J') P_{I'\cup j'}P_{J'\setminus j'} \ \middle| \ I'\in \binom{[n]}{r-1}, \ J' \in \binom{[n]}{s+1} \right\}
\end{equation}
where $\operatorname{sign}(j';I',J') = (-1)^{\#\{a\in I' \mid a< j'\} + \#\{b\in J' \mid b<j'\}}$.
The tropical incidence-Pl\"ucker relations are tropicalizations of these polynomials.  That is, we have $\mathscr P_{r,s;n}^{\trop} = \{f^{\trop} \mid f\in \mathscr P_{r,s;n}\}$.  Thus, if $K\subseteq L$, then the corresponding valuated matroids $\mu(K)$ and $\mu(L)$ form a valuated matroid quotient $\mu(K) \twoheadleftarrow \mu(L)$.  Valuated matroid quotients arising in this way are said to be \textbf{realizable (over $\kk$)}.
\end{rem}

\begin{rem}\label{rem:geomFlDr}
The (partial) flag variety $Fl(r_1, \ldots, r_k;n)$ consists of flags of linear subspaces $L_1 \subseteq \cdots \subseteq L_k \subseteq \kk^{[n]}$ with $\dim_\kk L_i = r_i$.   It has an embedding by $Fl(r_1, \ldots, r_k;n) \hookrightarrow Gr(r_1;n) \times \cdots \times Gr(r_k;n) \hookrightarrow \PP(\kk^{\binom{[n]}{r_1}}) \times \cdots \times \PP(\kk^{\binom{[n]}{r_1}})$.  When $\operatorname{char}\kk =0$, the Grassmann-Pl\"ucker relations \eqref{eqn:plucker1} combined with the incidence-Pl\"ucker relations \eqref{eqn:plucker2} generate the multi-homogeneous ideal of this embedding.  When $\operatorname{char}\kk >0$, they generate the ideal up to radical.  The multi-projective tropicalization of the flag variety $\overline{\trop}(Fl(r_1, \ldots, r_k;n))$ is thus a subset of $FlDr(r_1, \ldots, r_k;n)$.  The points of $\overline{\trop}(Fl(r_1, \ldots, r_k;n))$ correspond to valuated flag matroids that are \textbf{realizable (over $\kk$)}.  The inclusion of $\overline{\trop}(Fl(r_1, \ldots, r_k;n))$ in $FlDr(r_1, \ldots, r_k;n)$ is strict precisely when there are valuated flag matroids of rank $(r_1, \ldots, r_k)$ on $[n]$ that are not realizable (over $\kk$).
\end{rem}

\begin{rem}
A valuated matroid quotient $\mu\twoheadleftarrow \nu$, where $\mu$ and $\nu$ are realizable over $\kk$ and the underlying matroid quotient is realizable over $\kk$, can fail to be realizable over $\kk$.  See \Cref{eg:m4u26}.
\end{rem}

We address realizability of valuated flag matroids in more depth in \S \ref{sec:realizability}.

\subsection{Flags of projective tropical linear spaces}
\label{subsec:flagtroplin}

We show that a valuated matroid quotient is equivalent to the inclusion of the corresponding projective tropical linear spaces.  This proves the equivalence \ref{thm:pt-valflag}$\iff$\ref{thm:pt-troplin} in \Cref{thm:main}.

\begin{thm}
\label{thm:tropquotient}
Let $\mu$ and $\nu$ be valuated matroids of ranks $r$ and $s$ respectively on a common ground set $[n]$. Then $\mu\twoheadleftarrow \nu$ if and only if $\overline{\trop}(\mu)\subseteq \overline{\trop}(\nu)$.  In other words, a sequence $\boldsymbol \mu = (\mu_1, \ldots, \mu_k)$ of valuated matroids is a valuated flag matroid if and only if $\overline\trop(\mu_1)\subseteq \cdots \subseteq \overline\trop(\mu_k)$.
\end{thm}

\begin{proof}
By \Cref{thm:main2}.\ref{thm2:pt4}, the projective tropical linear space $\overline{\trop}(\mu)$ is the tropical span of its valuated cocircuits $\mathcal C^*(\mu)$. Hence, we have $\overline{\trop}(\mu)\subset \overline{\trop}(\nu)$ if and only if $\cc^*(\mu) \subset \overline \trop(\nu)$.

By \Cref{thm:main2}.\ref{thm2:pt2}, the projective tropical linear space $\overline{\trop}(\nu)$ is cut out by its valuated circuits:
\[\overline{\trop}(\nu):=\bigcap_{\mathbf C\in \cc(\nu)} \overline{\trop} \left(\bigoplus_{j'\in [n]} C_{j'} \odot x_{j'}\right).\]
The description of valuated circuits and cocircuits (\Cref{defn:valuatedcircuits}) implies that $\mathcal C^*(\mu) \subseteq \overline\trop(\nu)$ if and only if the minimum in 
\[ \left\{C_{\nu}(J')_{j'} + C^*_\mu(I')_{j'} \right\}_{{j'}\in [n]} = \left\{\nu(J'\setminus {j'}) \odot \mu(I'\cup {j'}) \right\}_{{j'}\in [n]}
\]
is attained at least twice for each $I'\in \binom{[n]}{r-1}$ and $J'\in \binom{[n]}{s+1}$.  Removing the terms that are $\infty$ on the right hand side, this is the same as saying that the minimum in
\[
\left\{\mu(I'\cup j')\odot \nu(J'\setminus j') \right\}_{j'\in J'\setminus I'}
\]
is attained at least twice for every $I'$ and $J'$, which is exactly the condition defined by the tropical incidence-Pl\"ucker relations \eqref{eqn:tropIP}.
\end{proof}

This proof closely mirrors that of \cite[Theorem 1]{Haq12}, where \Cref{thm:tropquotient} was proved for loopless matroids.
We list some properties of valuated matroid quotients that follow from \Cref{thm:tropquotient}.

\begin{cor}\label{cor:firstprop}
Let $\mu, \nu, \xi$ be valuated matroids on $[n]$.
\begin{enumerate}
\item A composition of valuated matroid quotients is a valuated matroid quotient.  That is, if $\mu\twoheadrightarrow \nu$ and $\nu\twoheadrightarrow \xi$ then $\mu\twoheadrightarrow \xi$.
\item If $\mu\twoheadleftarrow \nu$, then $\mu/_S \twoheadleftarrow \nu/_S$ and $\mu\setminus_S \twoheadleftarrow \nu\setminus_S$ for any subset $S \subseteq [n]$.
\item For any $S\subset [n]$, we have $\mu/_S \twoheadleftarrow \mu\setminus _S$. 
\item The underlying matroids $M$ and $N$ of $\mu$ and $\nu$ (respectively) satisfy $M\twoheadleftarrow N$ if $\mu\twoheadleftarrow \nu$.
\end{enumerate}
\end{cor}

\begin{proof}
(1) is immediate from \Cref{thm:tropquotient}.  For (2), the two statements are duals of each other since $\mu\twoheadleftarrow \nu$ if and only if $\nu^* \twoheadleftarrow \mu^*$, so we only need show $\mu/S\twoheadrightarrow \nu/S$, which follows from \Cref{thm:tropquotient} and \Cref{thm:main2}.\ref{thm2:pt1}.  For (3), note that $\mathcal C^*(\mu/_S) = \mathcal C(\mu^*\setminus_S) \subseteq \mathcal C(\mu^*/_S) = \mathcal C^*(\mu\setminus_S)$ by \Cref{thm:circuitsvalmat}, and then combine \Cref{thm:tropquotient} with \Cref{thm:main2}.\ref{thm2:pt4}.  For (4), again from \Cref{thm:tropquotient} and \Cref{thm:main2}.\ref{thm2:pt4}, we have that the the valuated cocircuits $\mathcal C^*(\mu)$ are in the tropical span of $\mathcal C^*(\nu)$.  Considering the supports of these as elements in $\TT^{[n]}$, we have that every cocircuit of $M$ is a union of cocircuits of $N$.  Hence, we have $N^* \twoheadleftarrow M^*$, or equivalently, $M\twoheadleftarrow N$.
\end{proof}

\begin{rem}
The results here about valuated matroid quotients generalize to \textbf{matroid morphisms} (see \cite{EH20}) in the following way.  For a map of finite sets $\varphi: [n] \to [m]$ and $M$ a matroid of rank $r$ on $[m]$, define a matroid $\varphi^{-1}M$ on $[n]$ by $\bb(\varphi^{-1}M) := \{B\in \binom{[n]}{r} \mid \varphi(B) \in \bb(M)\}$.  If $\mu$ is a valuated matroid on $M$, then $\varphi^{-1}\mu: B \mapsto \mu(\varphi(B))$ is a valuated matroid on $\varphi^{-1}M$.  
A \textbf{valuated matroid morphism} $\nu\to \mu$ consists of valuated matroids $\nu, \mu$ on $[n], [m]$ (respectively) and a map $\varphi: [n]\to [m]$ such that $\nu \twoheadrightarrow \varphi^{-1}(\mu)$.  By \Cref{thm:tropquotient}, this is equivalent to saying that the following diagram commutes
\begin{equation}\label{eqn:morphism}\tag{$\ddagger$}
\begin{tikzcd}
\overline{\trop}(\mu) \arrow[r] \arrow[d, hook] &\overline{\trop}(\nu) \arrow[d, hook]\\
\PP(\TT^{[m]}) \arrow[r, "\trop_\varphi"] &\PP(\TT^{[n]}),
\end{tikzcd}
\end{equation}
where $\trop_{\varphi}: \PP(\TT^{[m]}) \to \PP(\TT^{[n]})$ is given by $(u_{j})_{j\in [m]} \mapsto (u_{\varphi(i)})_{i\in [n]}$.  The diagram \eqref{eqn:morphism} mirrors the diagram defining realizable matroid morphisms \cite[Remark 2.1]{EH20}.  The map $\overline{\trop}(\mu) \to \overline{\trop}(\nu)$ in \eqref{eqn:morphism} is a tropical morphism in the sense of \cite{Mik06}.
\end{rem}

\subsection{Flag matroidal subdivisions of base polytopes}
\label{subsec:flagmatsubdiv}

We now come to subdivisions of base configurations of flag matroids. Let $\MM$ be a flag matroid.

\begin{defn}
 A subdivision of the base configuration $\bb(\MM)$ is \textbf{flag matroidal} if each face of the subdivision is a base configuration of a flag matroid.  A subdivision of the base polytope $Q(\MM)$ is \textbf{flag matroidal} if each face of the subdivision is a base polytope of a flag matroid.  
\end{defn}

A flag matroidal subdivision of $\bb(\MM)$ is necessarily mixed, and gives a flag matroidal subdivision of $Q(\MM)$.  In general one cannot recover the subdivision of a point configuration $\mathcal A$ from the resulting polyhedral subdivision of $\operatorname{Conv}(\mathcal A)$.  But in this case, since one can recover the constituent matroids $M_1, \ldots, M_k$ of a flag matroid $\MM$ from the data of its base polytope $Q( \MM)$ alone, a flag matroidal subdivision of the base polytope $Q(\MM)$ determines a subdivision of the base configuration $\bb(\MM)$. We now prove \ref{thm:pt-valflag}$\implies$\ref{thm:pt-subdiv} in Theorem \ref{thm:main}.


\begin{thm}\label{thm:valflagsubdiv}
Let $\boldsymbol\mu=(\mu_1,\dots, \mu_k)$ be a valuated flag matroid with underlying flag matroid $\MM = (M_1, \ldots, M_k)$.  Regard each $\mu_i$ as a weighted point configuration on $\bb(M_i)$.  Then, their Minkowski sum $\sum_{i=1}^k \mu_i$, which is a weight on the base configuration $\bb(\MM)$, induces a flag matroidal subdivision of $\bb(\MM)$.
\end{thm}


\begin{proof}
We begin by making the following observations for valuated matroids $\mu$ and $\nu$ with underlying matroids $M$ and $N$ (respectively).  Both observations are straightforward to verify.
\begin{itemize}
\item Fix $\bu\in \RR^{[n]}$, and consider a new weight $\mu'$ defined by
\[
\mu'(I) := \mu(I) + \langle \bu, \be_I \rangle \quad \textnormal{for }I\in \bb(M),
\]
and similarly define $\nu'$.  Then $\mu\twoheadleftarrow \nu$ if and only if $\mu' \twoheadleftarrow \nu'$.
\item Let $r$ be the rank of $M$.  Define a new weight $\mu_{min}: \binom{[n]}{r} \to \RR$ by $\mu_{min}(I) = \min(\mu)$ if $\mu(I) = \min(\mu)$ and $\infty$ otherwise, which is a (valuated) matroid also known as the \emph{initial matroid} of $\mu$ (see \cite[Definition 4.2.7]{MS15} and \cite[Definition 2.1]{BS20}).
Similarly define $\nu_{min}$.  Then $\mu\twoheadleftarrow \nu$ implies $\mu_{min}\twoheadleftarrow \nu_{min}$.  In particular, by \Cref{cor:firstprop}.(4), the underlying matroids of $\mu_{min}$ and $\nu_{min}$ form a matroid quotient if $\mu\twoheadleftarrow \nu$.
\end{itemize}
Now, fix an arbitrary $\bu\in \RR^{[n]}$ and consider the face $\Delta_{\sum_{i=1}^k \mu_i}^{\overline\bu}$ of the subdivision of $\bb(\MM)$.  By \Cref{lem:Minkowskifaces}, this face is the Minkowski sum $\sum_{i=1}^k \Delta_{\mu_i}^{\overline\bu}$.  By \Cref{lem:linearshift} and the first observation above, we may assume that $\bu = \mathbf 0$.  In this case, each face $\Delta_{\mu_i}^{\mathbf 0}$ is the underlying matroid of $(\mu_i)_{min}$, and the second observation thus implies that the faces form a flag matroid.
\end{proof}

The following theorem proves \ref{thm:pt-subdiv}$\implies$\ref{thm:pt-troplin} in Theorem \ref{thm:main}. The characterization in \Cref{thm:main2}.\ref{thm2:pt3} plays a fundamental role here.

\begin{thm}
\label{thm:troplinsubdiv}
Let $\MM = ({M_1}, \ldots, {M_k})$ be a flag matroid on $[n]$, and let $w$ be a weight on the base configuration $\bb(\underline{\mathcal M})$.  Suppose the coherent subdivision $\Delta_w$ is flag matroidal, and let $\mu_1, \ldots, \mu_k$ be any weights on $\bb(M_1), \ldots, \bb(M_k)$ satisfying $\Delta_w = \Delta_{\sum_{i=1}^k \mu_i}$.  Then $\mu_1, \ldots, \mu_k$ are valuated matroids, and they satisfy $\overline\trop(\mu_1) \subseteq \cdots \subseteq \overline\trop(\mu_k)$.
\end{thm}


\begin{lem}\label{lem:propagate}
Let $\MM = (M_1, \ldots, M_k)$ be a flag matroid on $[n]$.  If $M_i$ is loopless, then so are $M_j$ for any $1\leq i < j \leq k$.  Equivalently, if $M_i^*$ is coloopless, then so are $M_j^*$ for any $1\leq i < j \leq k$.
\end{lem}

\begin{proof}
The definition of matroid quotients by rank functions implies that if an element $l\in [n]$ satisfies $\operatorname{rk}_{M_j}(l) = 0$, then $\operatorname{rk}_{M_i}(l) = 0$ also for all $1\leq i<j\leq k$.
\end{proof}

\begin{proof}[Proof of \Cref{thm:troplinsubdiv}]
By \Cref{lem:Minkowskifaces}, for every $\bu\in \RR^{[n]}/\RR\mathbf 1$ the face $\Delta_w^{\overline\bu} = \Delta_{\sum_{i=1}^k \mu_i}^{\overline\bu}$ is the Minkowski sum $\sum_{i=1}^s \Delta_{\mu_i}^{\overline\bu}$.  Since $\Delta_w$ is flag matroidal, in particular each face of $\Delta_w$ is a Minkowski sum of base polytopes of matroids.  In other words, for each $1\leq i\leq k$ the face $\Delta_{\mu_i}^{\overline\bu}$ is a base polytope of a matroid.  Thus, each $\Delta_{\mu_i}$ is a subdivision of $\bb({M_i})$ whose faces are all also matroids.  Thus, by the equivalence of (a) and (c) in \Cref{thm:Dressians}, each $\mu_i$ is a valuated matroid.

We will apply \Cref{thm:main2}.\ref{thm2:pt3} to prove the rest of the theorem.  In preparation, we first note that for a matroid $M$, one has $Q(M^*) = -Q(M) + \mathbf 1$. Hence, if $\mu$ is a valuated matroid, then the map $\Delta_\mu \to \Delta_{\mu^*}$ defined by $\mathcal F \mapsto -\mathcal F + \mathbf 1$ is a bijection.  Therefore, the duals $\mu_1^*, \ldots, \mu_k^*$ induce a flag matroidal subdivision $\Delta_{\sum_{i=1}^k \mu_i^*}$ of the flag matroid base polytope of $(M_k^*, \dots, M_1^*)$, because its faces are in bijection with the faces of $\Delta_w$ by $\mathcal  F \mapsto -\mathcal F + k\mathbf 1$.  

Now, let $\overline\bu \in \PP(\TT^{[n]})$, and let $S\subseteq [n]$ be the subset such that $\overline\bu \in T_S$.  Write $\bu= \bu' \times \infty^{[n]\setminus S}$.
Combining \Cref{lem:boundaryface} and \Cref{lem:Mconvexhopf} implies that for some $\bu''\in \RR^{[n]\setminus S}$, we have $\Delta_{\mu_i^*}^{\bu'\times \bu''} = \Delta_{\mu_i^*|_S}^{\bu'} \times \Delta_{\mu_i^*/_S}^{\bu''}$, and thus \Cref{thm:GGMSmore}.(2) implies that $\Delta_{\sum_{i=1}^k \mu_i^*|_{S}}$ is a flag matroidal subdivision.  \Cref{thm:Mconvexdualcplx} therefore implies that the sequence of faces $(\Delta_{\mu_1^*}^{\overline\bu}, \ldots, \Delta_{\mu_k^*}^{\overline\bu})$ form the dual of a flag matroid, that is, $(\Delta_{\mu_k^*}^{\overline\bu}, \ldots, \Delta_{\mu_1^*}^{\overline\bu})$ is a flag matroid.  By \Cref{lem:propagate}, if the matroid of $\Delta_{\mu_i^*}^{\overline\bu}$ is coloopless for some $1\leq i \leq k$, then so are $\Delta_{\mu_j^*}^{\overline\bu}$ for any $1\leq i<j\leq k$.  The desired inclusion $\overline\trop(\mu_i) \subseteq \overline\trop(\mu_j)$ for all $1\leq i < j \leq k$ now follows from \Cref{thm:main2}.\ref{thm2:pt3}, which states that $\overline{\trop}(\mu_i) = \{\overline{\bu} \in \PP(\TT^{[n]}) \mid \textnormal{matroid of } \Delta_{\mu_i^*}^{\overline{\bu}} \textnormal{ is coloopless}\}$.  
\end{proof}

We have now proven \Cref{thm:main}.  The equivalence \ref{thm:pt-tropIP}$\iff$\ref{thm:pt-subdiv} states that a mixed coherent subdivision of a base configuration of a flag matroid is flag matroidal if and only if the weights form a valuated flag matroid.  One can further ask whether all coherent flag matroidal subdivisions arise in this way.
Combining Theorems \ref{thm:mixed} and \ref{thm:troplinsubdiv} implies the following.

\begin{cor}\label{cor:coherentflagmixed}
Every coherent flag matroidal subdivision of a base polytope of a flag matroid arises from a valuated flag matroid.
\end{cor}

We now feature an extended illustration of \Cref{thm:main}.

\begin{eg}
\label{eg:pointline} Consider the tropical prevariety of $Fl(1,2;4)$, the (closure of) the flag Dressian of the flag matroid $\boldsymbol{U}_{1,2;4}:=(U_{1,4},U_{2,4})$. Compare this example to \cite[Example 4.3.19]{MS15}.

We embed the variety $Fl(1,2;4)$ inside $\mathbb{P}^5\times\mathbb{P}^3$, where the first factor has Pl\"{u}cker coordinates $P_{ij}$ while the second factor has Pl\"{u}cker coordinates $P_{i}$ for $i,j = 1,\ldots,4$ and $i < j$. The equations defining $FlDr(1,2;4)$ in this embedding are given by
\begin{align*}
  \langle & P_{14}P_{23} - P_{13}P_{24} + P_{12}P_{34},\ \ \ 
  P_{4}P_{23} - P_{3}P_{24} + P_{2}P_{34},\ \ \  P_{4}P_{13} - P_{3}P_{14} + P_{1}P_{34},\\
  & P_{4}P_{12} - P_{2}P_{14} + P_{1}P_{24},\ \ \  P_{3}P_{12} - P_{2}P_{13} + P_{1}P_{23} \rangle.  
\end{align*}
We compute the tropical prevariety defined by these equations to obtain (the affine cone of) the flag Dressian $FlDr(\boldsymbol{U}_{1,2;4})$ using the command \texttt{tropicalintersection} in the software \texttt{gfan} \cite{gfan}.  In its 10 dimensional ambient space, the affine cone of $FlDr(\boldsymbol{U}_{1,2;4})$ is a pure simplicial fan of dimension 7 with a 5 dimensional lineality space. Modulo its lineality space, it consists of 10 rays and 15 two-dimensional cones. Intersected with the sphere, we obtain Figure~\ref{fig:df124}.  This is also $Dr(2;5)$, in agreement with $\widehat{FlDr}(1,2;4) = \widehat{Dr}(2;5)$ as we will see in \Cref{cor:affineconesame}.

\begin{figure}[h]
\centering
\begin{minipage}{.5\textwidth}
  \centering
\includegraphics[height=1.5in]{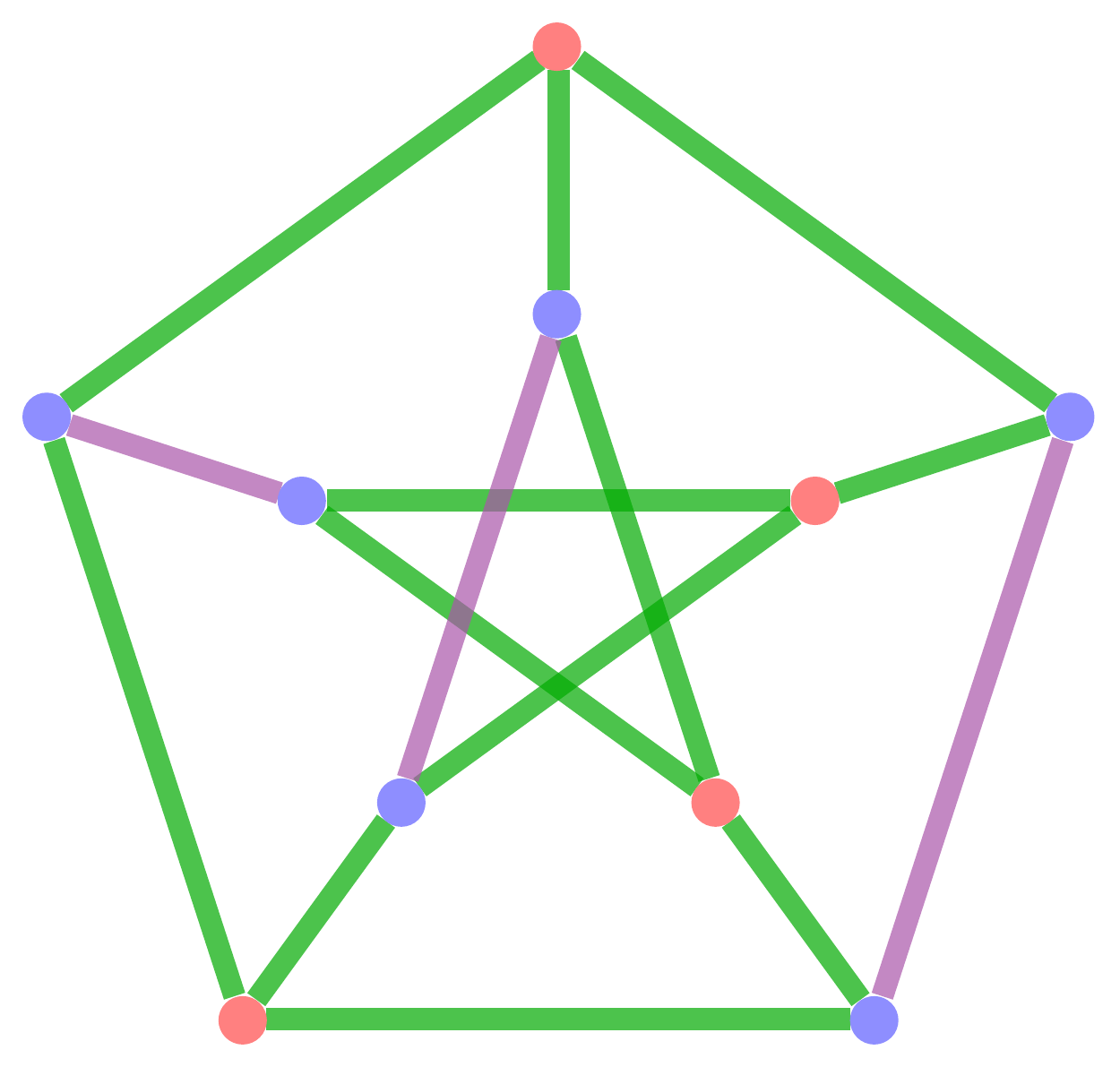}
  \captionof{figure}{The flag Dressian $FlDr(\UU_{1,2;4})$. 
  }
  \label{fig:df124}
\end{minipage}%
\begin{minipage}{.5\textwidth}
  \centering
      \includegraphics[height = 1.5 in]{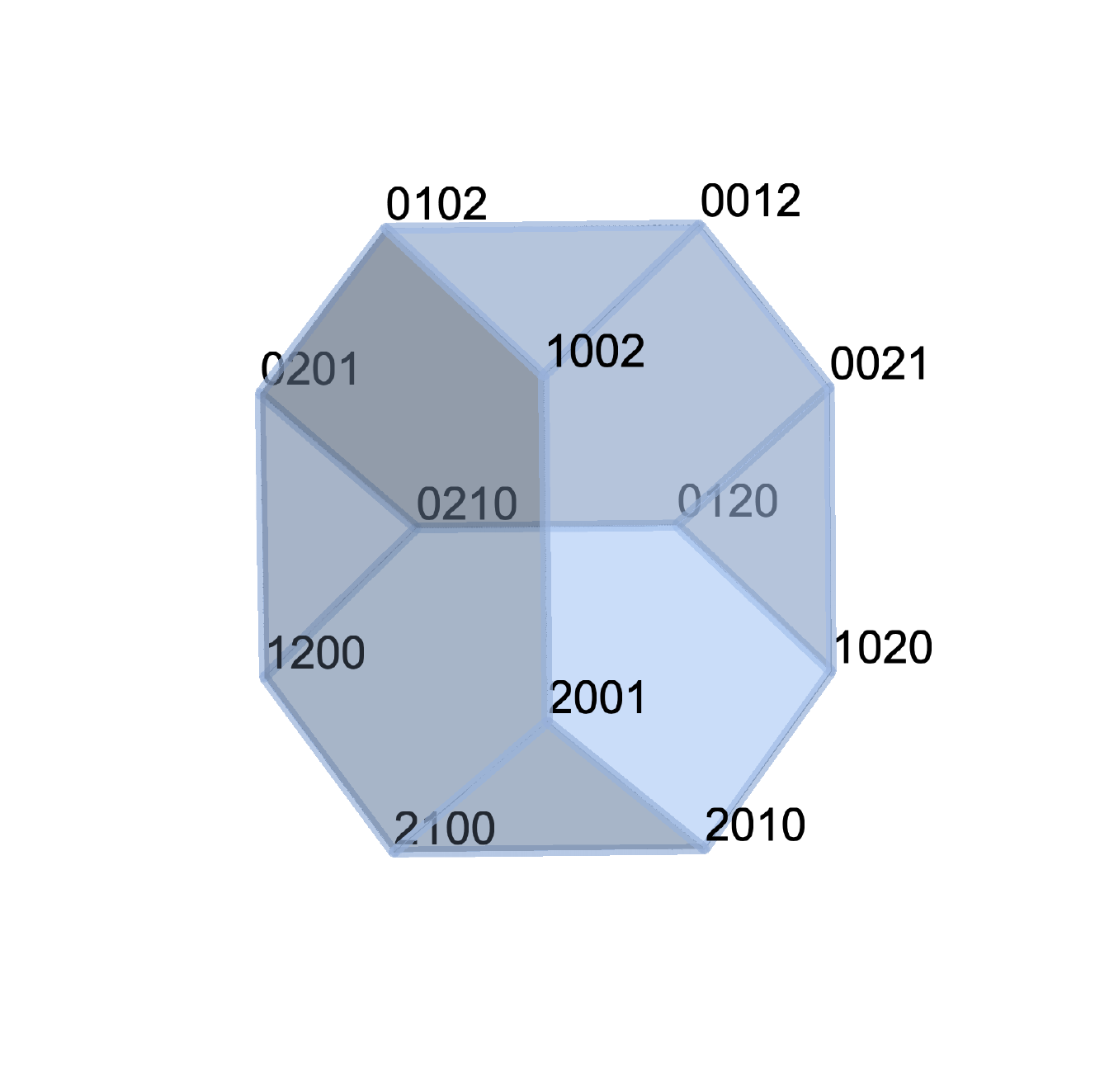}
  \captionof{figure}{The base polytope of the flag matroid $(\boldsymbol{U}_{1,2;4})$.}
  \label{fig:df124_poly}
\end{minipage}
\end{figure}

The base polytope $Q(\boldsymbol{U}_{1,2;4})$ is the \emph{truncated tetrahedron}, pictured in Figure~\ref{fig:df124_poly}. This is the orbit polytope
$
\text{Conv} \left \{
g \cdot (1,2,0,0) \subset \mathbb{R}^4 \ :\ g \in S_4
\right \}.
$
The subdivisions induced on $Q(\boldsymbol{U}_{1,2;4})$ by points in $FlDr(\boldsymbol{U}_{1,2;4})$ come in five types, as indicated in Figure \ref{fig:df124_subs}. These correspond to the colored edges and vertices in Figure \ref{fig:df124}. By Theorem \ref{thm:main} these are the subdivisions of $Q(\boldsymbol{U}_{1,2;4})$ into flag matroid polytopes. We display each subdivision with the corresponding flag of tropical linear spaces.

\begin{center}
\begin{figure}[h]
\begin{tabular}{c c c c c}
\includegraphics[width = 0.17 \linewidth]{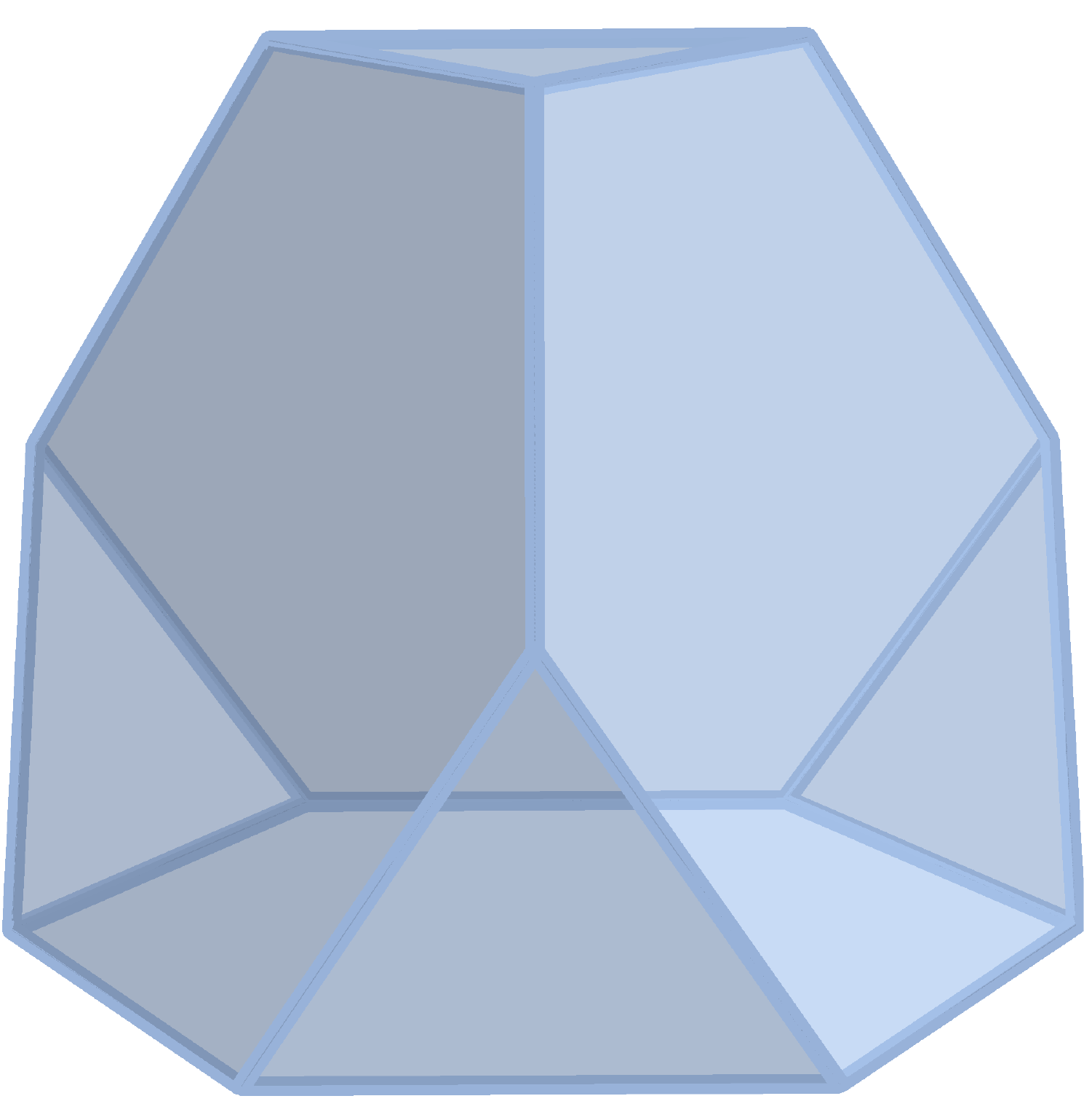} &
\includegraphics[width = 0.17 \linewidth]{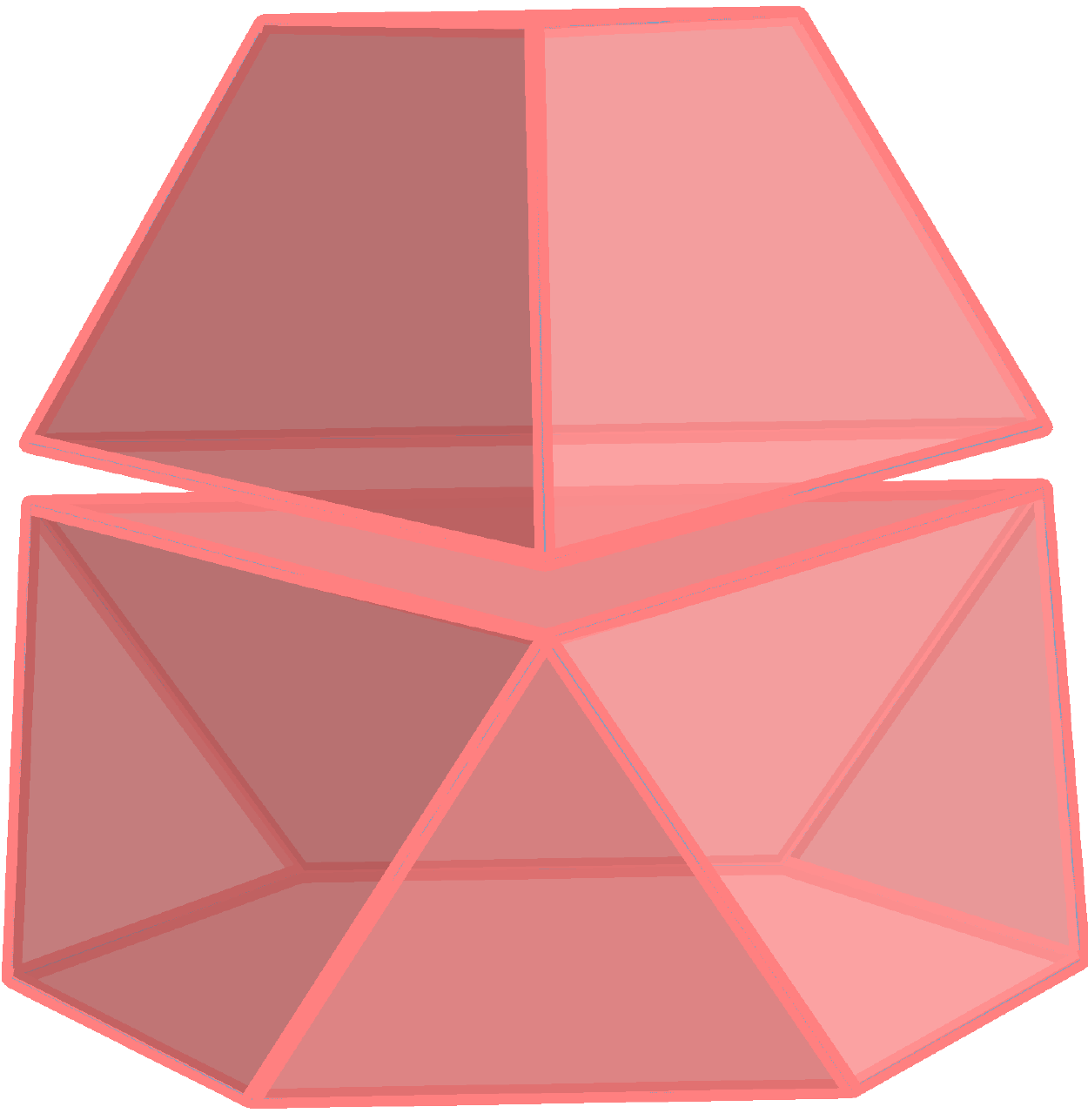} &
\includegraphics[width = 0.17 \linewidth]{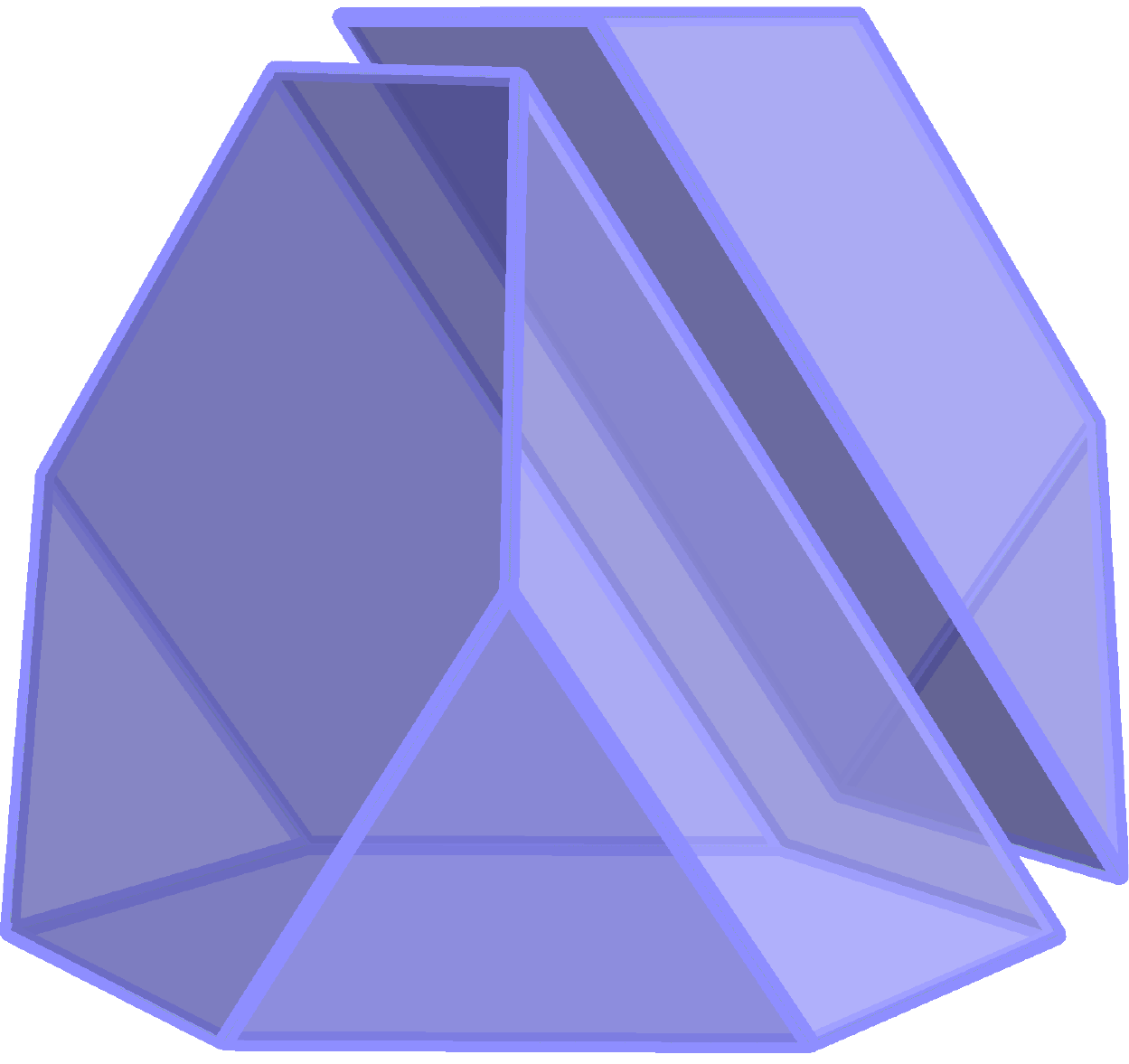} &
\includegraphics[width = 0.17 \linewidth]{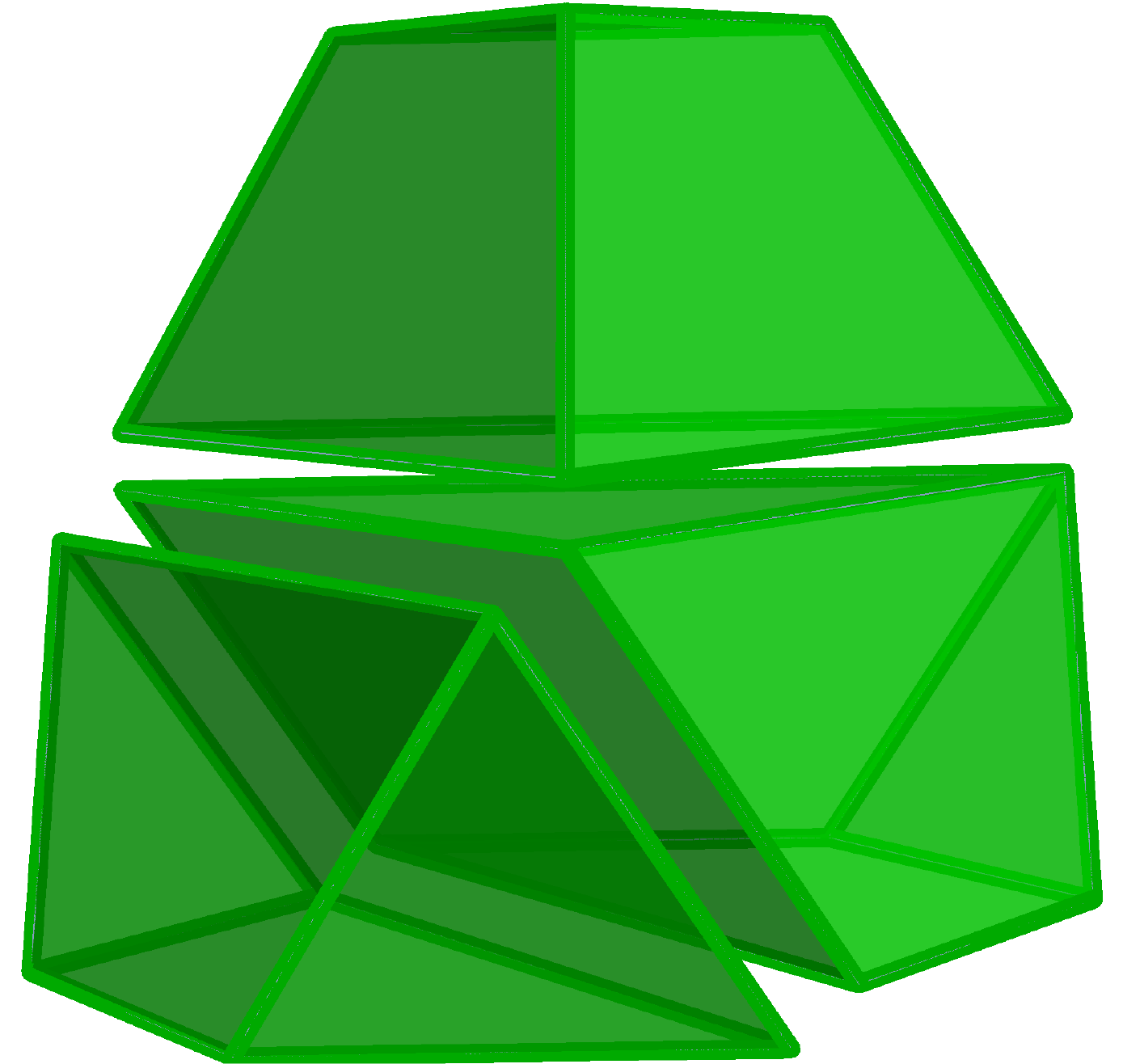} &
\includegraphics[width = 0.17 \linewidth]{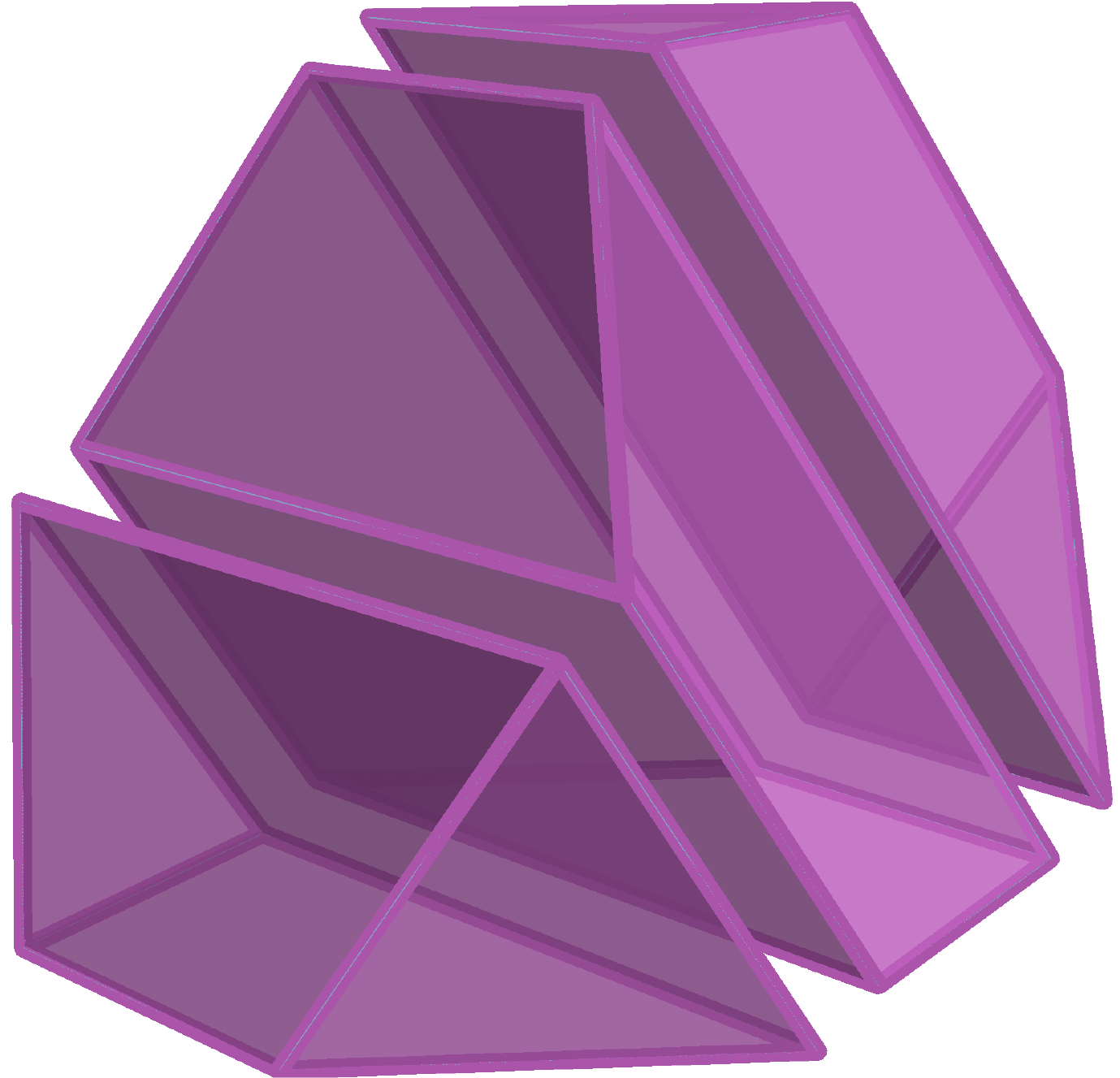} 
 \\
\includegraphics[width = 0.17 \linewidth]{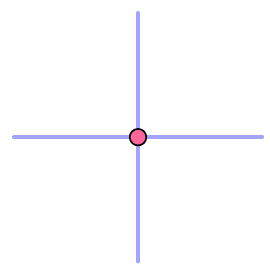} &
\includegraphics[height =1.1in,angle=180,origin=c]{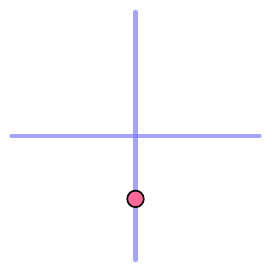}
&
\includegraphics[width = 0.17 \linewidth]{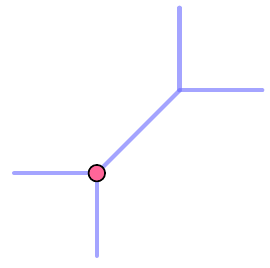}&
\includegraphics[width = 0.17 \linewidth]{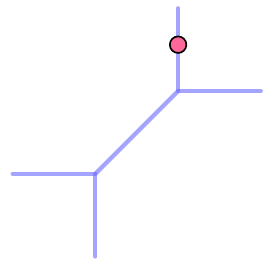}&
\includegraphics[width = 0.17 \linewidth]{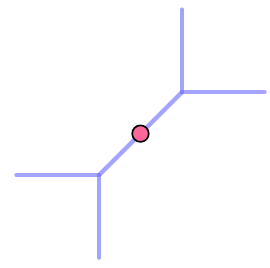}
\\
\includegraphics[width = 0.17 \linewidth]{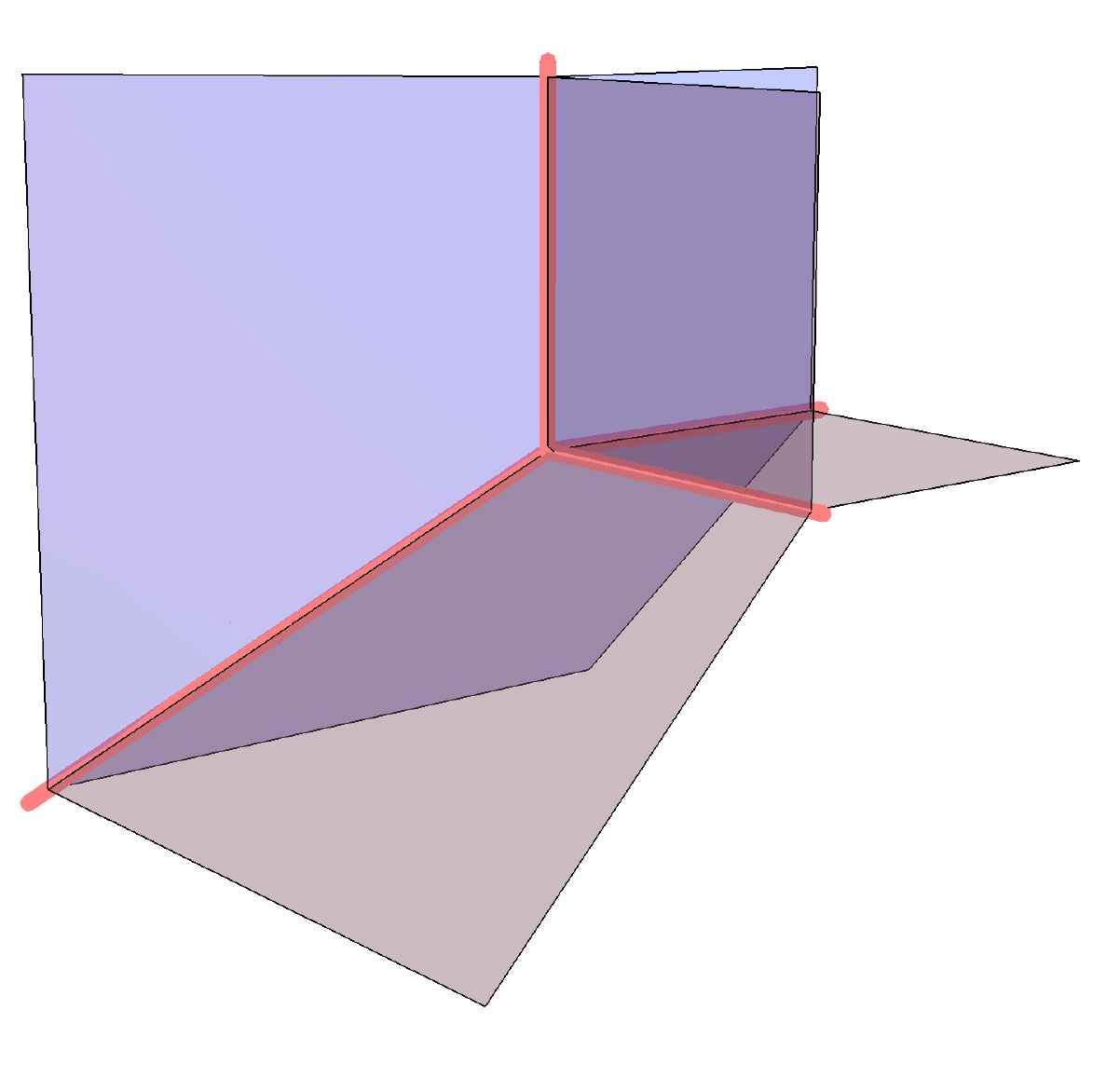} &
\includegraphics[width = 0.17 \linewidth]{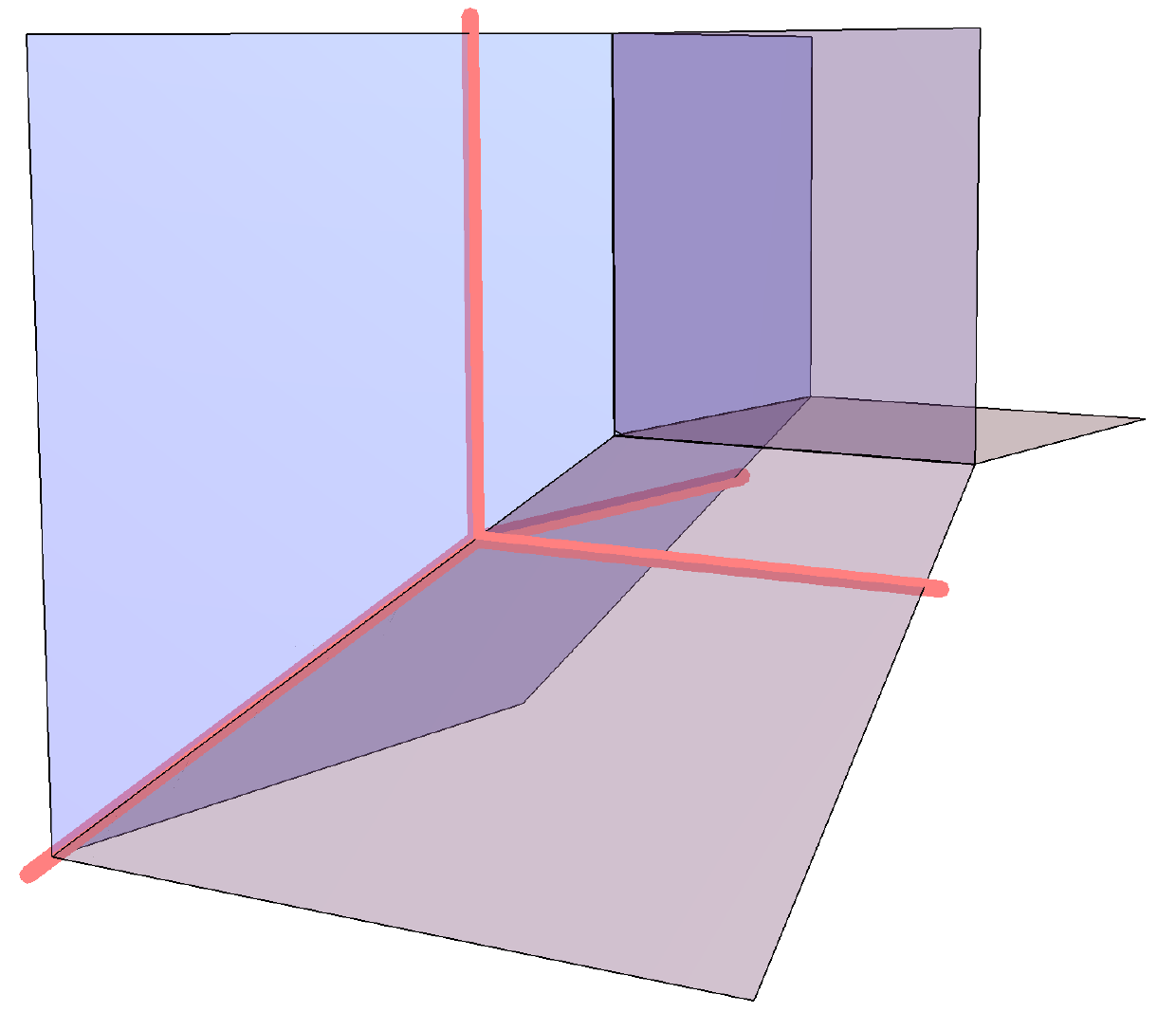} &
\includegraphics[width = 0.17 \linewidth]{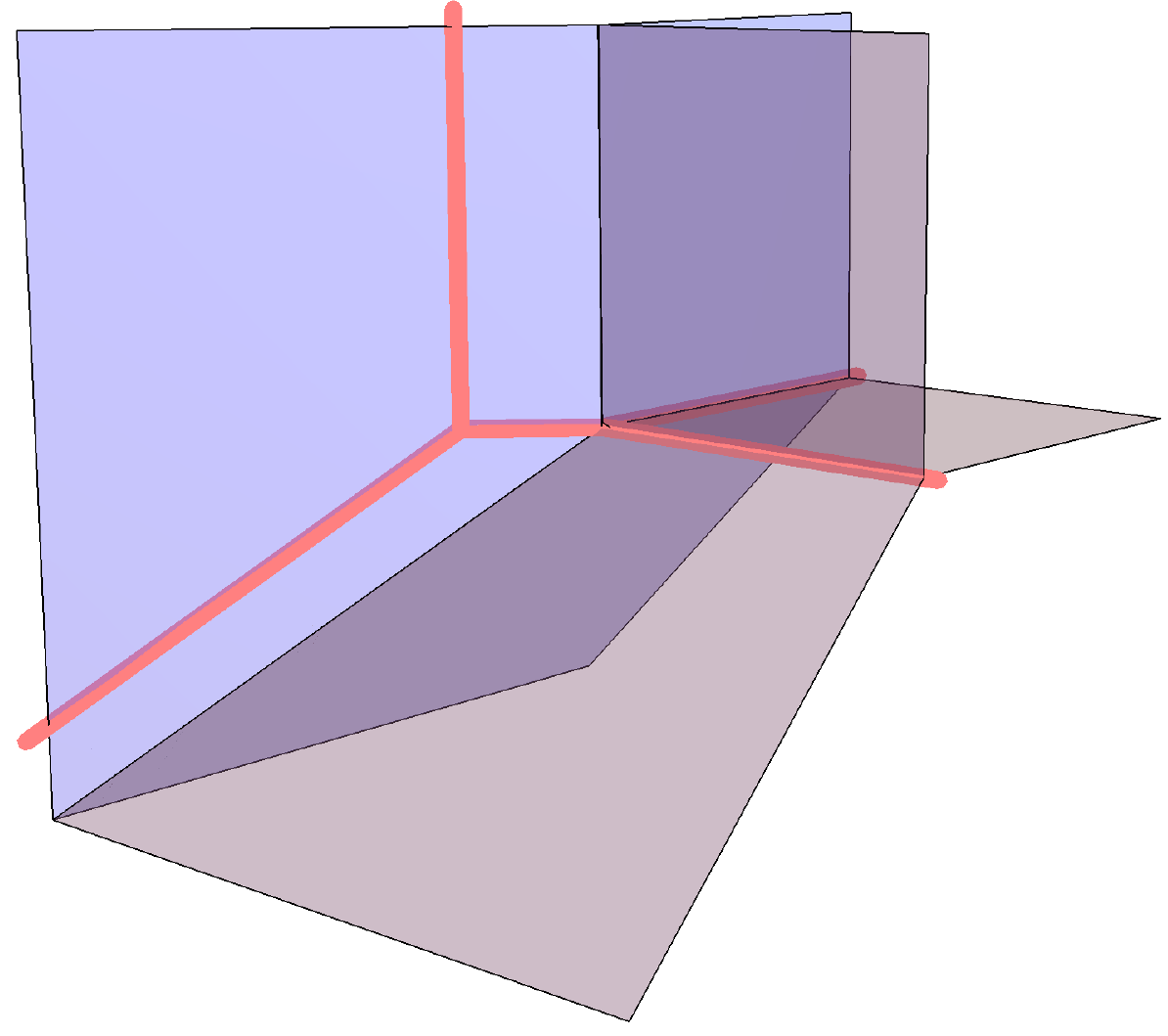} &
\includegraphics[width = 0.17 \linewidth]{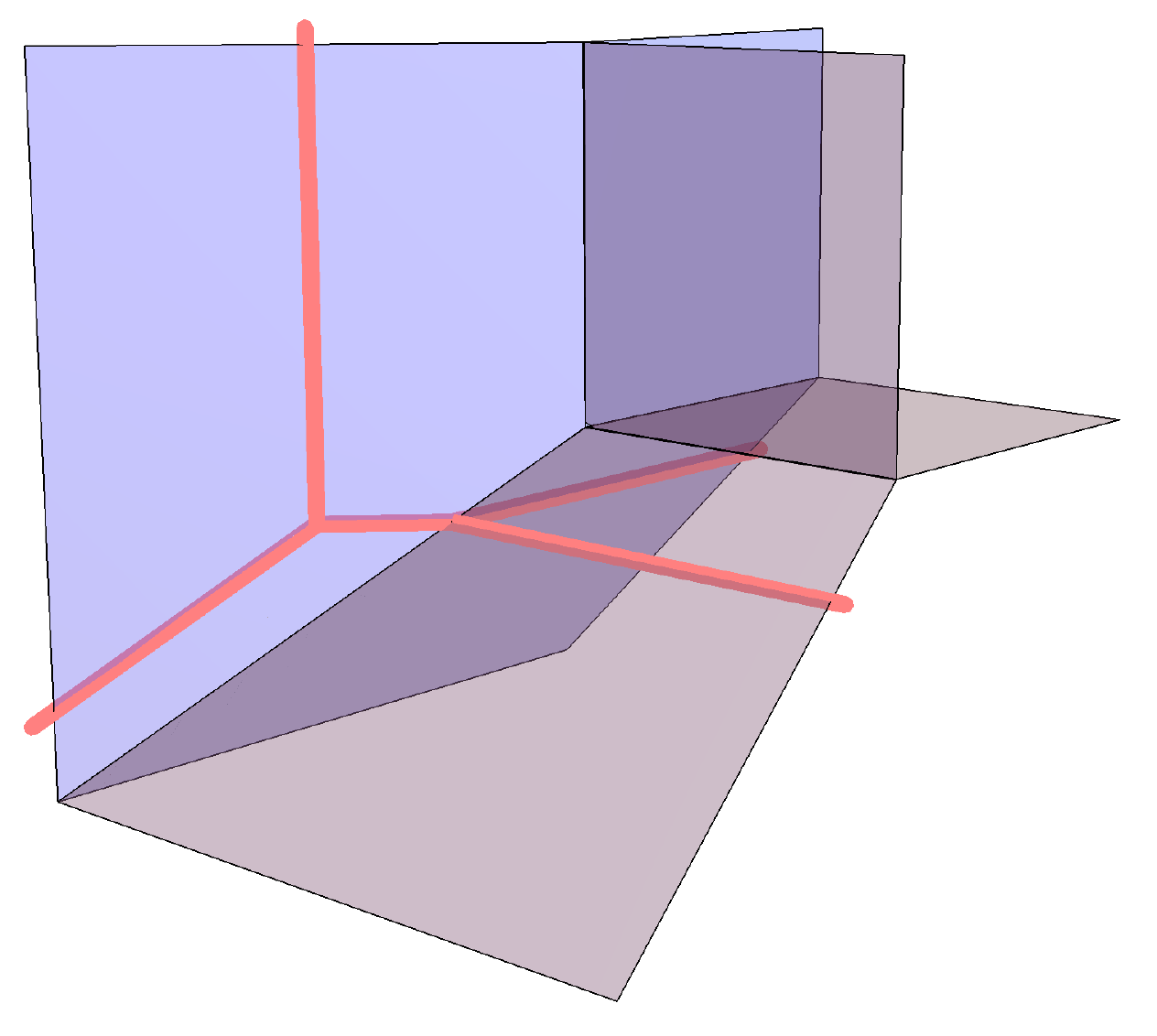} &
\includegraphics[width = 0.17 \linewidth]{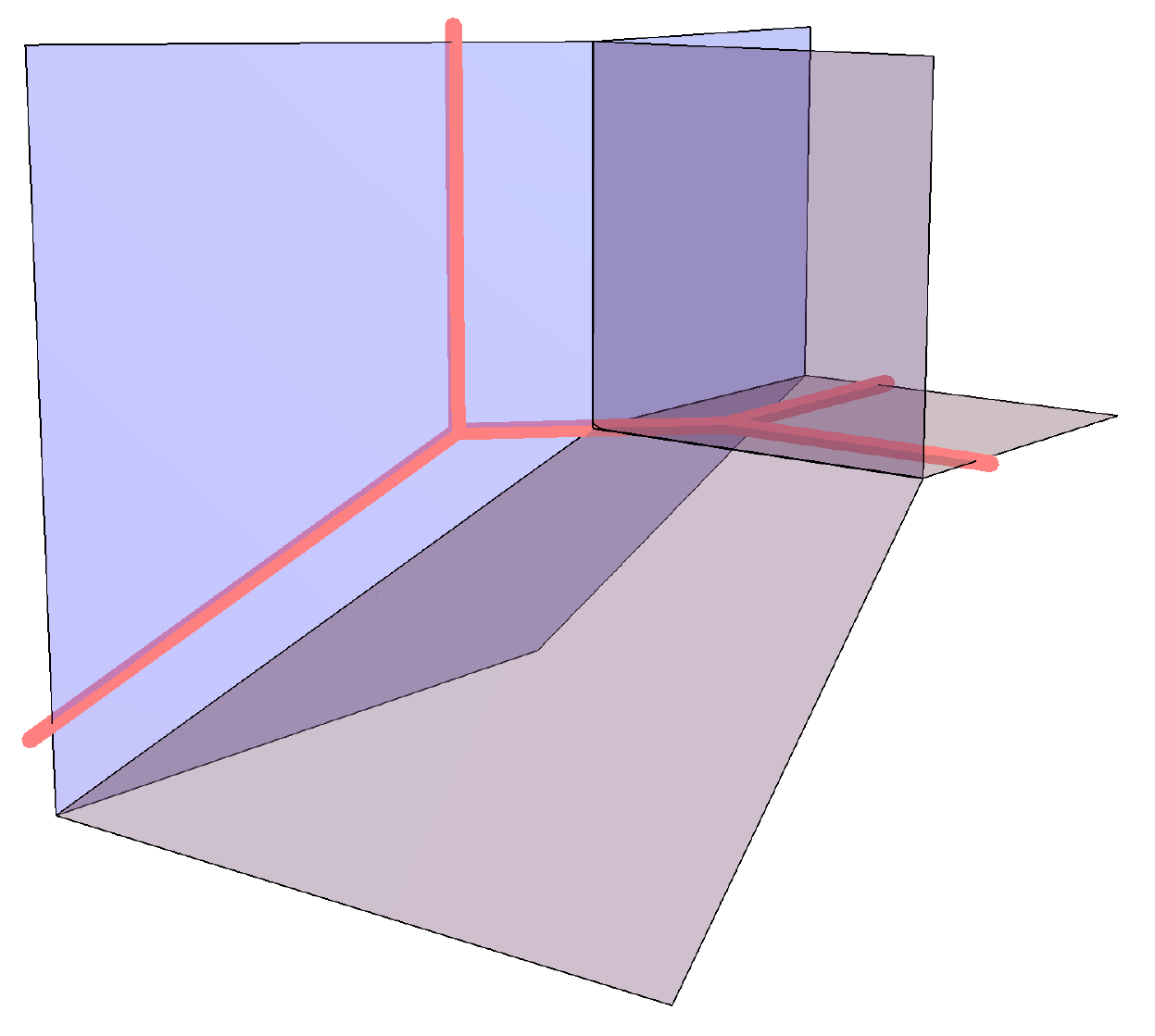} 
\end{tabular}
\caption{The subdivisions of $Q(\boldsymbol{U}_{1,2;4})$ induced by points in the flag Dressian $FlDr(\boldsymbol{U}_{1,2;4})$. The colors correspond to which points in $FlDr(\boldsymbol{U}_{1,2;4})$ induce that subdivision, see Figure \ref{fig:df124}. Each is displayed with the corresponding flag of a tropical point in a tropical line, and dually the tropical line in a tropical plane.}
\label{fig:df124_subs}
\end{figure}
\end{center}

\end{eg}

\section{Realizability}
\label{sec:realizability}

We now give an application of Theorem \ref{thm:main} to realizability. In Example \ref{eg:pointline} we saw that the flag Dressian $FlDr(1,2;4)$ is the Petersen graph, which is the same as the Dressian $Dr(2;5)$. In Theorem \ref{thm:fibration}, we explain this equality. In Theorem \ref{thm:Fl=FlDr} we see that every valuated flag matroid on ground set of size at most $5$ is realizable. We conclude with two examples. The first gives an interpretation of the tropicalization of the complete flag $Fl(1,2,3;4)$ as parameterizing two points in a tropical line. The second gives an example of non-realizability for a flag matroid on 6 elements.

\subsection{Relating Dressians and flag Dressians}
\label{subsec:drandfldr}

We begin by recalling a classical fact about matroid quotients with rank difference 1.  See \cite[\S7.3]{Oxl11} for a proof.

\begin{prop}\label{prop:elemquot}
Let $M'\twoheadleftarrow M$ be a matroid quotient on $[n]$ where the ranks of $M'$ and $M$ differ by 1.  Let $[\widetilde n] := \{0,1, \ldots, n\}$.  Then the collection of subsets of $[\widetilde n]$
\[
\bb(\widetilde{M}) := \{I'\cup 0 \mid I'\in \bb(M')\} \cup  \bb(M)
\]
is a set of bases of a matroid $\widetilde M$ on $[\widetilde n]$.  The matroid $\widetilde M$ is the unique matroid on $[\widetilde n]$ satisfying $M' = \widetilde{M}/_0$ and $M = \widetilde{M}\setminus_0$.
\end{prop}

We extend this fact to valuated matroid quotients with rank difference 1.  Let us first consider the following partially defined map
\[
\PP(\TT^{[\widetilde n]\choose r+1}) \dashrightarrow \PP(\TT^{[n] \choose r}) \times \PP(\TT^{[n]\choose r+1}) \textnormal{ defined by } (\bu_I)_{I\in {[\widetilde n]\choose r+1}} \mapsto (\bu_{I\setminus 0})_{I\ni 0} \times (\bu_I)_{I\not\ni 0}.
\]
This map is well-defined on the set 
\[
\Omega(r+1;\widetilde n) := \left\{\overline{\bu} \in \PP \big(\kk^{[\widetilde n]\choose r+1} \big) \middle | \begin{matrix} \bu_I \neq \infty \textnormal{ for some $I \in \textstyle \binom{[\widetilde n]}{r+1}$ with $I\ni 0$},\\ \textnormal{ and }\bu_{J} \neq \infty \textnormal{ for some $J\in \textstyle \binom{[\widetilde n]}{r+1}$ with $J\not\ni 0$}\end{matrix}\right\}.
\]
If $\widetilde \mu$ is a valuated matroid on $[\widetilde n]$ for which the map is well-defined, then its image under the map is the product of the two valuated matroids $\mu/_0$ and $\mu\setminus_0$ on $[n]$, which form a valuated matroid quotient by \Cref{cor:firstprop}.(3).  The following theorem generalizes \Cref{prop:elemquot} by showing that every valuated matroid quotient of rank difference 1 arises in this way.

\begin{thm}\label{thm:fibration}
Consider the Dressian $Dr(r+1;n+1)$ as a subset of $\PP(\TT^{[\widetilde n]\choose r+1})$.  The map 
\[
\Omega(r+1; \widetilde n) \cap  Dr(r+1;n+1) \to FlDr(r,r+1;n)
\]
induced by the partially defined map
$\PP(\TT^{[\widetilde n]\choose r+1}) \dashrightarrow \PP(\TT^{[n] \choose r}) \times \PP(\TT^{[n]\choose r+1})$
is surjective, and the fiber over a point $(\mu',\mu)\in FlDr(r,r+1;n)$ is
\[
\{(a\odot \mu' \oplus b\odot \mu) \in Dr(r+1;n+1) \mid a,b\in \RR\} \simeq \RR^2/\RR(1,1),
\]
where $\mu'$ and $\mu$ are considered as elements of $\PP(\TT^{[\widetilde n]\choose r+1})$ by 
\[
\mu'(I) = \begin{cases}
\mu'(I\setminus 0) & \textnormal{if $I \ni 0$}\\
\infty & \textnormal{otherwise}
\end{cases}
\quad \textnormal{and} \quad
\mu(J) = \begin{cases}
\mu(J) & \textnormal{if $J\not\ni 0$}\\
\infty & \textnormal{otherwise}.
\end{cases}
\]
\end{thm}



\begin{proof}
We have established that the map is well-defined: it sends a point of $Dr(r+1;n+1)$ to a point of $Fl(r,r+1;n)$ by \Cref{cor:firstprop}.(3).  Consider the fiber over a point $(\mu',\mu) \in FlDr(r,r+1;n)$, and write $(M',M)$ for the underlying flag matroid.  Let $\widetilde M$ be the matroid on $[\widetilde n]$ given by \Cref{prop:elemquot}.  For any $a,b\in \RR$, we need to show that $\widetilde \mu: {[\widetilde n]\choose r+1}\to \TT$ defined by
\[
\widetilde \mu(I) := \begin{cases} a+\mu'(I\setminus 0) & \textnormal{if $I\ni 0$}\\
b+\mu(I) & \textnormal{if $I\not\ni 0$}
\end{cases}
\]
is a valuated matroid.  By the equivalence of (b) and (c) in \Cref{thm:Dressians}, it suffices to show that the induced subdivision $\Delta_{\widetilde \mu}$ of the base polytope $Q(\widetilde M)$ consists only of base polytopes of matroids.

The base polytope $Q(\widetilde M)\subset \RR^{[\widetilde n]} = \RR\times \RR^{[n]}$ is the convex hull of $\{\be_0\}\times Q(M')$ and $\{0\} \times Q(M)$, so it is equivalent to the Cayley polytope of $Q(M')$ and $Q(M)$.  Thus, by the Cayley trick \cite[Theorem 9.2.16]{DLRS10}, the faces of the subdivision $\Delta_{\widetilde \mu}$ of $Q(\widetilde M)$ are in bijection with the faces of the subdivision $\Delta_{(a+\mu')+(b+\mu)}$ of $\bb(M') + \bb(M)$.  The subdivisions $\Delta_{(a+\mu')+(b+\mu)}$ and $\Delta_{\mu'+\mu}$ are the same, and by \Cref{thm:valflagsubdiv} each face of $\Delta_{\mu'+\mu}$ is a Minkowski sum of base polytopes of two matroids that form a matroid quotient.  Hence, we conclude from \Cref{prop:elemquot} that each face of $\Delta_{\widetilde \mu}$ is a base polytope of a matroid.
\end{proof}

\Cref{thm:fibration} is a tropical analogue of the following geometry.

\begin{rem}
Let $[\widetilde n] := \{0,1,\ldots, n\}$, and let $\{P_I \mid I\in {[\widetilde n]\choose r+1}\}$ be the Pl\"ucker coordinates of the embedding $Gr(r+1;n+1)\hookrightarrow \PP(\kk^{[\widetilde n]\choose r+1})$. Consider the rational map \[
\PP(\kk^{[\widetilde n]\choose r+1}) \dashrightarrow \PP(\kk^{[n] \choose r}) \times \PP(\kk^{[n]\choose r+1}) \textnormal{ where } (P_I)_{I\in {[\widetilde n]\choose r+1}} \mapsto (P_{I\setminus 0})_{I\ni 0}\times (P_I)_{I\not\ni 0}.
\]
With $Fl(r,r+1;n)$ embedded in $\PP(\kk^{[n] \choose r}) \times \PP(\kk^{[n]\choose r+1})$, this gives a rational map $Gr(r+1,n+1) \dashrightarrow Fl(r,r+1;n)$. 
The fiber over a point $(P_{I'})\times (P_J)\in Fl(r,r+1;n)$ is
\[
\{(aP_{I'\cup 0}, bP_J) \in Gr(r+1;n+1) \mid a,b\in \kk^*\} \simeq (\kk^*)^2/\kk^*,
\]
so that the map is a $\kk^*$-fibration.  \Cref{thm:fibration} shows that a similar map in the tropical setting is an $\RR$-fibration.
\end{rem}

\Cref{thm:fibration} relates Dressians and flag Dressians by their affine cones.

\begin{defn}
The \textbf{affine cone} of a projective tropical prevariety $X \subset \PP(\TT^E)$ is
\[
\widehat X := \left\{ \bu \in \TT^E \setminus\{\infty^E\} \mid \overline{\bu}\in X\right\} \cup \{\infty^E\}.
\]
Affine cones of multi-projective tropical prevarieties are similarly defined.
\end{defn}

\begin{cor}\label{cor:affineconesame}
Under the identification $\TT^{n+1\choose r+1} \simeq \TT^{n\choose r} \times \TT^{n\choose r+1}$, the affine cones $\widehat{Dr}(r+1;n+1)$ and $\widehat{FlDr}(r,r+1;n)$ of $Dr(r+1;n+1)$ and $FlDr(r,r+1;n)$ are identical. 
\end{cor}

\begin{proof}
Let $(\mu',\mu) \in \TT^{n\choose r} \times \TT^{n\choose r+1}$.  First consider the case where $\mu' = \infty^{\binom{n}{r}}$ or $\mu = \infty^{\binom{n}{r+1}}$.  Then $(\mu',\mu) \in \widehat{FlDr}(r,r+1;n)$ if and only if $\mu$ is a valuated matroid of rank $r+1$ on $n+1$ elements where the element $0$ is a loop (or respectively, $\mu'$ as a valuated matroid where $0$ is a coloop).  In other words $(\mu',\mu) \in \widehat{FlDr}(r,r+1;n)$ is equivalent to $(\mu',\mu) \in \widehat{Dr}(r+1;n+1)$ in this case.

If neither of $\mu'$ and $\mu$ is an all-$\infty$ vector, then \Cref{thm:fibration} implies that
\[
(\mu',\mu) \in \widehat{FlDr}(r,r+1;n) \implies (\mu',\mu)\in \widehat{Dr}(r+1;n+1),
\]
and \Cref{cor:firstprop}.(3) implies that $(\mu',\mu) \in Dr(r+1;n+1) \implies (\mu',\mu)\in  \widehat{FlDr}(r,r+1;n)$.
\end{proof}

When $r=1$, \Cref{cor:affineconesame} follows from observing that the collections of tropical Pl\"ucker relations that define $Dr(2;n+1)$ and $Fl(1,2;n)$ are identical after simply renaming the variables $P_i \in \PP(\TT^{[n]\choose 1})$ to $P_{i\cup 0}$.  This observation however fails for $r>1$.

\subsection{Realizability for small ground sets}
\label{subsec:smallground}

We compare the tropicalization of a partial flag variety and a flag Dressian in this subsection.  Due to the nature of this subsection, we use the contents of the geometric Remarks \ref{rem:projtrophyper}, \ref{rem:projtrop}, \ref{rem:geomDr}, \ref{rem:plucker}, and \ref{rem:geomFlDr}.

\medskip
A non-realizable valuated flag matroid corresponds to a point on the flag Dressian that does not lie in the tropicalization of the partial flag variety over any valued field $\kk$ (\Cref{rem:geomFlDr}).
Realizability of a valuated flag matroid can be subtle. In \Cref{eg:m4u26}, we give a valuated flag matroid $(\mu',\mu)$ that is not realizable, but its underlying flag matroid is realizable, and both valuated matroids $\mu'$ and $\mu$ are realizable over a common field.  For small ground sets realizability is guaranteed.   

\begin{thm}\label{thm:Fl=FlDr}
For $n\leq 5$, the tropicalization $\trop(Fl(r_1, \ldots, r_s;n))$ of a flag variety $Fl(r_1, \ldots, r_s;n)$ embedded in $\PP(\kk^{E\choose r_1}) \times \cdots \times \PP(\kk^{E\choose r_s})$ is equal to the flag Dressian $FlDr(r_1, \ldots, r_s;n)$.  Equivalently, for a valued field $\kk$ satisfying $\operatorname{val}(\kk) = \TT$, every valuated flag matroid on a ground set of size at most 5 is realizable over $\kk$.
\end{thm}

A valued field $\kk$ satisfying $\operatorname{val}(\kk) = \TT$ exists in every characteristic; see \cite[\S3]{Poo93} for an example known as Mal'cev-Neumann rings.  \Cref{thm:Fl=FlDr} fails for $n\geq 6$; see \Cref{eg:m4u26}.
We prepare the proof of the theorem with a lemma.  
\begin{lem}\label{lem:real}
Let $\kk$ be a valued field, and write $\Gamma := \operatorname{val}(\kk) \subseteq \TT$. Suppose ${\boldsymbol \mu} = (\mu_1, \mu_2, \ldots, \mu_k)$ is a valuated flag matroid on $[n]$ with $\operatorname{rk}(\mu_1) = 1$ such that $(\mu_2, \ldots, \mu_{k})$ is realizable over $\kk$, and $\mu_1$ as an element of $\TT^{\binom{[n]}{1}}$ has coordinates in $\Gamma$.  Then ${\boldsymbol \mu}$ is realizable over $\kk$.  By duality, if ${\boldsymbol \mu} = (\mu_1, \mu_2, \ldots, \mu_k)$ is a valuated flag matroid such that $(\mu_1, \ldots, \mu_{k-1})$ is realizable over $\kk$ and $\operatorname{rk}(\mu_{k}) = n-1$, then ${\boldsymbol \mu}$ is also realizable over $\kk$ when $\mu_k \in \Gamma^{\binom{[n]}{n-1}}$. 
\end{lem}

\begin{proof}
Let a flag $ L_2 \subset \cdots \subset L_{k}\subset \kk^E $ be a realization of $(\mu_2, \ldots, \mu_{k})$.  We need to show that that there exists a one-dimensional space $L_1$, that is, a point in $\PP(\kk^E)$, such that $L_1\subset L_{2}$ and $\overline{\trop}(\mu_1) = \overline{\trop}(L_1)$.  But since $\operatorname{rk}(\mu_1) = 1$, the space $\overline{\trop}(\mu_1)$ is a single point, which by \Cref{thm:tropquotient} is on $\overline{\trop}(\mu_{2}) = \overline{\trop}(L_{2})$.  By the lifting property in the Fundamental Theorem of Tropical Geometry \cite[Theorem 3.2.3, Theorem 6.2.15]{MS15}, there exists a point $p_1\in \PP(L_{2}) \subset \PP(\kk^E)$ with $\overline{\trop}(p_1) = \overline{\trop}(\mu_1)$.
\end{proof}

\begin{proof}[Proof of \Cref{thm:Fl=FlDr}]
We first note some previous results:
\begin{itemize}
\item One has $\trop(Gr(1;n)) = Dr(1;n) = \PP(\TT^E)$, and dually, $\trop(Gr(n-1;n)) = Dr(n-1;n)$.
\item For any $n$, one has $\trop(Gr(2;n)) = Dr(2;n)$, and dually, $\trop(Gr(n-2;n)) = Dr(n-2;n)$ \cite[Corollary 4.3.12]{MS15}.
\item One has $\trop(Gr(3;6)) = Dr(3;6)$ \cite[Example 4.4.10]{MS15}.
\end{itemize}
By \Cref{thm:fibration}, the desired statement thus holds for $Fl(1,2;n)$, its dual $Fl(n-2,n-1;n)$, and $Fl(2,3;5)$.  The rest of the cases for $n\leq 5$ then follow from \Cref{lem:real}.
\end{proof}


\begin{eg}
Let $Fl_4 := Fl(1,2,3;4)$, and denote by $\mathring{Fl}_4$ the very affine variety obtained as the intersection of $Fl_4$ embedded in $\PP^3 \times \PP^5 \times \PP^3$ with the torus $(\kk^*)^4/\kk^* \times (\kk^*)^6/\kk^* \times (\kk^*)^4/\kk^*$. The $f$-vector of its tropicalization $\trop(\mathring{Fl}_4)$, with the Gr\"obner complex for its polyhedral complex structure, was computed in \cite{BLMM17} to be $(1,20,79,78)$ with the aid of a computer.  We now give an explicit description of the combinatorial structure of $\trop(\mathring{Fl}_4)$.

By \Cref{thm:Fl=FlDr}, we have that $\trop(\mathring{Fl}_4) = FlDr(\UU_{1,2,3;4})$ where $\UU_{1,2,3;4} = (U_{1,4}, U_{2,4},U_{3,4})$.  If $\mu$ is a valuated matroid whose underlying matroid is $U_{2,4}$, then $\trop(\mu)$ is a translate of $\trop(\mu^*)$.  Thus, by \Cref{thm:tropquotient}, one can identify the space $FlDr(\UU_{1,2,3;4})$ as the parameter space of two labeled points on a tropical line.  Using this, we completely describe the polyhedral complex structure of $\trop(\mathring{Fl}_4) = FlDr(\UU_{1,2,3;4})$ in Figures \ref{fig:fl4_cells_edges} and \ref{fig:fl4_cells}.  The pictorial representations of the maximal cells in \cite[Fig.\ 2]{BLMM17} are related to but different from ours.

\begin{figure*}[h]
    \centering
    \begin{subfigure}[t]{0.14\textwidth}
        \centering
        \includegraphics[width = \linewidth]{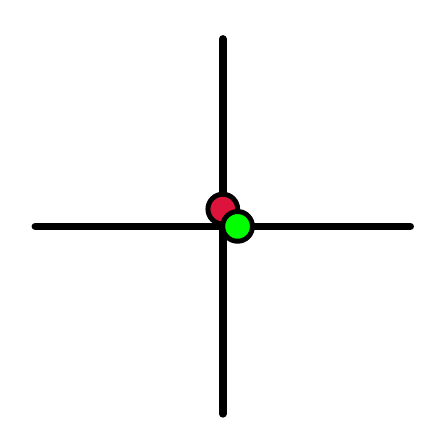}
        \caption*{The origin}
    \end{subfigure}%
    \hspace{0.2 in}
    \begin{subfigure}[t]{0.14\textwidth}
        \centering
        \includegraphics[width = \linewidth]{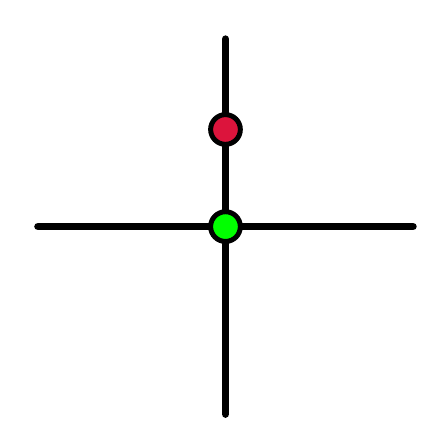}
        \caption*{$2\cdot 4$ rays}
    \end{subfigure}%
        \hspace{0.2 in}
        \begin{subfigure}[t]{0.14\textwidth}
        \centering
        \includegraphics[width = \linewidth]{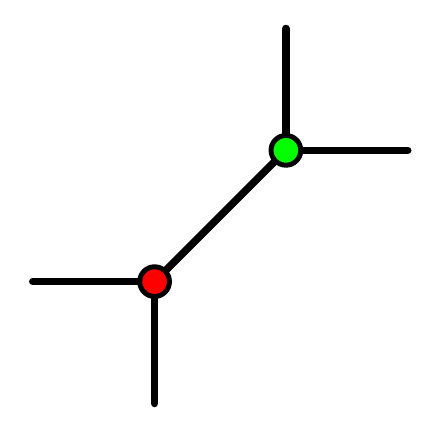}
        \caption*{$3 \cdot 2$ rays}
    \end{subfigure}%
        \hspace{0.2 in}
        \begin{subfigure}[t]{0.14\textwidth}
        \centering
        \includegraphics[width = \linewidth]{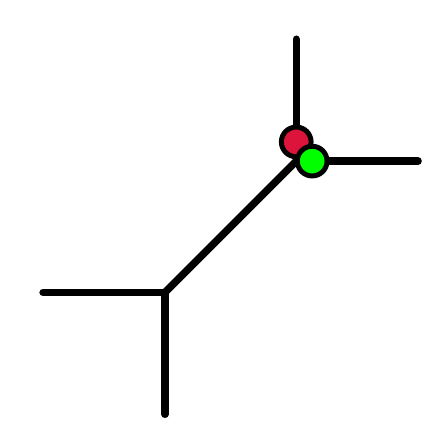}
        \caption*{$3 \cdot 2$ rays}
    \end{subfigure}%

    \begin{subfigure}[t]{0.30\textwidth}
        \centering
        \includegraphics[width = \linewidth]{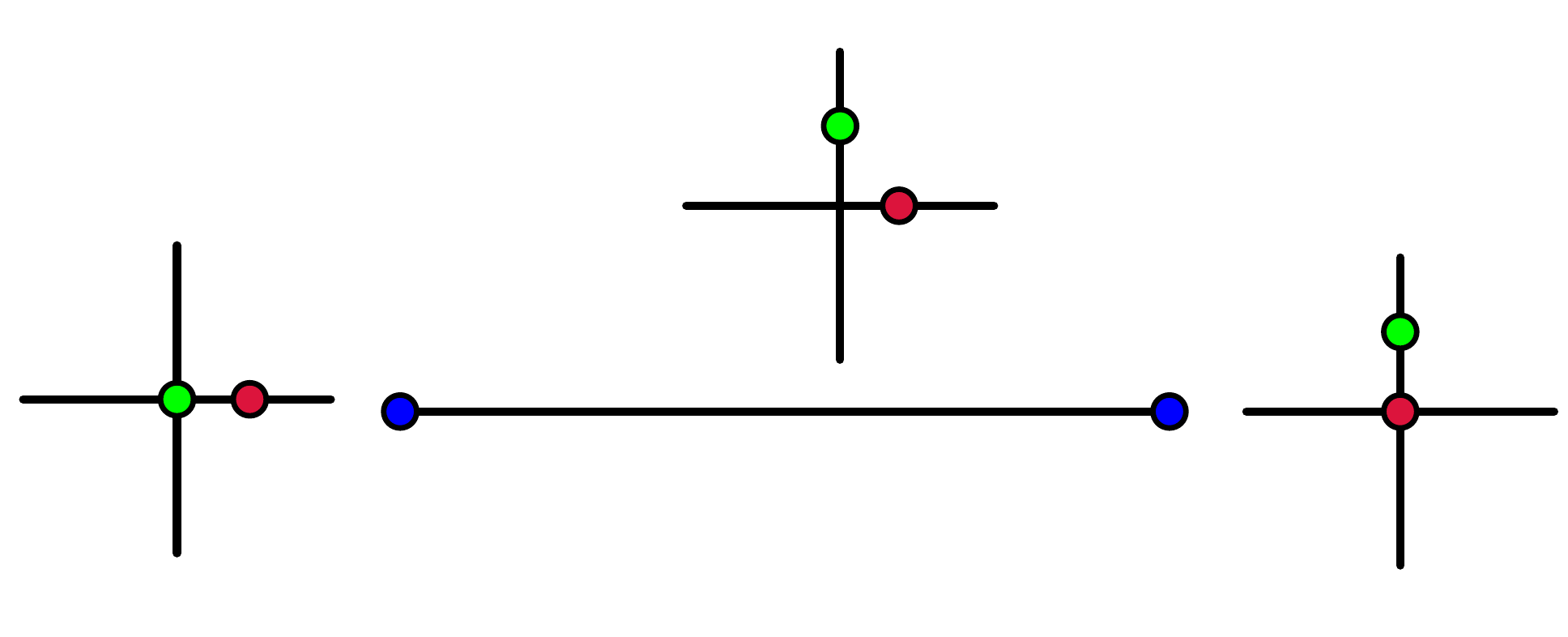}
        \caption*{$\binom{4}{2}\cdot 2$ edges}
    \end{subfigure}%
    \hspace{0.2 in}
    \begin{subfigure}[t]{0.30\textwidth}
        \centering
        \includegraphics[width = \linewidth]{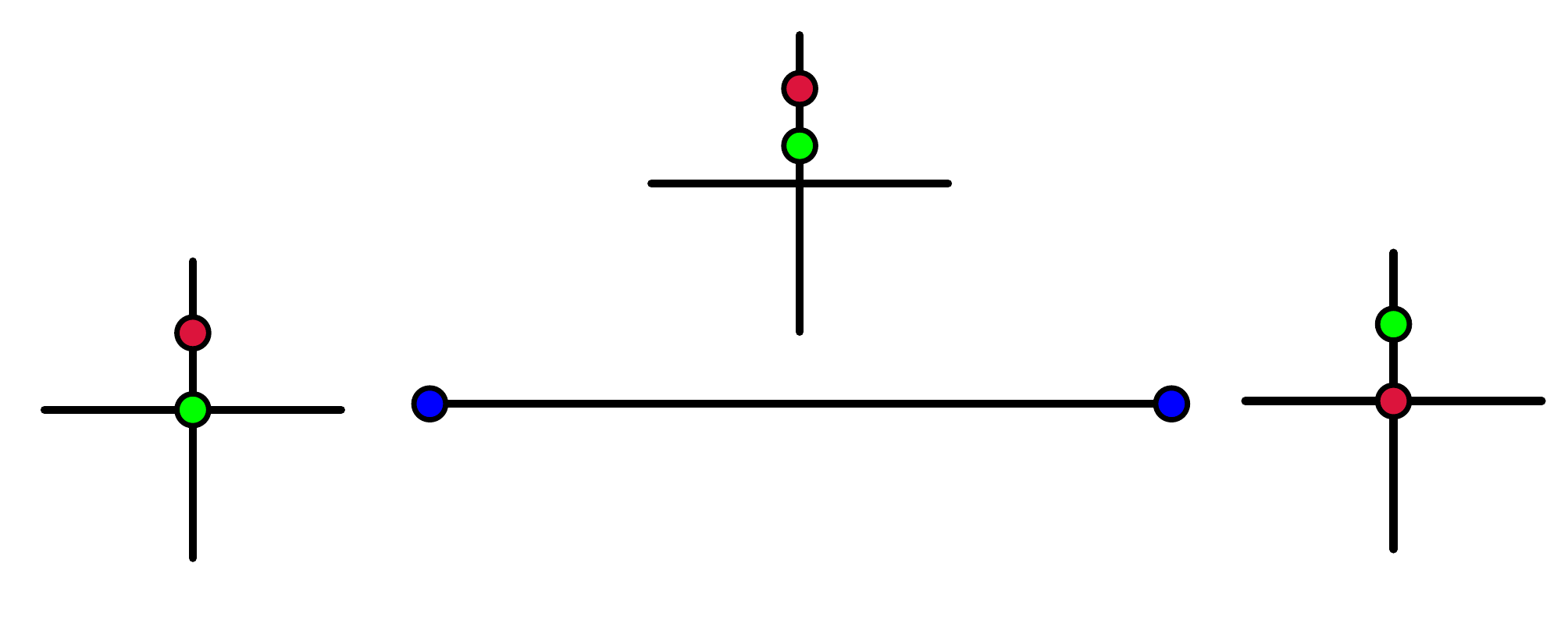}
        \caption*{4 edges}
    \end{subfigure}%
        \hspace{0.2 in}
    \begin{subfigure}[t]{0.30\textwidth}
        \centering
        \includegraphics[width = \linewidth]{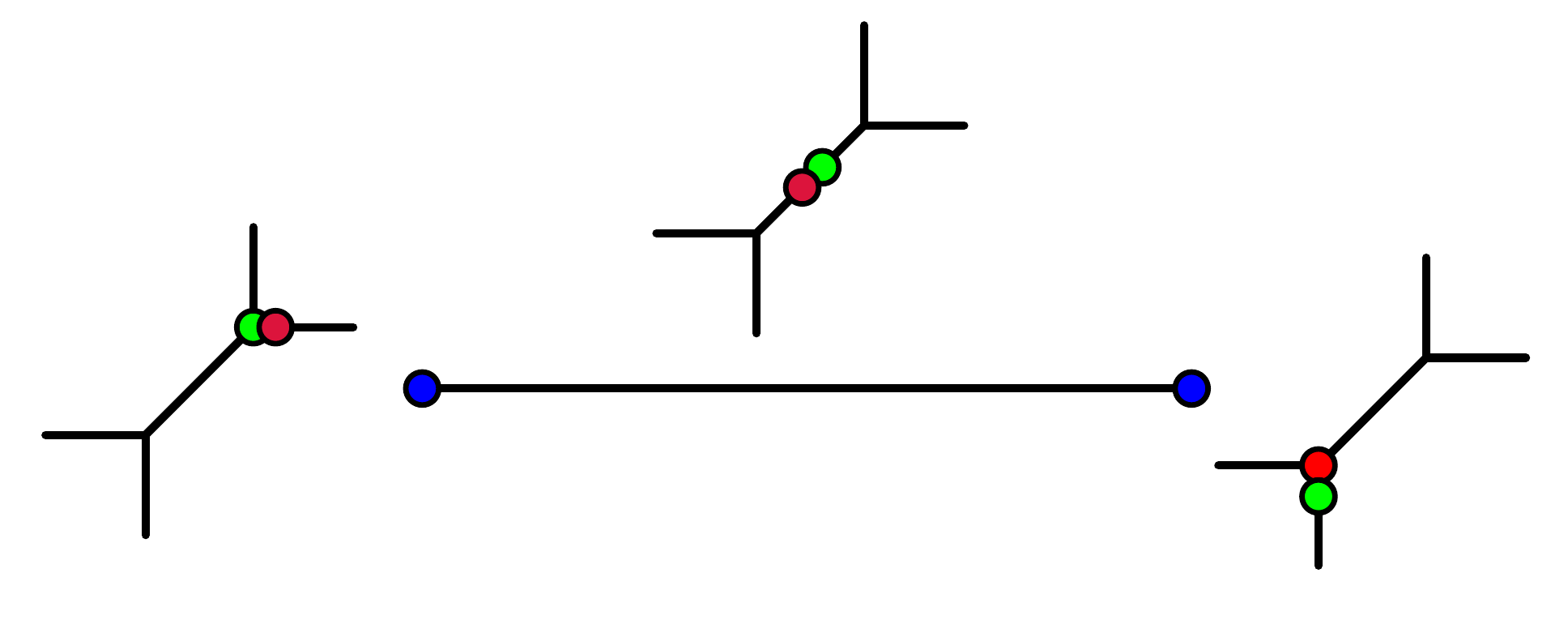}
        \caption*{3 edges}
    \end{subfigure}%
        \hspace{0.2 in}
    \begin{subfigure}[t]{0.30\textwidth}
        \centering
        \includegraphics[width = \linewidth]{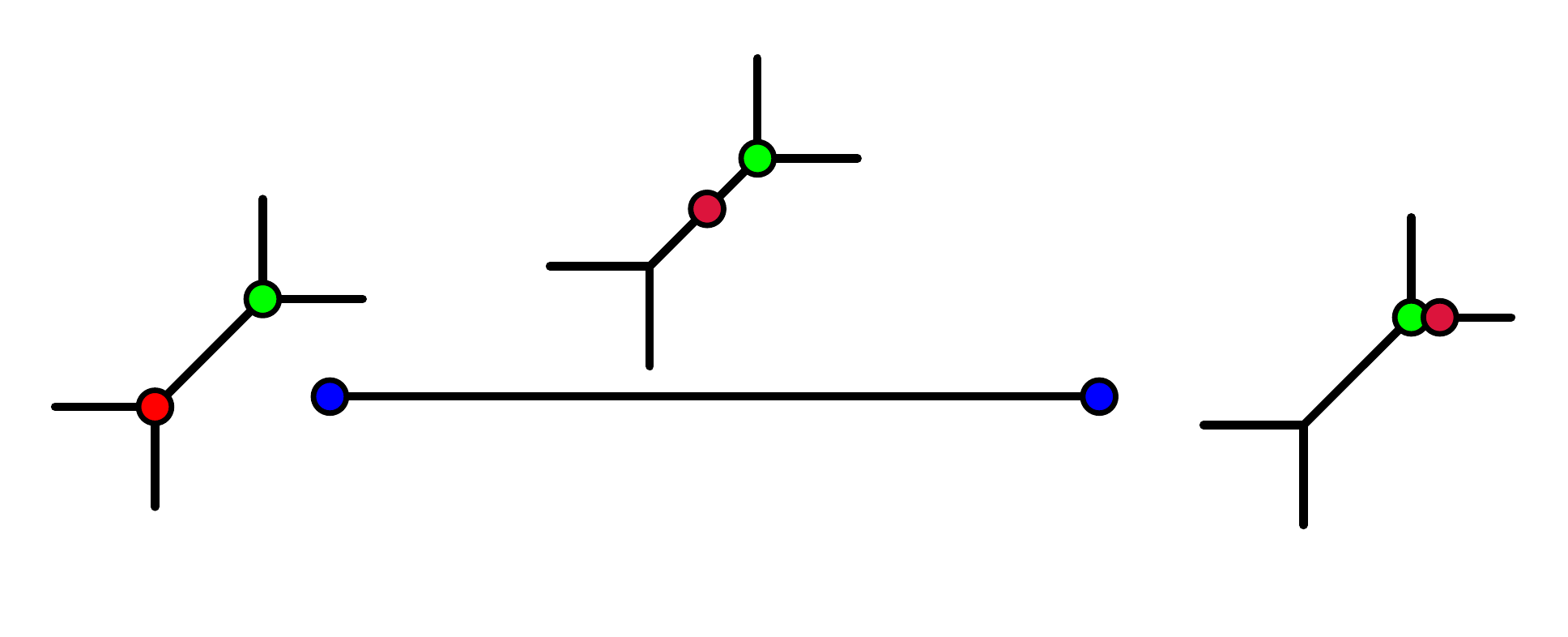}
        \caption*{$3\cdot2\cdot3$ edges}
    \end{subfigure}%
        \hspace{0.2 in}
        \begin{subfigure}[t]{0.30\textwidth}
        \centering
        \includegraphics[width = \linewidth]{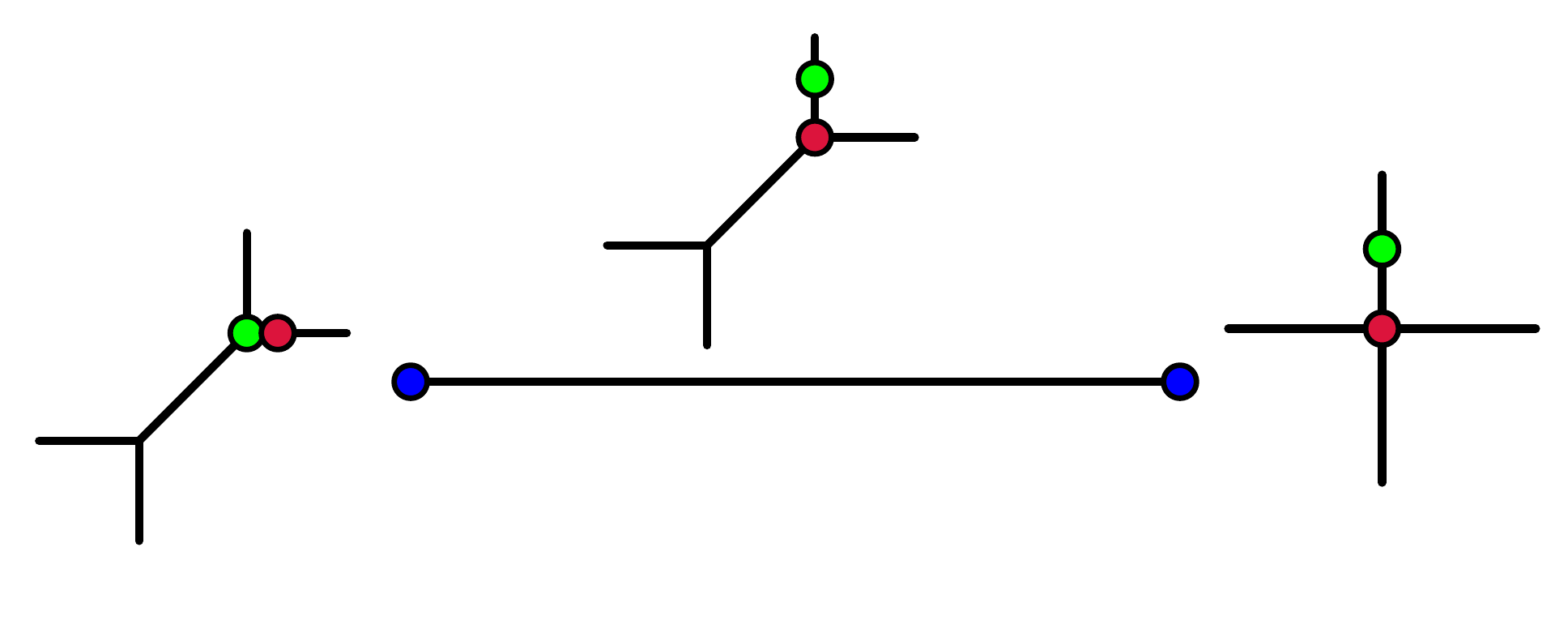}
        \caption*{$3 \cdot 2 \cdot 2 \cdot 2$ edges}
    \end{subfigure}%
        \hspace{0.2 in}
        \begin{subfigure}[t]{0.30\textwidth}
        \centering
        \includegraphics[width = \linewidth]{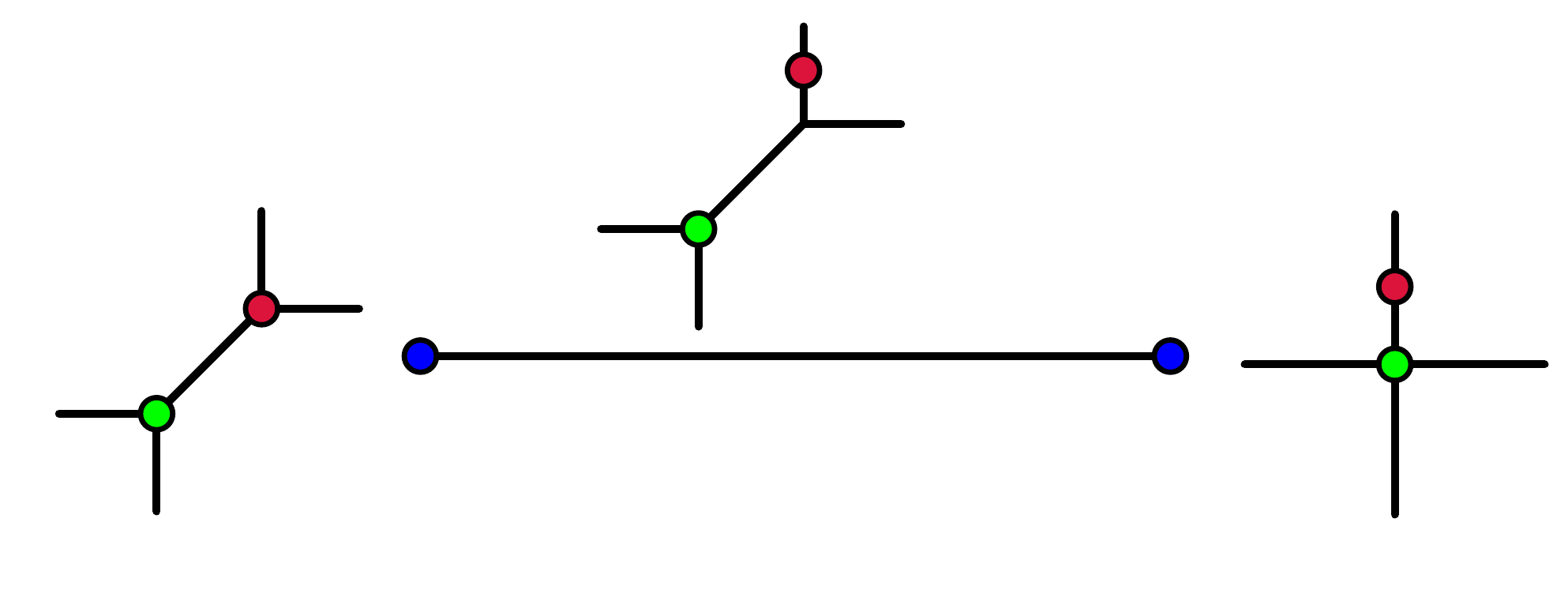}
        \caption*{$3 \cdot 2 \cdot 2 \cdot 2$ edges}
    \end{subfigure}%
    \caption{The origin, 20 rays, and 79 edges of  $FlDr(\UU_{1,2,3;4})$.}
    \label{fig:fl4_cells_edges}
\end{figure*}

\begin{figure*}[h]
    \centering

    \begin{subfigure}[t]{0.24\textwidth}
        \centering
        \includegraphics[width = \linewidth]{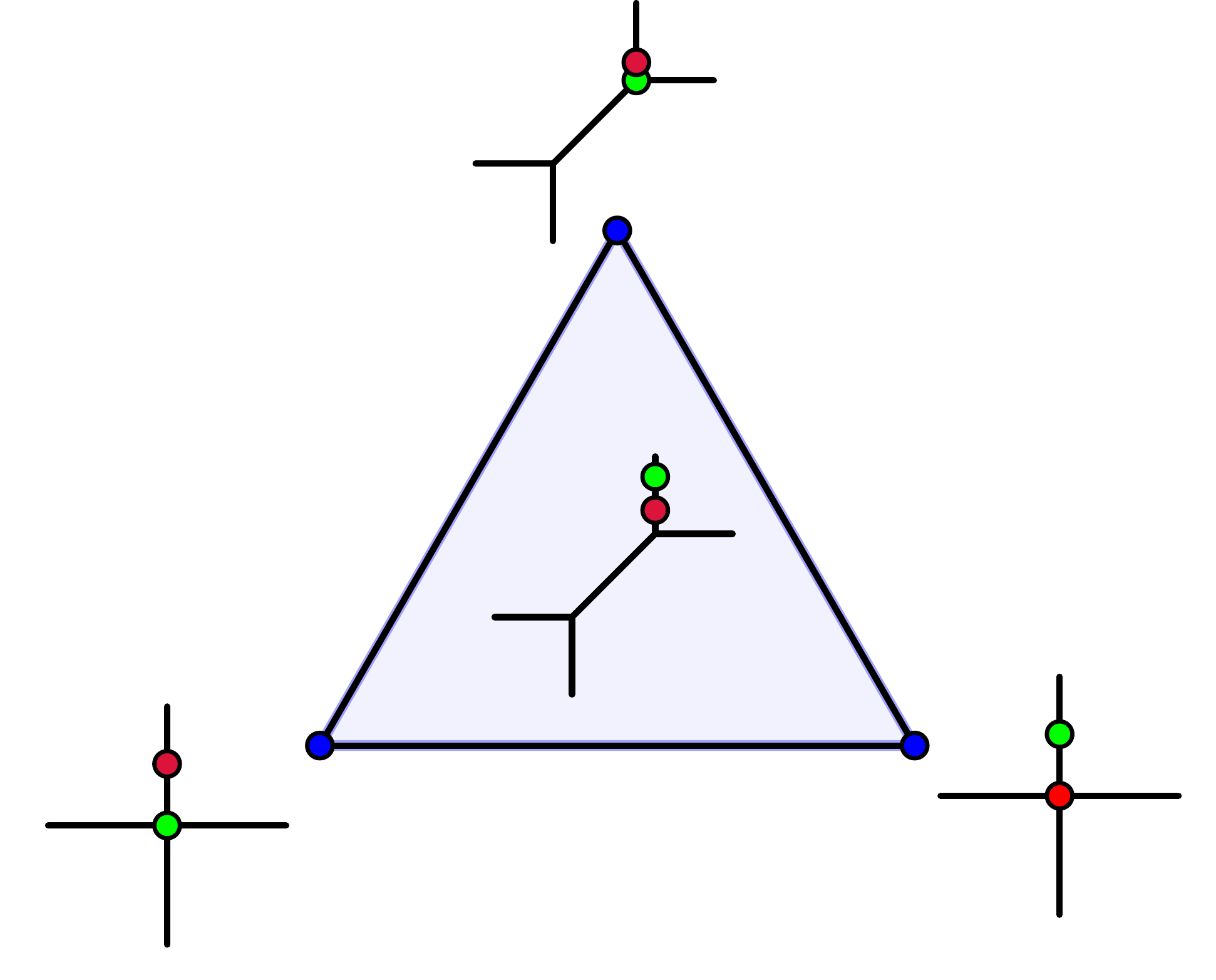}
        \caption*{$3 \cdot 4$ triangles}
    \end{subfigure}%
    \hspace{0.2 in}
    \begin{subfigure}[t]{0.24\textwidth}
        \centering
        \includegraphics[width = \linewidth]{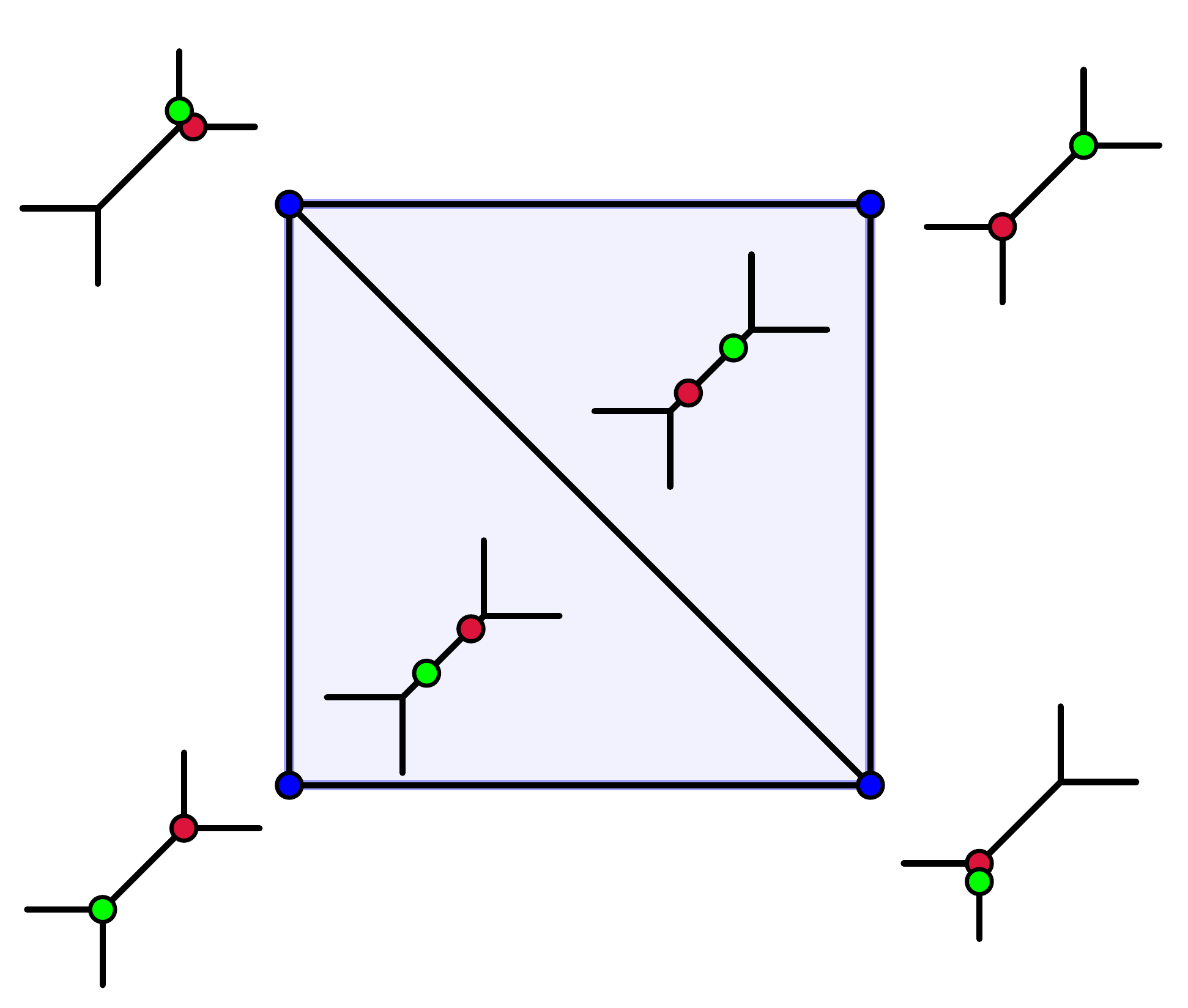}
        \caption*{$3 \cdot 2$ subdivided squares}
    \end{subfigure}%

        \hspace{0.2 in}
    \begin{subfigure}[t]{0.24\textwidth}
        \centering
        \includegraphics[width = \linewidth]{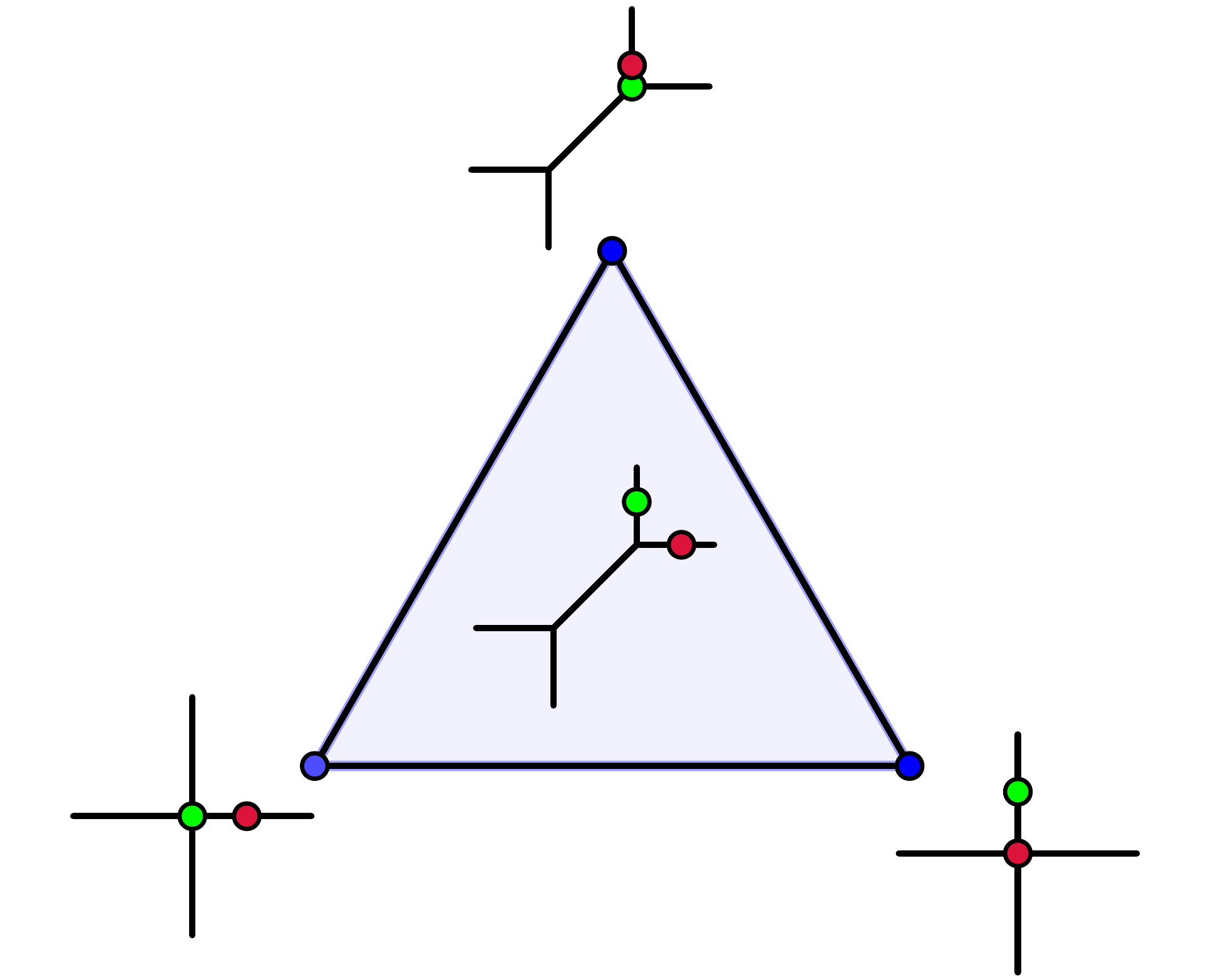}
        \caption*{$3\cdot 2 \cdot 2$ triangles }
    \end{subfigure}%
        \hspace{0.2 in}
        \begin{subfigure}[t]{0.24\textwidth}
        \centering
        \includegraphics[width = \linewidth]{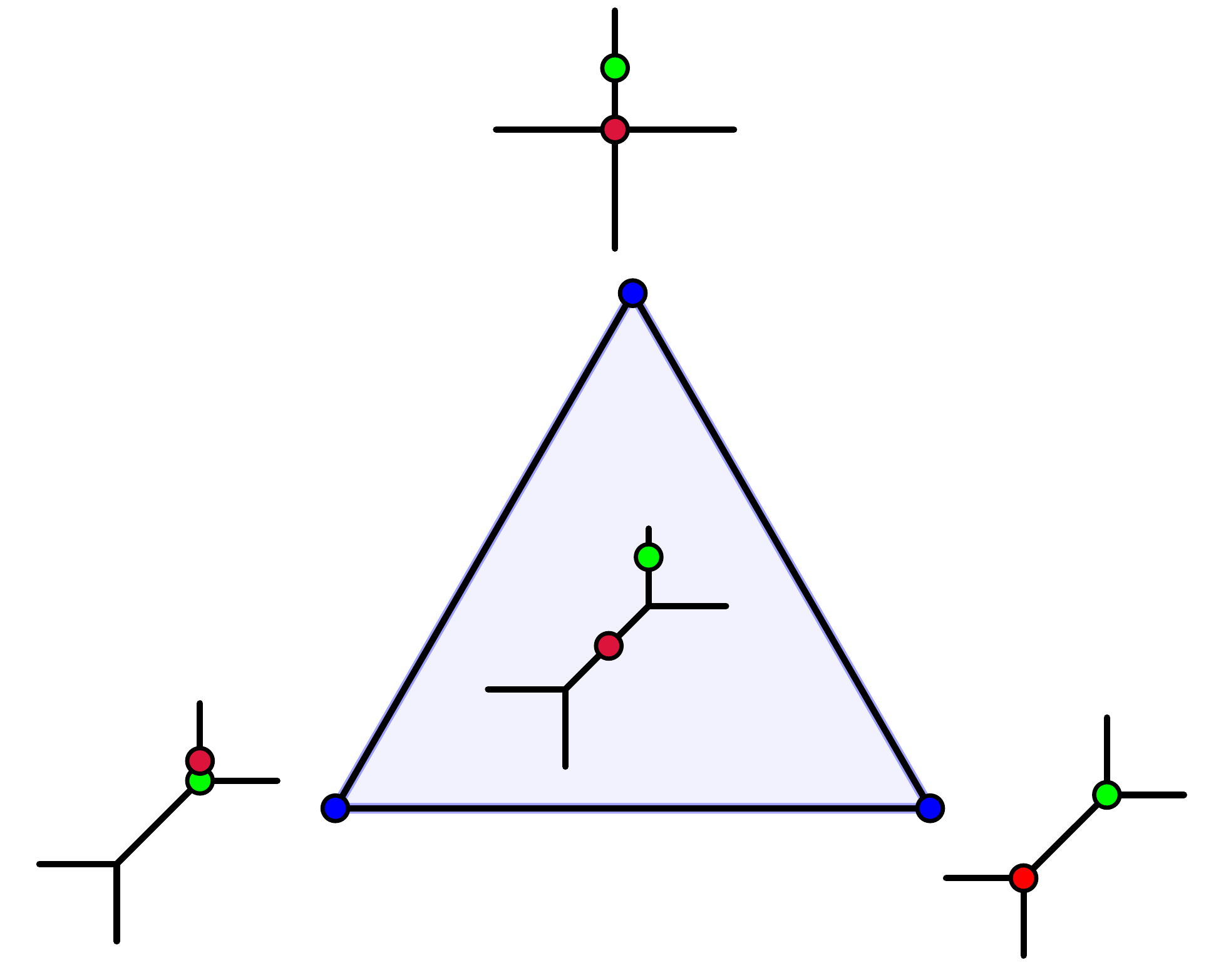}
        \caption*{$3\cdot 4 \cdot 2$ triangles }
    \end{subfigure}%
        \hspace{0.2 in}
        \begin{subfigure}[t]{0.24\textwidth}
        \centering
        \includegraphics[width = \linewidth]{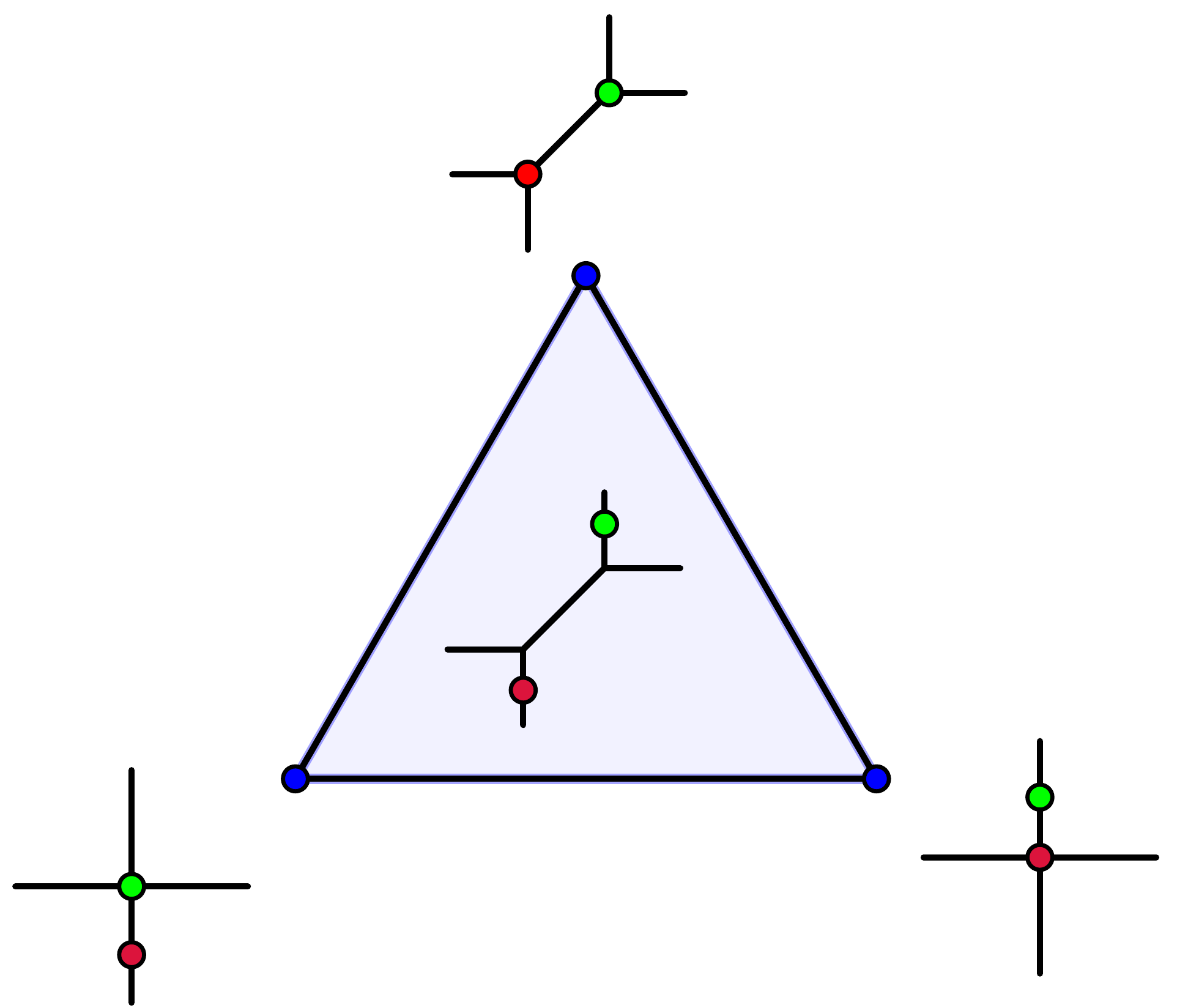}
        \caption*{$3 \cdot 2 \cdot 2 \cdot 2$ triangles }
    \end{subfigure}%
    \caption{The 78 2-cells of $FlDr(\UU_{1,2,3;4})$.  78 is also $3\cdot 5^2 +3$ (three kinds of tropical lines, each with five 1-cells, plus 3 from the subdivided squares).}
    \label{fig:fl4_cells}
\end{figure*}
\end{eg}

\newpage
The next example highlights the subtleties of realizability.
In light of \Cref{thm:fibration}, the example below is closely related to \cite[Example 4.3.14]{MS15}.

\begin{eg}
\label{eg:m4u26}
Consider the flag matroid $\MM = (U_{2,6},M_4)$ pictured in Figure \ref{fig:m4u26}.
\begin{figure}[h]
    \centering
    \includegraphics[height = 1.8 in]{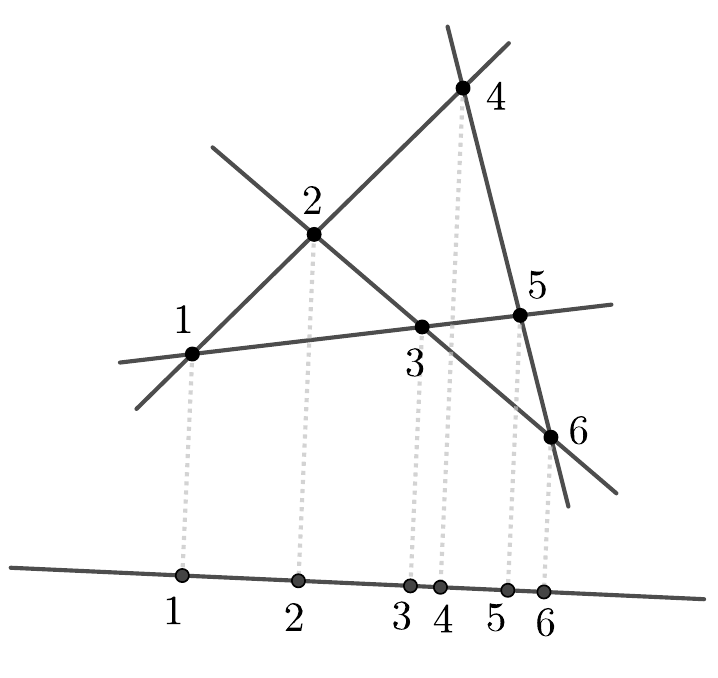}
    \caption{The flag matroid considered in Example \ref{eg:m4u26}.}
    \label{fig:m4u26}
\end{figure}
The matroid $M_4$ is the rank $3$ matroid on $E=\{1,\ldots, 6\}$ with circuit hyperplanes $\{124,135,236,456\}$. The affine cone of the flag Dressian $FlDr(\MM)$ is a 10 dimensional fan, with a 7 dimensional lineality space. Modulo lineality and intersecting with a sphere, it has 13 rays, 21 edges, and one triangle, depicted in Figure~\ref{fig:m4u26_dress}.
\begin{figure}[h]
    \centering
    \includegraphics[width=0.80\linewidth]{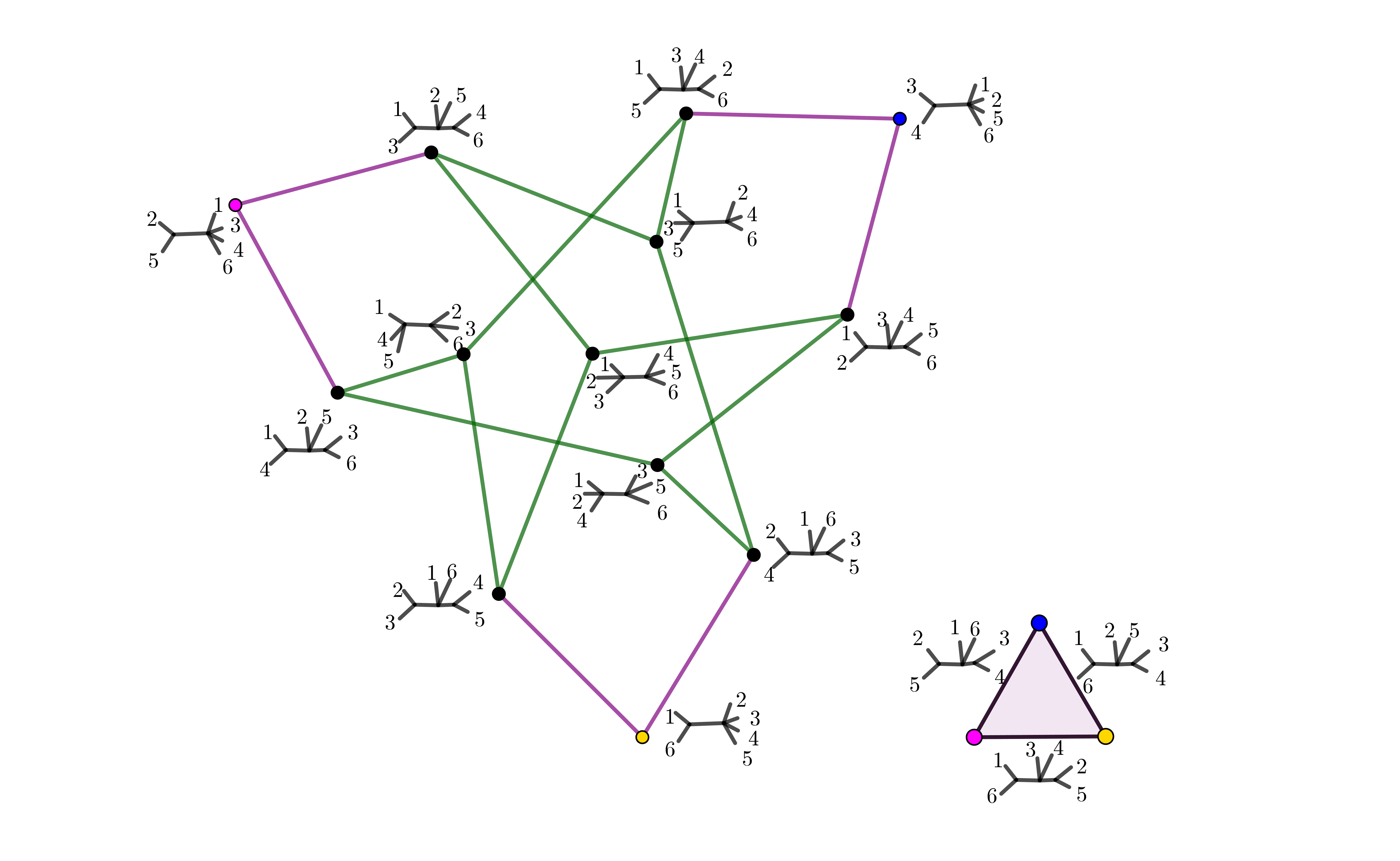}
    \caption{The Dressian of the flag matroid in Example \ref{eg:m4u26}.}
    \label{fig:m4u26_dress}
\end{figure}
In Figure~\ref{fig:m4u26_dress}, rays are labeled with the tree given by the $U_{2,6}$ coordinates. The green edges in the graph correspond to points where the corresponding tree is a caterpillar, and the purple points give snowflake trees.
The triangle is glued to the pink, blue, and yellow vertices as indicated.

Let $V \subset Fl(2,3;6)$ be the space of realizations of $(M_4,U_{2,6})$.  The affine cone of this 
subvariety $V$ of $Fl(2,3;6)$
has dimension 9, but the affine cone of the flag Dressian $FlDr(\MM)$ has dimension 10. Since $\dim { \operatorname{trop}}(V)=\dim V$ \cite[Theorem 3.3.8]{MS15}, 
the flag Dressian $FlDr(\MM)$ must strictly contain $\overline{\operatorname{trop}}(V).$
Indeed, the tropicalization $\trop(V)$
when $\kk$ has characteristic 0 consists of all zero and one dimensional cells in Figure \ref{fig:m4u26_dress}; the interior of the single triangle is removed. Over characteristic 2, some points in the interior of the triangle may be on  $\trop(V)$, but not the entire triangle.

Let us understand the non-realizable points in the interior of this triangle in more detail when $\kk$ has characteristic 0. The Dressians for $U_{2,6}$ and $M_4$ are each tropical varieties, meaning that every point $w$ in each of their Dressians can be realized as vectors over $\kk$ whose Pl\"{u}cker coordinates valuate to $w$. So, the points in the interior of the triangle in $FlDr(\MM)$ correspond to two realizable valuated matroids that fail to form a realizable valuated matroid quotient. 

We see why it is not possible to realize these points as follows. Points on the interior of the triangle correspond to snowflake trees with pairs $\{2,5\},$ $\{1,6\},$ and $\{3,4\}$. In order to realize this over $\kk$, we would need to make a configuration as in Figure \ref{fig:m4u26} such that over the residue field, the projections of the points $\{2,5\}$ coincide, the projections of the points $\{1,6\}$ coincide, and the projections of the points $\{3,4\}$ coincide. The dual picture is shown in Figure \ref{fig:m4dual}. In order to realize the desired snowflake, we need to find a line that intersects the six lines pictured at each of the points of intersection of the lines $2$ and $5$,  $1$ and $6$, and $3$ and $4$. This is only possible over fields of characteristic 2, where the Fano plane is realizable.
\begin{figure}[h]
    \centering
    \includegraphics[height = 2 in]{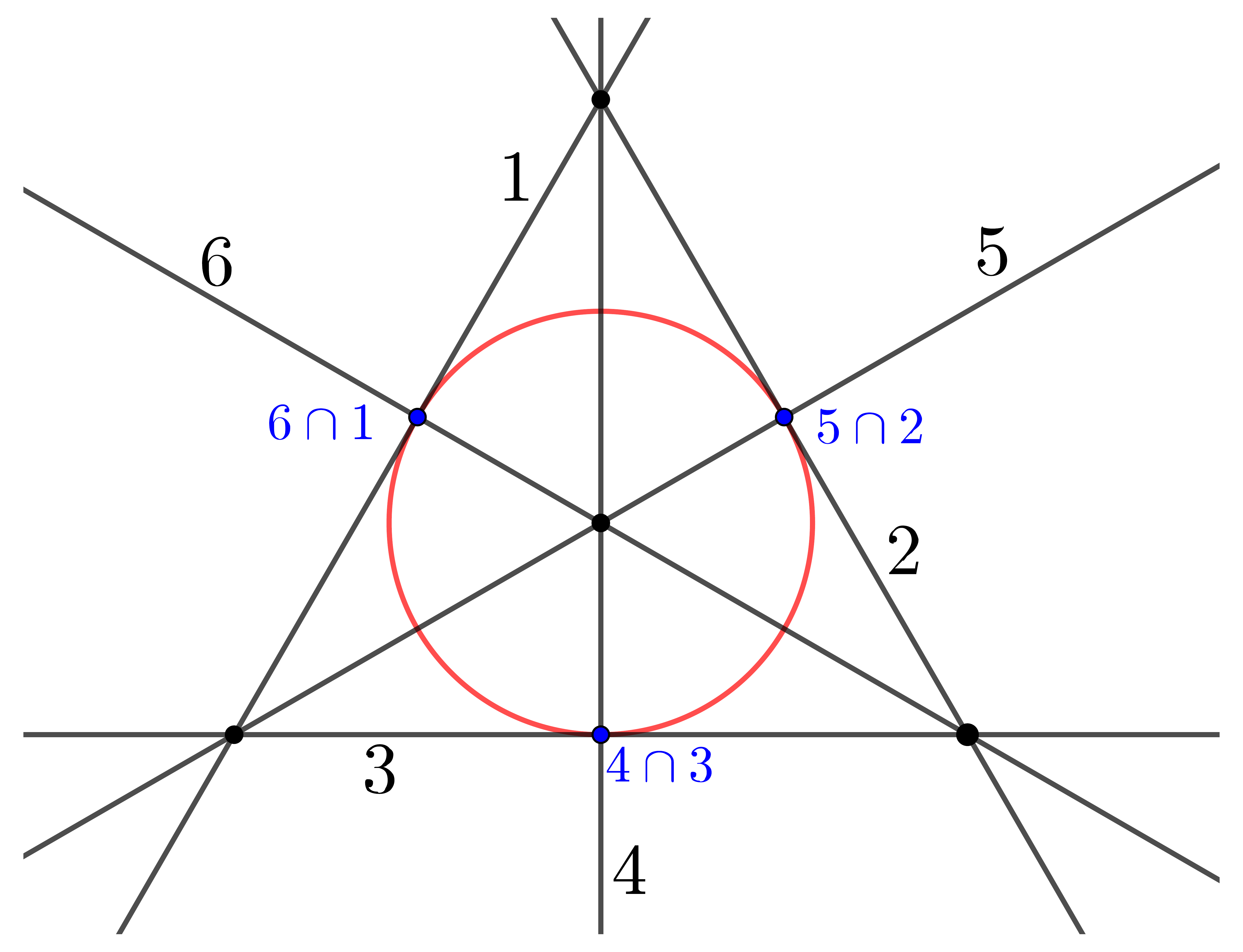}
    \caption{The dual matroid to $M_4$ is pictured in black. The red line must be realizable to produce the trees appearing in the interior of the triangle in Figure \ref{fig:m4u26_dress}.}
    \label{fig:m4dual}
\end{figure}
\end{eg}

\subsection*{Acknowledgements}
We thank Bernd Sturmfels for several useful conversations throughout the development of this article, as well as for his feedback on drafts of this paper. We also thank the Max Planck Institute for Mathematics in the Sciences for its hospitality while working on this project.
We thank Sara Lamboglia for pointing us to Mohammad Haque's work as well as for assistance with \texttt{Macaulay2}. The second author thanks Matt Baker for introducing him to matroids over hyperstructures, and thanks June Huh, Michael Joswig, Gaku Liu, Francisco Santos, and Kristin Shaw for helpful communications.  Leon Zhang was partially supported
by a National Science Foundation Graduate Research Fellowship.

\printbibliography

\end{document}